\documentclass[11pt]{article}

\usepackage{graphicx}
\usepackage{latexsym,amsmath,amsfonts,amscd, amsthm, dsfont}
\usepackage{bm,color}
\usepackage{epsfig,verbatim,epstopdf,graphics}
\usepackage{subfigure}
\usepackage{changebar}
\usepackage{multirow}

\usepackage[ruled,vlined]{algorithm2e}

\usepackage{yhmath}
 \usepackage{booktabs} 
 \usepackage{tikz}
\usepackage{verbatim}
\usetikzlibrary{arrows,backgrounds,snakes,shapes}

 \numberwithin{equation}{section}

\graphicspath{{./}{./figure/}}
\allowdisplaybreaks

\newtheoremstyle{plainNoItalics}{}{}{\normalfont}{}{\bfseries}{.}{ }{}

\theoremstyle{plain}
\newtheorem{thm}{Theorem}[section]

\theoremstyle{plainNoItalics}

\newtheorem{rem}[thm]{Remark}
\newtheorem{prop}[thm]{Proposition}
\newtheorem{exa}[thm]{Example}

\newcommand{\bx}{{\bf x}}

\newcommand{\bv}{{\bf v}}
\newcommand{\bw}{{\bf w}}

\newcommand{\bE}{{\bf E}}
\newcommand{\bJ}{{\bf J}}

\newcommand{\bU}{{\bf U}}

\newcommand{\ijindex}[1]{{#1}_{i,j}^{ig,jg}}

\newcommand{\beq}{\begin{equation}}
\newcommand{\eeq}{\end{equation}}
\newcommand{\bit}{\begin{itemize}}
\newcommand{\eit}{\end{itemize}}
\newcommand{\be}{\begin{eqnarray}}
\newcommand{\ee}{\end{eqnarray}}
\newcommand{\beno}{\begin{eqnarray*}}
\newcommand{\eeno}{\end{eqnarray*}}

\newcommand\QQ[1]{\textcolor{blue}{#1}}

\newcommand{\Rmnum}[1]{\expandafter\@slowromancap\romannumeral #1@}
\makeatother

\usepackage{enumerate}

\begin{document}

\baselineskip=1.8pc


\begin{center}
{\bf
A Local Macroscopic Conservative (LoMaC) low rank tensor method with the discontinuous Galerkin method for the Vlasov dynamics
}
\end{center}

\vspace{.2in}
\centerline{
 Wei Guo,\footnote{
Department of Mathematics and Statistics, Texas Tech University, Lubbock, TX, 70409. E-mail:
weimath.guo@ttu.edu. Research is supported by NSF grant NSF-DMS-1830838 and NSF-DMS-2111383, Air Force Office of Scientific Research FA9550-22-1-0390.
} 
Jannatul Ferdous Ema, \footnote{Department of Mathematics and Statistics, Texas Tech University, Lubbock, TX, 70409. E-mail: Jannatul-Ferdous.Ema@ttu.edu.}
and 
Jing-Mei Qiu\footnote{Department of Mathematical Sciences, University of Delaware, Newark, DE, 19716. E-mail: jingqiu@udel.edu. Research supported by NSF grant NSF-DMS-1818924 and 2111253, Air Force Office of Scientific Research FA9550-22-1-0390.}
}

\bigskip
\noindent
{\bf Abstract.} In this paper, we propose a novel Local Macroscopic Conservative (LoMaC) low rank tensor method with discontinuous Galerkin (DG) discretization for the physical and phase spaces
for simulating the Vlasov-Poisson (VP) system. The LoMaC property refers to the exact local conservation of macroscopic mass, momentum and energy at the discrete level. The recently developed LoMaC low rank tensor algorithm (arXiv:2207.00518) simultaneously evolves the macroscopic conservation laws of mass, momentum and energy using the kinetic flux vector splitting; then the LoMaC property is realized by projecting the low rank kinetic solution onto a subspace that shares the same macroscopic observables. 

This paper is a generalization of our previous work, but with DG discretization to take advantage of its compactness and flexibility in handling boundary conditions and its superior accuracy in the long term. The algorithm is developed in a similar fashion as that for a finite difference scheme, by observing that the DG method can be viewed equivalently in a nodal fashion. With the nodal DG method, assuming a tensorized computational grid, one will be able to (1) derive differentiation matrices for different nodal points based on a DG upwind discretization of transport terms, and (2) define a weighted inner product space based on the nodal DG grid points. The algorithm can be extended to the high dimensional problems by hierarchical Tucker decomposition of solution tensors and a corresponding conservative projection algorithm. In a similar spirit, the algorithm can be extended to DG methods on nodal points of an unstructured mesh, or to other types of discretization, e.g. the spectral method in velocity direction.  Extensive numerical results are performed to showcase the efficacy of the method. 


\section{Introduction}

Numerical simulation of the Vlasov-Poisson (VP) system plays a fundamental role in understanding complex dynamics of plasma and has a wide range of applications in science and engineering, such as fusion energy. The well-known challenges for VP simulations include the high dimensionality of the phase space, resolution of multiple scales in time and in phase space, preservation of physical invariants, among many others. In this paper, we develop a novel Local Macroscopic Conservative (LoMaC) low rank tensor method with discontinuous Galerkin (DG) discretization. The LoMaC property means that the algorithm can conserve locally densities of macroscopic observables at the discrete level.

This paper is a generalization of LoMaC low rank tensor method with finite difference discretization in \cite{guo2022local}. In the introduction of  \cite{guo2022local},  we have discussed the application background and existing works on low rank approach for time-dependent dynamics. Below we only highlight several key ingredients to realize accuracy, robustness, computational efficiency and local conservation for macroscopic observables of the newly proposed algorithm. 
\begin{enumerate}
\item {\em Low rank representation of solutions and high order discretizations \cite{guo2022low}.} In this low rank approach, the solution is being written in the form of Schmidt decomposition, where the basis in each dimension are being dynamically updated from a high order discretization of PDEs together with a singular value type truncation for sparsity in function representation and efficiency for computational complexity. The original idea is presented in  \cite{guo2022low}. In this paper, we generalize the algorithm to nodal DG type spatial discretization on tensor product of computational meshes. The nodal DG differentiation operator, as well as the weights in the discrete inner product space, will depends on the mesh spacing and the associated Gaussian quadrature nodes in each computational cell. The new method allows the flexibility in mesh spacing, e.g. using a not smooth nonuniform mesh, yet achieves high order spatial accuracy. Meanwhile the method take advantages of the compactness of the DG discretization in boundary treatment. With the weighted inner product space, we perform a scaling procedure, followed by a standard SVD truncation, and finished by a rescaling procedure to remove redundancy for data sparsity.  For time discretization, we apply the strong-stability-preserving (SSP) multi-step time discretizations \cite{gottlieb2011strong}. 
\item {\em Simultaneous update of macroscopic mass, moment and energy in a locally conservative manner.} This step is the key novelty in \cite{guo2022local} in locally preserving mass, momentum and even energy in an explicit scheme. In this paper, we use a nodal DG scheme for macroscopic conservation laws, with the numerical fluxes from taking moment integration of kinetic probability density functions via the kinetic flux vector splitting (KFVS) fluxes \cite{mandal1994kinetic, xu1995gas}. Meanwhile, the updated macroscopic mass, momentum and energy are used to correct the kinetic solutions via a macroscopic conservative projection. Figure~\ref{f1} from \cite{guo2022local} shows the interplay between numerical solutions for kinetic model and the corresponding macroscopic system. The kinetic solution $f$ is used as the kinetic flux to advance solutions for macroscopic systems, while the updated macroscopic mass, momentum and energy are used to perform a conservative correction to kinetic solution $f$ via a macroscopic conservative projection. 
\begin{figure}[h!]
	\centering
	        {\includegraphics[height=20mm]{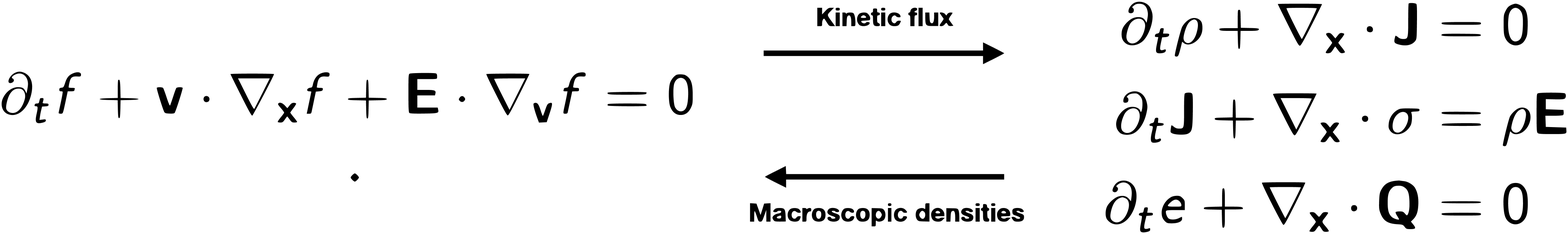}}
	        \caption{Illustration of LoMaC scheme.}
\label{f1}
\end{figure}
The newly developed low rank DG algorithm is theoretically proved and numerically verified to be locally mass, momentum and energy conservative. 
\item {\em Hierarchical Tucker (HT) representation of high dimensional tensors.} We further generalize the algorithm to high-dimensional problems with the HT decomposition, which attains a storage complexity that is linearly scaled with the dimension and polynomial scaled with the rank, mitigating the curse of dimensionality. 
The HT format \cite{hackbusch2009new,grasedyck2010hierarchical} is motivated by the classical Tucker format \cite{tucker1966some,de2000multilinear}, but considering a dimension tree and taking advantage of the hierarchy of the nested subspaces. A hierarchical high order singular value decomposition (HOSVD)\cite{hackbusch2009new,grasedyck2010hierarchical} can be performed to strike a balance between data complexity and numerical feasibility. In this paper, we use the same dimension tree as in our earlier work \cite{guo2022local} for 2D2V Vlasov system, with full rank in the physical spaces, and low rank in velocity spaces, and low rank between physical and velocity spaces. 
\end{enumerate}

As far as we are aware of, this paper is a first paper on coupling the DG discretization with the low rank tensor framework for kinetic simulations. It well combines the merits of DG discretization with that of low rank tensor approach: for the DG method in flexibility and robustness in using nonuniform or unstructured meshes, in treating complex boundary conditions and in realizing superconvergence properties in a long time simulation, and for the low rank tensor approach in reducing computational complexity. Although we haven't extended the algorithm to unstructured triangular meshes or for complex boundary conditions here, this paper serves as a first step in this direction, and shows the proof of concept on the potential of the algorithm for complex and high dimensional problems. 

This paper is organized as follows. In Section 2, we introduce the kinetic Vlasov model and the corresponding macroscopic conservation laws. Section 3 is the main section to introduce the proposed algorithm. We introduce the DG and nodal DG discretization in Section 3.1;  we discuss the low rank framework with tensor product of nodal DG meshes, the weighted inner product spaces, and the corresponding macroscopic conservative projection and weighted SVD truncation in Section 3.2; we propose the LoMaC low rank DG algorithm in Section 3.3 with remarks on further generalization of the algorithm to high dimensional problems with HT format and to unstructured meshes. In Section 4, we present numerical results on an extensive set of 1D1V and 2D2V problems to demonstrate the efficacy the proposed algorithm.  We conclude in Section 5.


\newcommand{\xL}{{x_{i-\frac{1}{2}}}}
\newcommand{\xR}{{x_{i+\frac{1}{2}}}}
\newcommand{\vL}{{v_{j-\frac{1}{2}}}}
\newcommand{\vR}{{v_{j+\frac{1}{2}}}}
\newcommand{\iL}{{i-\frac{1}{2}}}
\newcommand{\iR}{{i+\frac{1}{2}}}
\newcommand{\shalf}{\tiny{\frac12}}

\section{The kinetic Vlasov model and the corresponding macroscopic systems}
In this section, we introduce the Vlasov model and the corresponding macroscopic systems. 
We consider the dimensionless VP system 
\beq
\frac{\partial f}{\partial t}  
+  {\bf{v}} \cdot \nabla_{\bf{x}}  f 
+ {\bf{E}} ({\bf{x}},t) \cdot \nabla_{\bf{v}}  f = 0,
\label{vlasov1}
\eeq
\beq
 {\bf E}( {\bf x},t) = - \nabla_{\bf x} \phi({\bf x},t),  \quad -\triangle_{\bf x} \phi ({\bf x},t) = {{\bf \rho} ({\bf x},t)} - \rho_0,
\label{poisson}
\eeq
which describes the dynamics of  the probability distribution function $f({\bf x}, {\bf v},t)$ of electrons in a collisionless plasma. 
Here ${\bf E}$ is the electric field and $\phi$ is the self-consistent electrostatic potential  determined by  Poisson's equation. $f$ couples to the long range fields via the charge density ${\bf \rho}({\bf x},t) = \int_{\Omega_{\bv}} f({\bf x}, {\bf v},t) d {\bf v}$, where we take the limit of uniformly distributed infinitely massive ions in the background.  

The Vlasov dynamics are well-known to conserve several physical invariants. In particular, let 
\be
\label{eq: mass_d}
\mbox{charge density:}&& \rho (\bx, t) = \int_{\Omega_{\bv}} f(\bx, \bv,t) d \bv, \\
\label{eq: current_d}
\mbox{current density:} &&\bJ (\bx, t) = \int_{\Omega_{\bv}}f(\bx, \bv,t) \bv d \bv,\\
\label{eq: kenergy_d}
\mbox{kinetic energy density:} && \kappa(\bx,t) = \frac{1}{2} \int_{\Omega_{\bv}} |\bv|^{2}  f(\bx, \bv,t) d \bv,\\
\label{eq: energy_d}
\mbox{energy density:} && e(\bx,t)=\kappa(\bx,t)+\frac{1}{2} \bE(\bx)^2.
\ee
Then, by taking the first few moments of the Vlasov equation,
the following conservation laws of mass, momentum and energy can be derived
\begin{align}
\partial_{t} \rho + \nabla_\bx \cdot \bJ &= 0\label{eq:mass}\\
\partial_{t} \bJ +\nabla_{\bx} \cdot \mathbf{\sigma}&= \rho\bE \label{eq:mom}\\
\partial_{t} e +\nabla_{\bx} \cdot \mathbf{Q}& =0,\label{eq:ener} 
\end{align}
where $ \sigma(t, \bx)=\int_{\Omega_{\bv}}(\bv \otimes \bv) f(\bx, \bv,t) d \bv$ and $\mathbf{Q}(\bx,t) =\frac12\int_{\Omega_{\bv}}\bv|\bv|^2 f(\bx, \bv,t) d \bv$. 
It is well-known that local conservation property is essential to capture correct entropy solutions of hyperbolic systems such as \eqref{eq:mass}-\eqref{eq:ener}. 

\section{A LoMaC low rank tensor approach with DG discretizations for the Vlasov dynamics}

For simplicity of illustrating the basic idea, we only discuss a 1D1V example in this section. The low rank tensor approach \cite{guo2021lowrank} is designed based on the assumption that our solution at time $t$ has a low rank representation in the form of 
\begin{equation}
\label{eq: fn1}
f(x, v, t) = \sum_{l=1}^{r} \left(C_l(t) \  U_l^{(1)}(x, t) U_l^{(2)}(v, t)\right),
\end{equation}
where $\left\{U_l^{(1)}(x, t)\right\}_{l=1}^{r}$ and $\left\{U_l^{(2)}(v, t)\right\}_{l=1}^{r}$ are a set of time-dependent low rank orthonormal  basis in $x$ and $v$ directions, respectively, $C_l$ is the coefficient for the basis $U_l^{(1)}(x, t)U_l^{(2)}(v, t)$, and $r$ is the representation rank.  \eqref{eq: fn1} can be viewed as a Schmidt decomposition of functions in $(x, v)$ by truncating small singular values up to rank $r$. 

\subsection{DG discretization with nodal Lagrangian basis functions.}

We perform a DG discretization with a piecewise $Q^k$ polynomial space for $f$ on a truncated 1D1V domain of $\Omega = [x_{\min}, x_{\max}] \times [-v_{\max}, v_{\max}]$. 
We start with a tensor product Cartesian partition of $\Omega$ denoted by $\Omega_h$ with
$$x_{\min}=x_{\frac12}<x_{\frac{3}{2}}<\cdots <x_{N_x+\frac12}=x_{\max},$$ 
$$-v_{\max}=v_{\frac12}<v_{\frac{3}{2}}<\cdots <v_{N_v+\frac12}=v_{\max}.$$
 Denote an element as $I_{ij}=[\xL, \xR]\times[\vL, \vR]\in \Omega_h$  with element size $h_{x, i} h_{v, j}$ and the center $x_{i} = \frac12(x_{i-\frac12}+x_{i+\frac12})$ and $v_{j} = \frac12(v_{j-\frac12}+v_{j+\frac12})$ . Let $h_x=\max_{i=1}^{N_x}h_{x, i}$ and $h_v=\max_{j=1}^{N_v}h_{v, j}$. Given any non-negative integer $k$, we define a finite dimensional discrete space with piecewisely defined $Q^k$ polynomials, 
\begin{equation}
Q_h^k=\left\{p(x, v)\in L^2(\Omega): p|_{I_{ij}}\in Q^k(I_{ij}),\, \forall I_{ij}\in \Omega_h \right\}.
\label{eq:DiscreteSpace}
\end{equation}
The local space $Q^k(I)$ consists of polynomials with terms in the form of $x^m v^n$ with $\max(m, n)\le k$ on $I\in\Omega_h$.
To distinguish the left and right limits of a function $p\in Q_h^k$ at $(x_{i+\frac{1}{2}}, v)$, \QQ{we let
$p_{i+\frac{1}{2}, v}^\pm=\lim_{\delta \rightarrow \pm 0}p(x_{i+\frac{1}{2}}+\delta, v)$.}

A semi-discrete DG method for the Vlasov equation \eqref{vlasov1} is: find $f_h(\cdot, \cdot, t)\in {Q}_h^k$ , such that
$\forall \phi \in Q_h^k$ and $\forall I_{ij}\in\Omega_h$,
\begin{align}
\label{eq:DG}
\int_{I_{ij}} \partial_t f_h \phi dxdv &= \int_{I_{ij}} v f_h \phi_x dxdv - \int_\vL^{\vR} v(\hat{f}_{i+\frac{1}{2}, v} \phi^-_{i+\frac{1}{2}, v}- \hat{f}_{i-\frac{1}{2}, v} \phi^+_{i-\frac{1}{2}, v}) dv \\
& + \int_{I_{ij}} E f_h \phi_v dxdv - \int_\xL^{\xR} E(x) (\hat{f}_{x, j+\frac{1}{2}} \phi^-_{x, j+\frac{1}{2}}- \hat{f}_{x, j-\frac{1}{2}} \phi^+_{x, j-\frac{1}{2}}) dx.\notag
\end{align}

To implement the DG scheme under the low rank framework, we use the nodal basis to represent functions in the discrete space $Q_h^k$,
in conjunction with rewriting and/or approximating the integrals in the schemes by numerical quadratures.
We consider a reference cell $I=[-\frac12, \frac12]\times[-\frac12, \frac12]$ and the tensor product of Gaussian quadrature points in each direction $\{\xi_{ig},\eta_{jg}\}_{ig, jg=0}^{k}$. We further let $\{\omega_l\}^k_{l=0}$ denote the corresponding quadrature weights on the reference element. 
The local nodal Lagrangian basis on the reference cell is $\{L_{ig, jg}(\xi, \eta)\}_{ig, jg=0}^{k}$ in $Q^k(I)$ with 
\begin{equation}
L_{ig, jg} (\xi_{ig'},\eta_{jg'})=\delta_{ig, ig'}\delta_{jg, jg'},\quad ig, ig', jg, jg'=0,\cdots, k.
\end{equation}
Here $\delta_{\cdot, \cdot'}$ is the Kronecker delta function. In fact, 
$
L_{ig, jg} (\xi, \eta) = L_{ig} (\xi) L_{jg} (\eta), 
$
where $L_{ig}$ and $L_{jg}$ are the 1D Lagrangian nodal basis functions associated with the corresponding Gaussian nodes. For a computational cell $I_{ij}$, we can perform a linear transformation to the reference cell, with $\xi = \frac{x-x_i}{h_{x, i}}, \eta = \frac{v-v_j}{h_{v,j}}$, and denote by the shifted Gaussian nodes $x_{i,ig} = x_i + h_{x,i}\xi_{ig}$, $v_{j,jg} = v_j + h_{v,j}\eta_{jg}$.

With the nodal basis functions, the DG scheme \eqref{eq:DG} on a computational cell $I_{ij}$ can be equivalently written with the test functions being taken as $L_{ig', jg'}(\xi, \eta)$, $ig', jg'=0,\cdots,k$. We look for the DG solution expressed in the form of $f_{h, i, j}(x, v, t) = \sum_{ig, jg=0}^{k} f^{ig, jg}_{h, i,j}(t) L_{ig, jg}(\xi(x), \eta(v))$, with its nodal values satisfying the following equations:
 \begin{align}
\label{eq:DG_0}
&h_{x, i}h_{v, j} \omega_{ig}\omega_{jg} \left(\frac{d}{dt}f^{ig, jg}_{h,i,j}(t)\right) \notag\\
=& h_{x, i}h_{v, j} \omega_{jg} v_{j,jg}  \sum_{ig''} \omega_{ig''}\left(\frac{d}{dx} L_{ig}(\xi_{ig''})f^{ig'', jg}_{h,i,j}(t) \right)
 - h_{v, j} \omega_{jg}  v_{j,jg} \left(\hat{f}_{i+\frac12, jg} L_{ig}(\frac12)- \hat{f}_{i-\frac12, jg} L_{ig}(-\frac12)\right) \notag\\
+& h_{x, i}h_{v, j}  \omega_{ig} E_{i,ig}  \sum_{jg''} \omega_{jg''} \left(  \frac{d}{dv} L_{jg}(\eta_{jg''}) f^{ig, jg''}_{h,i,j}(t) \right) - h_{x, i} \omega_{ig}  E_{i,ig} \left(\hat{f}_{ig, j+\frac12} L_{jg}(\frac12)- \hat{f}_{ig, j-\frac12} L_{jg}(-\frac12)\right). 
\end{align}
Dividing by $h_{x, i}h_{v, j} \omega_{ig}\omega_{jg}$, the above equation becomes
 \begin{align}
\label{eq:DG}
\frac{d}{dt}f^{ig, jg}_{h, ij}(t) 
=& \frac{1}{\omega_{ig}}\left(v_{j,jg}  \sum_{ig''} \omega_{ig''}\left(\frac{d}{dx} L_{ig}(\xi_{ig''})f^{ig'', jg}_{h, i,j}(t) \right)
 -   \frac{v_{j,jg}}{h_{x, i}}\left(\hat{f}_{i+\frac12, jg} L_{ig}(\frac12)- \hat{f}_{i-\frac12, jg} L_{ig}(-\frac12)\right)\right) \notag\\
+&  \frac{1}{\omega_{jg}}\left(E_{i,ig}  \sum_{jg''} \omega_{jg''} \left(  \frac{d}{dv} L_{jg}(\eta_{jg''}) f^{ig, jg''}_{h, i,j}(t) \right) - \frac{E_{i,ig}}{h_{v, j}} \left(\hat{f}_{ig, j+\frac12} L_{jg}(\frac12)- \hat{f}_{ig, j-\frac12} L_{jg}(-\frac12)\right)\right),
\end{align}
where $\hat{f}_{i\pm\frac12, jg}$ and $\hat{f}_{ig, j\pm\frac12}$ are taken as monotone upwind fluxes and $E_{i,ig}$ denotes the electric field at $x_{i,ig}$. In particular, let $v^+ = \max(v, 0)$, $v^- = \min(v, 0)$, $E^+ = \max(E, 0)$, $E^- = \min(E, 0)$, \eqref{eq:DG} becomes the following with a simple upwind flux
	\begin{align}
& \partial_t f^{ig, jg}_{h, i,j}(t) \notag\\
=& \frac{v^+_{j,jg}}{\omega_{ig}h_{x,i}} \left(\sum_{ig''} \omega_{ig''}\frac{dL_{ig}}{d\xi}(\xi_{ig'}) f^{ig'', jg}_{h, i,j}
 -  f_{h,i, j}^{ig'', jg} L_{ig''}(\shalf)
  L_{ig}(\shalf)-  f_{h, i-1, j}^{ig'', jg} L_{ig''}(\frac12) L_{ig}(-\frac12)\right)\notag \\
+& \frac{v^-_{j,jg}}{\omega_{ig}h_{x,i}} \left(\sum_{ig''} \omega_{ig''}\frac{d L_{ig}}{d\xi}(\xi_{ig''})f^{ig'', jg}_{h, i,j}
 -  f_{h,i+1, j}^{ig'', jg} L_{ig''}(-\frac12)
  L_{ig}(\frac12) + f_{h, i, j}^{ig'', jg} L_{ig''}(-\frac12) L_{ig}(-\frac12)\right)\notag\\
  +&  \frac{E^+_{i,ig}}{\omega_{jg}h_{v,j}} \left(\sum_{jg''} \omega_{jg''}  \frac{dL_{jg}}{d\eta} (\eta_{jg''}) f^{ig, jg''}_{h, i,j}
   - f_{h,i, j}^{ig, jg''} L_{jg''}(\frac12) L_{jg}(\frac12)+ f_{h,i, j-1}^{ig, jg''} L_{jg''}(\frac12) L_{jg}(-\frac12)\right) \notag\\
     +&  \frac{E^-_{i,ig}}{\omega_{jg}h_{v,j}} \left(\sum_{jg''} \omega_{jg''}  \frac{dL_{jg}}{d\eta} (\eta_{jg''}) f^{ig, jg''}_{h, i,j}
   - f_{h,i, j+1}^{ig, jg''} L_{jg''}(-\frac12) L_{jg}(\frac12)+ f_{h,i, j}^{ig, jg''} L_{jg''}(-\frac12) L_{jg}(-\frac12)\right)
    \label{eq:DG_1}
\end{align}

We denote the first two terms on the RHS of \eqref{eq:DG_1} as 
\beq
\label{eq: x-der-dg}
v^+_{j,jg}\cdot D^{+}_{x, i,ig} {\bf f}^{+,:, jg}_{h, i, j}, \quad v^-_{j,jg} \cdot D^-_{x, i,ig} {\bf f}^{-,:, jg}_{h, i, j},
 \eeq
 as standard 1D upwind DG discretizations of $x$ derivative at the $ig$-th Gaussian node of the $i$-th cell for positive/negative velocity, respectively. Here we assume that the $v$- grid is fixed at $jg$-th Gaussian node of the $j$-th cell, and 
 \begin{align*}
 {\bf f}^{+,:, jg}_{h, i, j}&=(f_{h,i-1,j}^{0,jg},\ldots,f_{h,i-1,j}^{k,jg},f_{h,i,j}^{0,jg},\ldots,f_{h,i,j}^{k,jg}),\\
  {\bf f}^{-,:, jg}_{h, i, j}&=(f_{h,i,j}^{0,jg},\ldots,f_{h,i,j}^{k,jg},f_{h,i+1,j}^{0,jg},\ldots,f_{h,i+1,j}^{k,jg}).
 \end{align*}
 Similarly, the other two terms  are denoted as 
 \beq
\label{eq: v-der-dg}
E^+_{i,ig}\cdot D^{+}_{v, j,jg} {\bf f}^{+,ig, :}_{h, i, j}, \quad E^-_{i,ig}\cdot D^-_{v, i,ig} {\bf f}^{-,ig,:}_{h, i,j},
 \eeq
 where 
\begin{align*}
{\bf f}^{+,ig, :}_{h, i, j}&=(f_{h,i,j-1}^{ig,0},\ldots,f_{h,i,j-1}^{ig,k},f_{h,i,j}^{ig,0},\ldots,f_{h,i,j}^{ig,k}),\\
{\bf f}^{-,ig, :}_{h, i, j}&=(f_{h,i,j}^{ig,0},\ldots,f_{h,i,j}^{ig,k},f_{h,i,j+1}^{ig,0},\ldots,f_{h,i,j+1}^{ig,k}).
\end{align*}
 
\begin{rem}
One observation in the above formulation is that, although DG method formulate the scheme in an element-by-element fashion, the evaluation of solution derivatives in $x$- and $v$-directions, at Gaussian nodal points of each cell, actually occurs in a dimension-by-dimension manner. In other words, we can formulate a DG differentiation operator $D^\pm_{x}$ by concatenating $D^\pm_{x,i,ig}$.
%
%
 Similar comments can be applied to $D^\pm_{v}$ as the DG differentiation operator for the $v$-derivative.
\end{rem}

\subsection{Nodal DG solutions on grid points and weighted SVD}

In this subsection, we first set up the nodal DG solutions at Gaussian grid points on each computational cell, which comes from a tensor product of $x$ and $v$ discretizations. Then we introduce several basic tools for performing the LoMaC DG low rank tensor approach in the next subsection. These tools include the weights and definition for the discrete inner product space, the orthogonal projection for conservation of macroscopic observables in the weighted inner product space,  as well as the weighted singular value truncation.  

The nodal grid points for the DG discretization, as tensor product of $(k+1)N_x \times (k+1)N_v$ points from $N_x \times N_v$ computational cells, are
\beq
\label{eq: x_grid}
x_{\text{grid}}: \quad x_{\min}< \cdots < (x_{i, 0} < \cdots<x_{i, k}) \cdots <  x_{\max}, 
\eeq
\beq
\label{eq: v_grid}
v_{\text{grid}}: \quad -v_{\max}< \cdots < (v_{j, 0} < \cdots<v_{j, k}) \cdots <  v_{\max}.
\eeq
Here $\{x_{i, ig}\}_{ig=0}^{k}$ and $\{v_{j, jg}\}_{jg=0}^{k}$ are the shifted Gaussian points on the cell $[\xL, \xR]$ and $[\vL, \vR]$ respectively. 
DG nodal solutions on the tensor product of grids \eqref{eq: x_grid}-\eqref{eq: x_grid} are organized as ${\bf f} \in \mathbb{R}^{(k+1)N_x \times (k+1) N_v}$ with each of its component $f^{ig, jg}_{h, i,j}(t)$ being an approximation to point values of the solution on the tensor product of grids \eqref{eq: x_grid}-\eqref{eq: v_grid}. It has a corresponding low rank decomposition, similar to \eqref{eq: fn1}, as
\begin{equation}
\label{eq: fn2}
{\bf f} = \sum_{l=1}^{r} \left(C_l \  {\bf U}_l^{(1)}  \otimes {\bf U}_l^{(2)}\right), \quad 
(\mbox{or element-wise:} \quad
\ijindex{f} = \sum_{l=1}^{r}  C_l \ {U}_{l, i, ig}^{(1)} {U}_{l, j, jg}^{(2)}), 
\end{equation}
where ${\bf U}_l^{(1)} \in \mathbb{R}^{(k+1)N_x}$ and ${\bf U}_l^{(2)} \in \mathbb{R}^{(k+1)N_v}$ can be viewed as approximations to corresponding grid point values of the basis functions in \eqref{eq: fn1}. \eqref{eq: fn2} can also be viewed as a weighted SVD of the matrix ${\bf f} \in \mathbb{R}^{(k+1)N_x \times (k+1)N_v}$, where the weight 
\beq
\bm{ \omega} = \bm{ \omega}_x \otimes \bm{ \omega}_v
\label{eq: weight}
\eeq
with  
\[
\bm{ \omega}_{x}\in \mathbb{R}^{(k+1)N_x}, \quad { \omega}_{x, i, ig} = h_{x, i} \omega_{ig}, \quad i = 1, \cdots N_x, \quad ig = 0, \cdots,k,
\]
\[ 
\bm{ \omega}_{v}\in \mathbb{R}^{(k+1)N_v}, \quad { \omega}_{v, j, jg} = h_{v, j} \omega_{jg}, \quad j = 1, \cdots N_v, \quad jg = 0, \cdots, k.
\]
Next, we introduce three basic operations for the discrete weighted inner product spaces: (1) the computation of macroscopic observations; (2) the orthogonal projection of ${\bf f}$ for conservation of macroscopic observables; (3) a weighted singular value truncation.  

\bit
\item {\bf Macroscopic quantities of ${\bf f}$.}
In order to perform the projection, we first compute macroscopic quantities of ${\bf f}$, i.e. the discrete macroscopic charge, current and kinetic energy density ${\boldsymbol \rho}$, ${\bf J}$ and ${\boldsymbol \kappa} \in \mathbb{R}^{N_x}$ by quadrature 
   \begin{align}
\left(\begin{array}{l}
{\boldsymbol \rho}\\
{\bf J}\\
{\boldsymbol \kappa} 
\end{array}
\right )
 = \sum_{l=1}^{r} C_l
 \left 
 \langle \bU^{(2)}_{l}, 
 \left(\begin{array}{l}
 {\bf 1}_v \\
\bv\\
\frac12\bv^2
\end{array}
\right )
\right \rangle_v
\ \bU^{(1)}_l. 
\label{eq:rho_j_kappa}
\end{align}   
and the inner product $\langle \cdot, \cdot \rangle_v$ is defined as
\beq
\label{eq: inner}
\langle {\bf f}, {\bf g} \rangle_v \doteq \sum_{j, jg} f_{j, jg} g_{j, jg} \omega_{v, j,jg}, \quad {\bf f}, {\bf g} \in \mathbb{R}^{(k+1)N_v},
\eeq
in analogue to the continuous inner product $\int_{\Omega_v} f(v)g(v)dv$.
\item {\bf An orthogonal projection with preservation of macroscopic densities.}
Following the conservative projection idea in \cite{guo2022lowrank}, we propose to project a kinetic solution ${\bf f}$ to a subspace 
\beq
\mathcal{N}\doteq \text{span}\{{\bf 1}_v, \bv, \bv^2\},
\eeq
where  ${\bf 1}_v\in \mathbb{R}^{(k+1)N_v}$ is the vector of all ones,  $\bv$ is the v-grid \eqref{eq: v_grid} and $\bv^2$ $\in \mathbb{R}^{(k+1)N_v}$ is the element-wise square of $\bv$. We use a weight function $w_M(v)= \exp(-v^2/2)$ with exponential decay to ensure proper decay of the projected function as $v \to \infty$. 
We introduce the weighted inner product and the associated norm as
\beq
\label{eq: inner_prod_d}
\langle {\bf f},  {\bf g} \rangle_{{\bf w}_M} = \sum_{j, jg} f_{j, jg} g_{j, jg} w_{M, j, jg} \omega_{v, j,jg},  \quad \|{\bf f}\|_{{\bf w}_M} =\sqrt{\langle {\bf f},  {\bf f} \rangle_{{\bf w}_M}}, 
\eeq
where ${\bf w}_M \in \mathbb{R}^{(k+1)N_v}$ with $w_{M, j, jg} = w_M(v_{j, jg})$ and $\omega_{v,j, jg}$ is the quadrature weights for $v$-integration.
Correspondingly, we let
$
l^2_{{\bf w}_M} = \{{\bf f}\in\mathbb{R}^{(k+1)N_v}: \|{\bf f}\|_{{\bf w}_M} < \infty\}.
$
With the weight function, we first scale ${\bf f}$ as 
\beq
\label{eq: rescale}
\tilde{\bf f} = \frac{1}{{\bf w}_M} \star {\bf f} =  \sum_{l=1}^{r} \left(C_l \ \ {\bf U}_l^{(1)}  \otimes \left(\frac{1}{{\bf w}_M} \star  {\bf U}_l^{(2)}\right)\right), 
\eeq
where $\star$ is the element-wise product in the $v$-dimension. 
We perform an orthogonal projection of $\tilde{\bf f}$ with respect to the inner product \eqref{eq: inner_prod_d} onto subspace $\mathcal{N}$, i.e.
\beq
\label{eq: proj}
\langle P_{\mathcal{N}}(\tilde{\bf f}), {\bf g} \rangle_{\bw_M}
= \langle \tilde{\bf f}, {\bf g} \rangle_{\bw_M}, 
\quad \forall {\bf g}\in \mathcal{N}. 
\eeq
It can be shown that ${\bf w}_M\star P_{\mathcal{N}}(\tilde{\bf f})$ preserves the mass, momentum and kinetic energy densities of ${\bf f}$ in the discrete sense. 
With the orthogonal project, a conservative decomposition of ${\bf f}$ \cite{guo2022lowrank} can be performed as 
\beq
\label{eq: f_decom_d}
{\bf f} = {\bw_M} \star (P_{\mathcal{N}}(\tilde{{\bf f}}) + (I-P_{\mathcal{N}})(\tilde{{\bf f}})) 
\doteq {\bw_M} \star (\tilde{{\bf f}}_1 + \tilde{{\bf f}}_2)
\doteq {\bf f}_1 + {\bf f}_2, 
\eeq
where $\mathbf{f}_1$ can be represented as  a rank three tensor
 \begin{align}\label{eq:f1}
 {\bf f}_1 ({\boldsymbol \rho}, {\bf J}, {\boldsymbol \kappa})= & \frac{\boldsymbol \rho}{\|{\bf 1}_v\|_{\bw_M}^2} \otimes ({\bw_M} \star {\bf 1}_v) 
 +  \frac{\bf J}{\|{\bf v}\|_{\bw_M}^2} \otimes ({\bw_M} \star {\bf v})  +  \frac{2 {\boldsymbol \kappa}-c{\boldsymbol \rho}}{\|{\bf v}^2- c {\bf 1}_v\|_{\bw_M}^2} \otimes ({\bw_M} \star ({\bf v}^2- c {\bf 1}_v) ), 
 \end{align} 
 where $c=\frac{\langle \bm{1}_v, {\bf v}^2\rangle_{\bw_M}}{\|{\bf 1}_v\|_{\bw_M}^2}$ is computed so that $\{ {\bf 1}_v, {\bf v}, {\bf v}^2- c {\bf 1}_v\}$ forms an orthogonal set of basis and 
   ${\boldsymbol \rho}$, ${\bf J}$ and ${\boldsymbol \kappa}$ are the discrete mass, momentum and kinetic energy density of ${\bf f}$ from \eqref{eq:rho_j_kappa}. ${\bf f}_1$ preserves the discrete mass, momentum and kinetic energy density of ${\bf f}$, while the remainder part ${\bf f}_2 = {\bf f} -{\bf f}_1$ has zero of them.

\item {\bf Weighted SVD procedure with preservation of macroscopic observables.} The remainder part in the orthogonal decomposition ${\bf f}_2$ can be shown to have zero macroscopic mass, momentum and kinetic energy. In order to perform a singular value truncation to remove redundancy in basis representation, as well as maintain the zero macroscopic observables, we perform a weighted SVD truncation, where the weights comes from the quadrature weights associated with quadrature nodes as well as the weight function $w_M$ at quadrature nodes. 
A weighted SVD procedure assumes a weighted inner product space $\langle \cdot, \cdot \rangle$ in the following sense:
\beq
\label{eq: inner_xv}
\langle {\bf f}, {\bf g} \rangle \doteq \sum_{i, ig; j, jg} \ijindex{f} \ijindex{g}\omega_{x, i, ig} \omega_{v, j, jg} w_{M, j, jg},  \quad {\bf f}, {\bf g} \in \mathbb{R}^{(k+1)N_x \times (k+1)N_v}.
\eeq
The weighted SVD procedure consists of three steps: first a scaling step with element-wise multiplication by $\frac{1}{\bf \sqrt{\bm{ \omega} \star {\bf w}_M}}$ with $\bm{ \omega}$ in \eqref{eq: weight} and ${\bf w}_M$ as in \eqref{eq: inner_prod_d}, followed by a traditional SVD procedure, and finally a rescaling step with element-wise multiplication by ${\sqrt{\bm{ \omega}\star {\bf w}_M}}$. The associated storage cost is $\mathcal{O}(r N)$, where  $N:= \max\{(k+1)N_x, (k+1)N_v\}$. The scaling and rescaling can be performed with respect to the basis in $x$ and $v$ directions with the cost of $\mathcal{O}(r N)$. We denote this weighted SVD truncation procedure as $\mathcal{T}_{\varepsilon, \bm{ \omega} \star {\bf w}_M}$. In the algorithm, it will be applied to the remainder ${\bf f}_2$ in \eqref{eq: f_decom_d}, i.e. $\mathcal{T}_{\varepsilon, \bm{ \omega} \star {\bf w}_M}({\bf f}_2)$ to realize data sparsity. In summary, we have the following weighted SVD truncation procedure for ${\bf f}_2 \in \mathbb{R}^{(k+1)N_x \times (k+1)N_v}$.
\beq
\label{eq: weighted_SVD}
\boxed{{\bf f}_2}\stackrel{scaling}{ \Longrightarrow} \boxed{\tilde{\bf f}_2 \doteq \frac{{\bf f}_2}{\sqrt{\bm{ \omega}\star {\bf w}_M}}} \stackrel{truncation}{\Longrightarrow} \boxed{\mathcal{T}_{\varepsilon}(\tilde{\bf f}_2)} \stackrel{rescaling}{ \Longrightarrow} \boxed{\sqrt{\bm{ \omega}\star {\bf w}_M}\star \mathcal{T}_{\varepsilon}(\tilde{\bf f}_2)}
\eeq
with the output being 
\beq
\label{eq: weighted_T} 
\mathcal{T}_{\varepsilon, \bm{ \omega} \star {\bf w}_M}({\bf f}_2)\doteq \sqrt{\bm{ \omega}\star {\bf w}_M}\star \mathcal{T}_{\varepsilon}(\tilde{\bf f}_2).
\eeq
\eit 

\begin{rem}
We now summarize by recognizing that there are three different discrete inner product spaces we introduced in this subsection: the first is defined by \eqref{eq: inner} as a discrete analog of a standard $L^2$ inner product in $v$ direction only for computing macroscopic observables, the second is defined by \eqref{eq: inner_prod_d} as as a discrete analog of a weighted inner product product space in $v$ direction for projection purpose, and the third is defined by \eqref{eq: inner_xv} as a discrete analog of weighted inner product in $x-v$ directions for weighted SVD truncation for the remainder ${\bf f}_2$ in \eqref{eq: f_decom_d} to realize data sparsity via removing redundancy in basis representation in each dimension.
\end{rem}

\subsection{LoMaC low rank approach with DG discretization}
In this subsection, we introduce the proposed LoMaC low rank approach with DG discretization. The flow chart of the algorithm is in a similar spirit to that we introduced in \cite{guo2022local}. We outline the scheme flow chart with special discussion on the nodal discretization DG spatial discretization and the corresponding weighted orthogonal decomposition and weighted SVD truncation.

Below, we assume the solution in the form of \eqref{eq: fn2} with superscript $n$ for the solution at $t^n$.

\begin{enumerate}
\item [Step 0.] {\em Initialization.} We assume that the analytic initial condition can be written as or approximated by a linear combination of separable functions, then the DG solutions can be constructed directly from those separable functions on Gaussian nodal points. 
\item [Step 1.] {\em Add basis and obtain an intermediate solution ${\bf f}^{n+1, *}$.} 
A second order multi-step discretization of time derivative in \eqref{vlasov1} gives
\begin{equation}
\label{eq: fn3}
{f}^{n+1, *} = \frac14{f}^{n-2}+\frac34 {f}^{n}- \frac32\Delta t \left(v \partial_x ({f}^n) + E^n \partial_v ({f}^n)\right).
\end{equation}
Here the electric field $E^n$ is solved by a Poisson solver. Thanks to the tensor friendly form of the Vlasov equation, assuming the low rank format of solutions at $t^{n-2}$ and $t^n$, ${\bf f}^{n+1, *}$ can be represented in the following low rank format:
\begin{align}
\label{eq:lowrankmethod}
{\bf f}^{n+1, *} =& \frac14 \sum_{l=1}^{r^{n-2}} C_l^{n-2} \left( {\bf U}_l^{(1), n-2} \otimes {\bf U}_l^{(2), n-2}\right) 
+\frac34 \sum_{l=1}^{r^n} C_l^n \left( {\bf U}_l^{(1), n} \otimes {\bf U}_l^{(2), n}\right)  \\
& -\frac32  \Delta t 
\left( D_x {\bf U}_l^{(1), n} \otimes \bv \star {\bf U}_l^{(2), n} + \bE^n \star {\bf U}_l^{(1), n} \otimes D_v {\bf U}_l^{(2), n}
\right),
\end{align}
Here, with a slight abuse of notation, $\bv \in\mathbb{R}^{N_v}$ denotes the coordinates of $v_{grid}$ introduced in \eqref{eq: v_grid}. $D_x$ and $D_v$ represent high order spatial differentiations, and $\star$ denotes an element-wise multiplication operation. 
For example the discretization of $D_x {\bf U}_l^{(1), n} \otimes \bv \star {\bf U}_l^{(2), n}$ follows 
\begin{equation}
D^+_x {\bf U}_l^{(1), n} \otimes \bv^+ \star {\bf U}_l^{(2), n} + D^-_x {\bf U}_l^{(1), n} \otimes \bv^- \star {\bf U}_l^{(2), n}, 
\end{equation}
where $D^+_x$ and $D^-_x$ are a $(k+1)^{th}$ order conservative upwind DG discretization of positive and negative velocities respectively, with $\bv^+ = \max(\bv, 0)$ and $\bv^-=\min(\bv, 0)$. For example, see \eqref{eq: x-der-dg} for the derivative at the $i, ig$-th nodal points. Similar comments can be applied to the $D_v$ operator in $\bE^n \star {\bf U}_l^{(1), n} \otimes D_v {\bf U}_l^{(2), n}$. 
\item [Step 2.] {\em Perform a macroscopic conservative decomposition} as in \eqref{eq: f_decom_d} 
\beq
\label{eq: decomp}
{\bf f}^{n+1, *}  = {\bf f}_1 + {\bf f}_2.
\eeq 
Here ${\bf f}_1$ is computed from \eqref{eq:f1} with the macroscopic observables computed as in \eqref{eq:rho_j_kappa}; 
${\bf f}_2 = {\bf f}-{\bf f}_1$ is the remainder term, where we apply a weight SVD truncation of the remainder term $T_{\epsilon, {\bf w} \star {\bf w}_M}({\bf f}_2)$ as in the previous subsection.
\item [Step 3.]{\em Conservative update of macroscopic variables.} 
Let $U \doteq (\rho, {J}, {e})^\top$,  $F \doteq ({J}, \sigma,  {\bf Q})^\top$ and  $S = (0, \rho E,0)^\top$, then the macroscopic system \eqref{eq:mass}-\eqref{eq:ener} becomes
\beq
\label{eq:U}
U_t + F_x = S.
\eeq
The numerical solutions for $U$ are denoted as $\boldsymbol\rho^{M}$, $\bJ^{M}$, $\boldsymbol\kappa^{M}$, where $M$ is for ``Macroscopic variables". 
In the DG setting, they are nodal values of DG solutions of size $(k+1)N_x$, and are computed with a high order nodal DG spatial discretization coupled with the second order SSP multi-step time integrator for system \eqref{eq:U}: 
\begin{align}
\label{eq:Uupdate_0}
U_{i, ig}^{n+1} &  = \frac14U^{n-2}_{i, ig} + \frac34U^{n}_{i, ig} + \frac32\Delta t (D^{+}_{x, i, ig} {\bf F}^{n, +}_{i,:} + D^{-}_{x, i, ig} {\bf F}^{n, -}_{i,:} + {S}^n_{i, ig})
\end{align}
where $U^{n}_{i, ig} = (\rho_{i, ig}^n, J^{n}_{i, ig} , e^{n}_{i, ig})^\top$ and $S_{i, ig}^n = (0, \rho_{i, ig}^nE_{i, ig}^n,0)^\top$, $i=1,\ldots,N_x$, $ig = 0, \cdots, k$. 
The ${\bf F}^{n, \pm} \in \mathbb{R}^{(k+1)N_x}$ are given by the kinetic flux vector splitting scheme \cite{guo2022local} with
 \begin{align}
{\bf F}^{n, +}
& = \sum_{l=1}^{r^n} C^{n}_l
 \left 
 \langle \bU^{(2), n}_{l}, 
 \left(\begin{array}{c}
\bv^+\\
(\bv^+)^2\\
\frac12 (\bv^+)^3
\end{array}
\right )
\right \rangle_v
\ \bU^{(1), n}_l \\
{\bf F}^{n, -}
& = \sum_{l=1}^{r^n} C^n_l
 \left 
 \langle \bU^{(2), n}_{l}, 
 \left(\begin{array}{c}
\bv^-\\
(\bv^-)^2\\
\frac12 (\bv^-)^3
\end{array}
\right )
\right \rangle_v
\ \bU^{(1), n}_l,
\label{eq:Fpm}
\end{align}
where $\bv^+ = \max(\bv, 0)$, $\bv^- = \min(\bv, 0)$ and the weighted inner product is defined in \eqref{eq: inner}. $D^{\pm}_{x, i, ig}$ are defined in a similar fashion as in \eqref{eq: x-der-dg}, and 
$${\bf F}^{n, +}_{i,:} =(F_{i-1,0}^{n,+},\ldots,F_{i-1,k}^{n,+},F_{i,0}^{n,+},\ldots,F_{i,k}^{n,+}),$$
$${\bf F}^{n, -}_{i,:} =(F_{i,0}^{n,-},\ldots,F_{i,k}^{n,-},F_{i+1,0}^{n,-},\ldots,F_{i+1,k}^{n,-}).$$
 From the updated $U^{n+1}_{i, ig}$, we can compute 
\begin{equation}
\label{eq:kinetic_update}
{\kappa}_{i, ig}^{n+1, M} = {e}_{i, ig}^{n+1, M} - \frac12|E^{n+1, M}_{i, ig}|^2,
\end{equation}
where $\bE^{n+1, M}$ is computed directly from $\boldsymbol \rho^{n+1,M}$ via Poisson's equation using the local DG method \cite{arnold2002unified}.  Finally, we construct ${\bf f}^M_1$ according to \eqref{eq:f1}, with the macroscopic observables from this step of macroscopic update. 
\item [Step 4.] We update the low rank solution as
\beq
\label{eq: f_update}
{\bf f}^{n+1}  = {\bf f}^M_1 + \mathcal{T}_{\varepsilon, {\bm{\omega}} \star {\bf w}_M}({\bf f}_2),
\eeq 
where ${\bf f}^M_1$ computed from Step 3 and the weighted SVD truncation operator $T_{\varepsilon, {\bm{\omega}} \star {\bf w}_M}$ as in defined \eqref{eq: weighted_T}. Here $ {\bf f}^M_1$ is used, as a correction to $ {\bf f}_1$ for local conservation of mass, momentum and energy densities. 
\end{enumerate}

In summary, the proposed LoMaC low rank DG scheme updates the VP solution by first adding basis through traditional high order nodal DG discretizations for spatial/velocity derivatives and an SSP multi-step time integrator. Then we perform an orthogonal decomposition, with respect to a weighted inner product space, for preservation of macroscopic observables. Meanwhile, we update macroscopic conservation laws using KFVS fluxes for local conservation of macroscopic mass, momentum and energy density. Finally, we correct the solution via \eqref{eq: f_update} with macroscopic densities agree with those from macroscopic updates and with a weighted SVD truncation on the remainder term to realize optimal data sparsity. Note that the Step 2 and Step 3 above can be implemented in parallel, i.e. no need to be in a sequential order. We have the following proposition for local and global macroscopic conservation properties of the proposed scheme.  

\begin{prop} (Local mass, momentum and energy conservation.) The proposed LoMaC low rank DG algorithm locally conserves the macroscopic mass, momentum and energy. 
\end{prop}
\begin{proof} The proof follows directly from the construction of the algorithm, and the fact that the DG algorithm for macroscopic systems locally conserve the mass, momentum and energy. 
\end{proof}

Finally, we comment on the algorithm extension of the above proposed LoMaC low rank DG algorithm to a general setting. In a high dimensional setting (e.g. 2D2V), the above DG algorithm can be generalized using the hierarchical Tucker (HT) format as in \cite{guo2022local}. If the mesh for spatial discretization comes from tensor product of 1D discretization, then the algorithm can be directly generalized following the steps in \cite{guo2022local}, but with DG discretization on spatial/velocity derivatives and using a weighted inner product space on DG nodal solutions. We will not repeat the details, but refer to \cite{guo2022local}. Alternatively, one could consider nodal DG solutions on an unstructured mesh for the spatial dimensions for flexibility in geometry and boundary conditions, and use a HT dimension tree with full rank in spatial dimensions, but low rank between spatial and velocity dimensions, and within velocity dimensions. Further, it is possible to use DG for spatial discretization for compact boundary treatment and use spectral methods for high order accuracy in velocity directions. Similar LoMaC property can be achieved for the corresponding high dimensional algorithm.

\section{Numerical results}\label{sec:numerical}
In this section, we present a collection of numerical examples to demonstrate the efficacy of the proposed LoMaC low rank tensor DG methods.  
The second order SSP multi-step method is employed for time integration. We numerically verify the LoMaC property by tracking the time evolution of total mass, total momentum and  total energy.

\subsection{Linear advection: convergence and superconvergence}
\begin{exa}\label{ex:linear} We solve the following simple 2D linear advection problem
$$
u_t + u_{x_1} +  u_{x_2} = 0,\quad x_1,\, x_2 \in [0,2\pi],
$$
with periodic boundary conditions. We choose the initial condition $u(x_1,x_2,t=0) = \sin(x_1+x_2)$, and the exact solution is known as  $$u(x_1,x_2,t) = \sin(x_1+x_2-2t),$$
which is smooth and stays very low rank over time. We make use of this example to investigate the convergence and superconvergence of the proposed low rank DG method. It is well known that the full grid DG solution is superconvergent in the negative-order norm with order $2k+1$, based on which the DG solution over a translation invariant grid can be post-processed so that the convergence order is enhanced from $k+1$ to $2k+1$ in the $L^2$ norm \cite{cockburn2003enhanced}. In the simulation, we let $k=1$ and employ a set of uniform meshes with $N_{x_1}=N_{x_2}$. The time step is chosen as $\Delta t= \left(\frac{h_x}{3} \right)^{1.5}$ to minimize the effect of temporal errors. The truncation threshold is set to be $\varepsilon=10^{-4}$. We solve the problem up to $t=1$. At the end of the computation, we post-processes the low rank DG solutions by convolving the basis $\bU^{(1)}$ and $\bU^{(2)}$ with the kernel given in \cite{cockburn2003enhanced}. The numerical results are summarized in Table \ref{tb:linear}. It is observed that the low rank solution before post-processing is second order accurate ($k+1$); after post-processing the low rank DG solution, the accuracy is enhanced to third order ($2k+1$). The CPU time  approximately scales as $2^{1.5}$ with mesh refinement, indicating that the curse of dimensionality is avoided for this problem. In Figure  \ref{fig:linear_error}, we plot the errors before and after post-processing, and it is observed that the errors of the low rank DG solutions are highly oscillatory before post-processing, implying that the solution is superconvergent in the negative-order norm. After post-processing, the error plots become much smoother, and the magnitude is reduced significantly. Lastly, we plot the time histories of the ranks of the DG solution in Figure \ref{fig:linear_rank}, and we can see that the representation ranks of the solutions stay two during the time evolution for all sets of meshes used. The numerical evidence indicates that the proposed low rank DG method with the adding and removing basis procedure preserves the superconvergence property of the standard DG method. The superconvergence phenomenon due to the DG discretization is preserved well under the low rank truncation setting, if the solution stays low rank.

\begin{table}[!hbp]
	\centering
	\caption{Example \ref{ex:linear}. $t=1$. $k=1$. Convergence study. }
	\label{tb:linear}
	\begin{tabular}{|c|c|c|c|c|c|c|c|c|c|}
		\hline
	 \multirow{2}{*}{$N_{x_1}\times N_{x_2}$}  & \multicolumn{4}{|c|} {Before post-processing} & \multicolumn{4}{|c|} {After post-processing} & \multirow{2}{*}{CPU} \\\cline{2-9}
		  & $L^2$ error & order & $L^\infty$ error & order & $L^2$ error & order & $L^\infty$ error & order &\\\hline
$16\times16$	&	1.59E-01	&		&	5.55E-02	&		&	3.24E-02	&		&	7.33E-03	&	&0.28s	\\\hline
$32\times32$	&	3.73E-02	&	2.09	&	1.34E-02	&	2.05	&	4.20E-03	&	2.95	&	9.48E-04	&	2.95& 0.45s	\\\hline
$64\times64$	&	9.03E-03	&	2.05	&	3.29E-03	&	2.03	&	5.35E-04	&	2.97	&	1.21E-04	&	2.97& 1.22s	\\\hline
$128\times128$	&	2.22E-03	&	2.02	&	8.13E-04	&	2.02	&	6.75E-05	&	2.99	&	1.52E-05	&	2.99&3.06s	\\\hline
	\end{tabular}
\end{table}

 \begin{figure}[h!]
	\centering
       \subfigure[$N_{x_1}\times N_{x_2}=16\times16$]{\includegraphics[height=40mm]{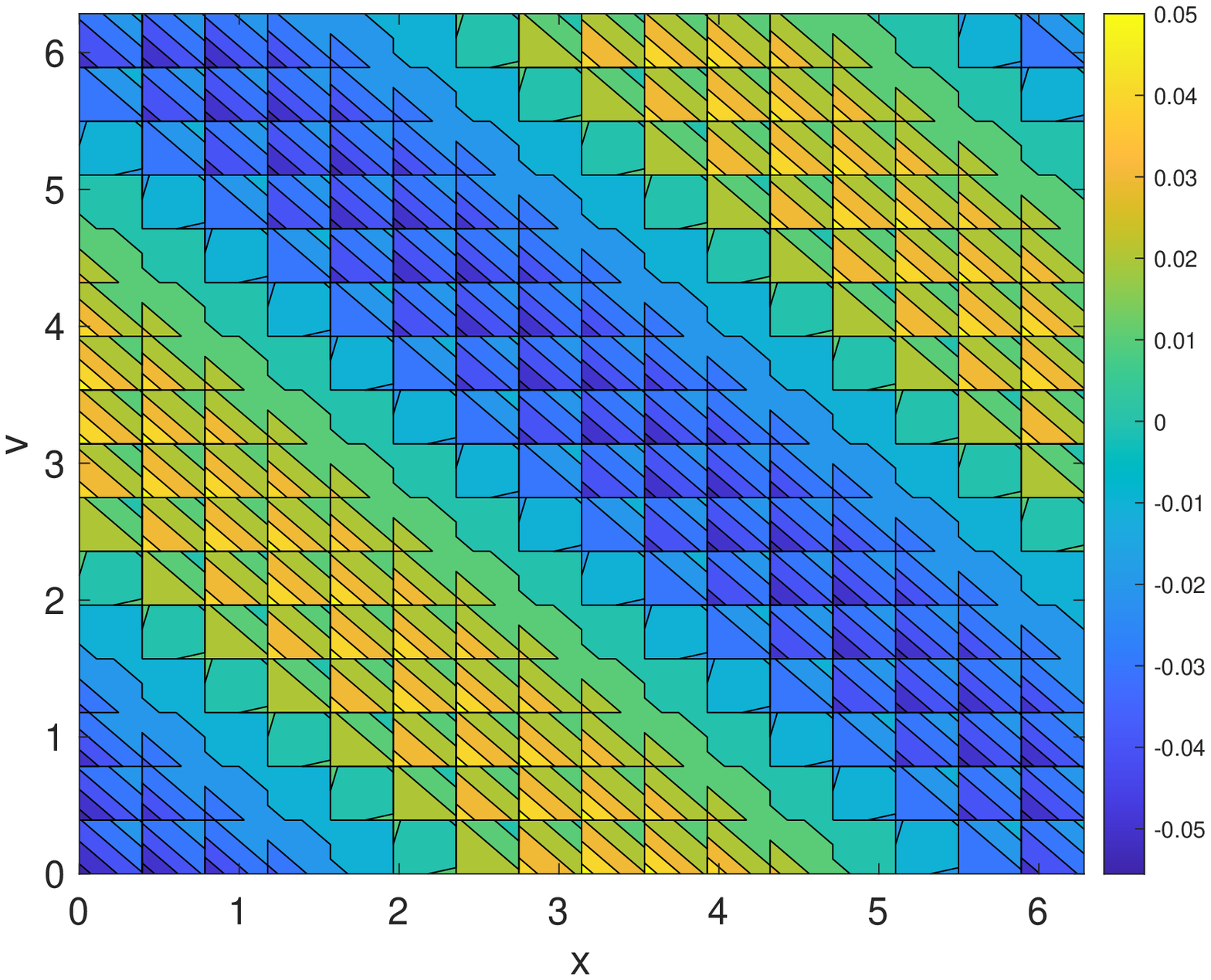}}
	\subfigure[$N_{x_1}\times N_{x_2}=32\times32$]{\includegraphics[height=40mm]{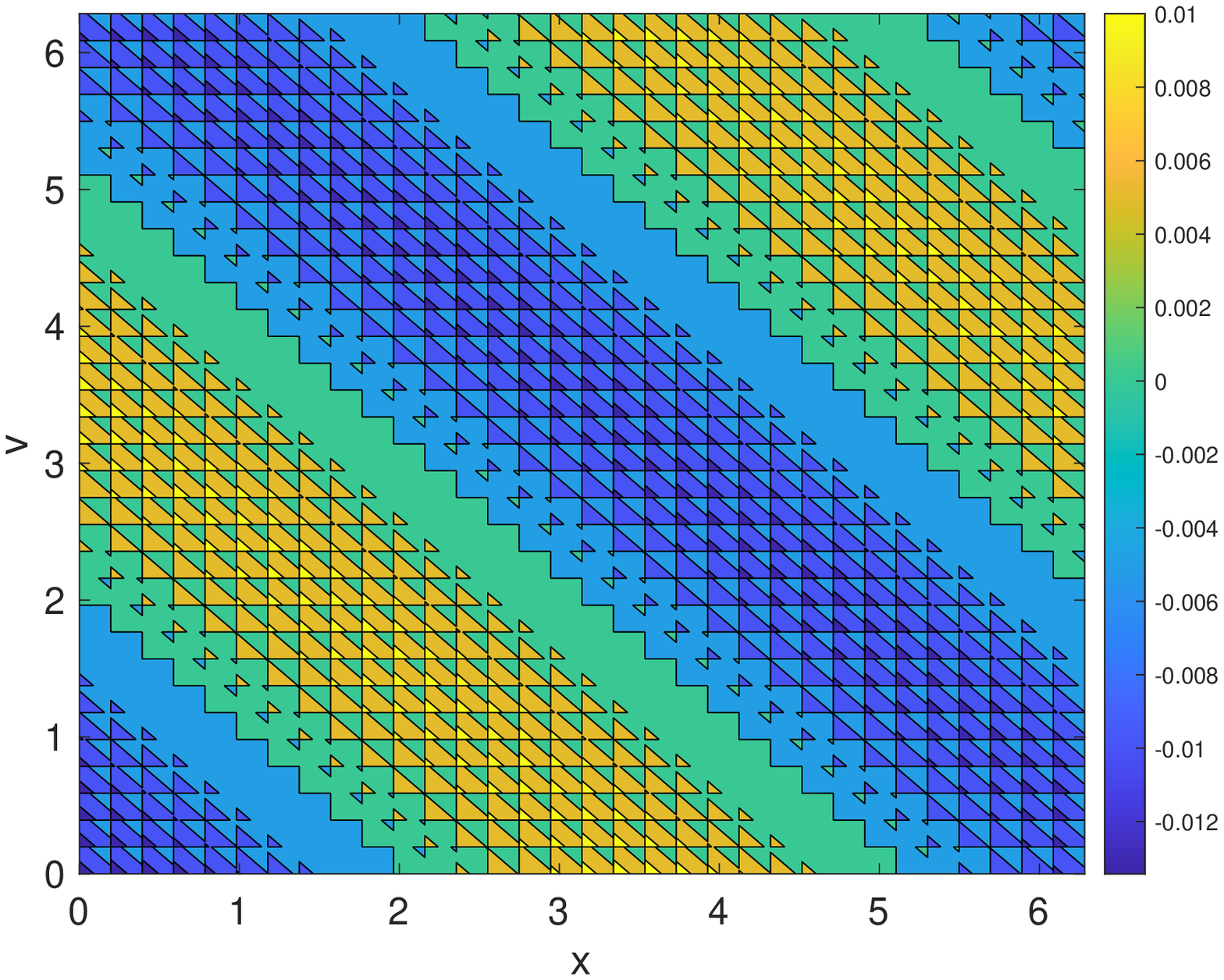}}
	\subfigure[$N_x\times N_v=64\times64$]{\includegraphics[height=40mm]{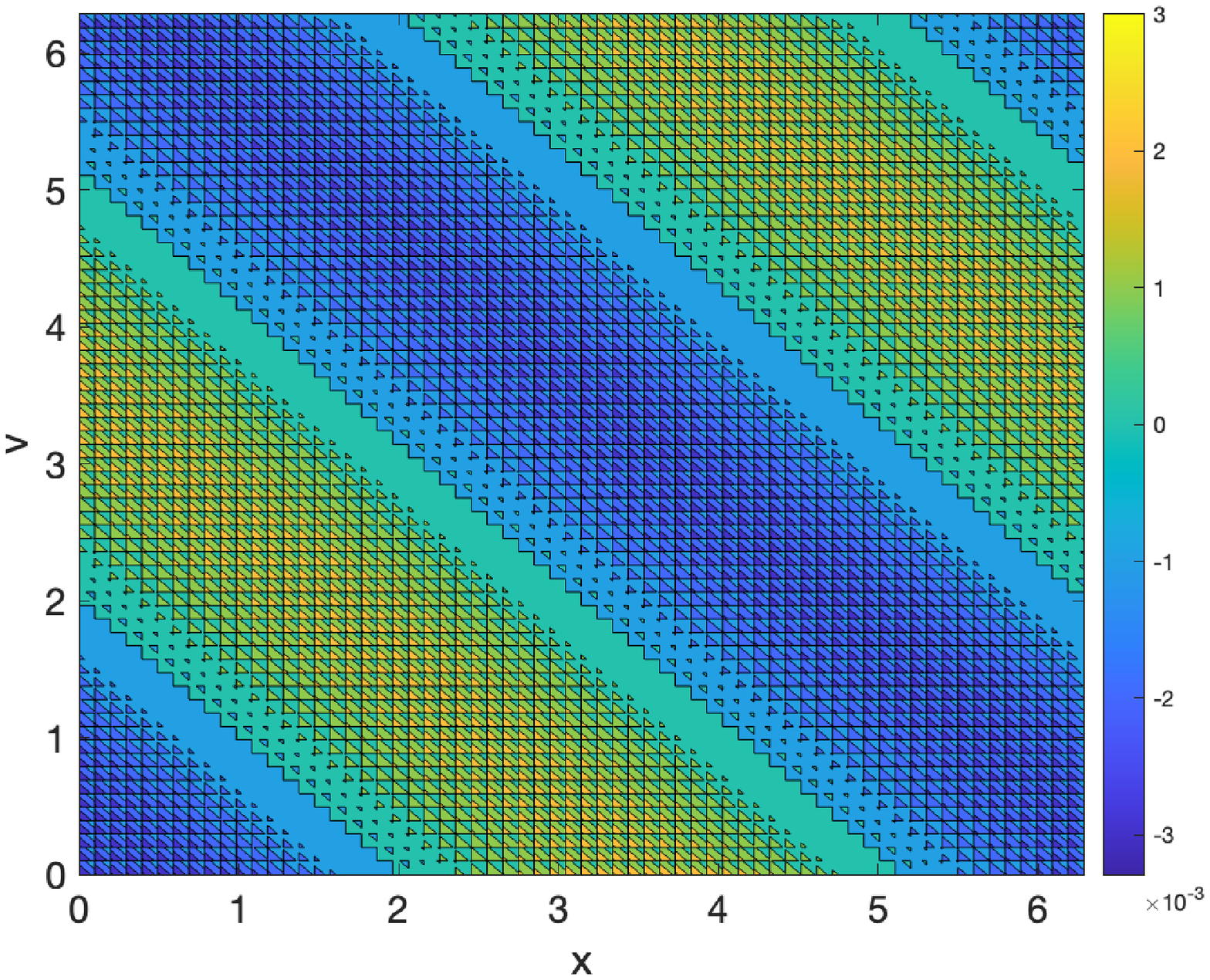}}
       \subfigure[$N_{x_1}\times N_{x_2}=16\times16$]{\includegraphics[height=40mm]{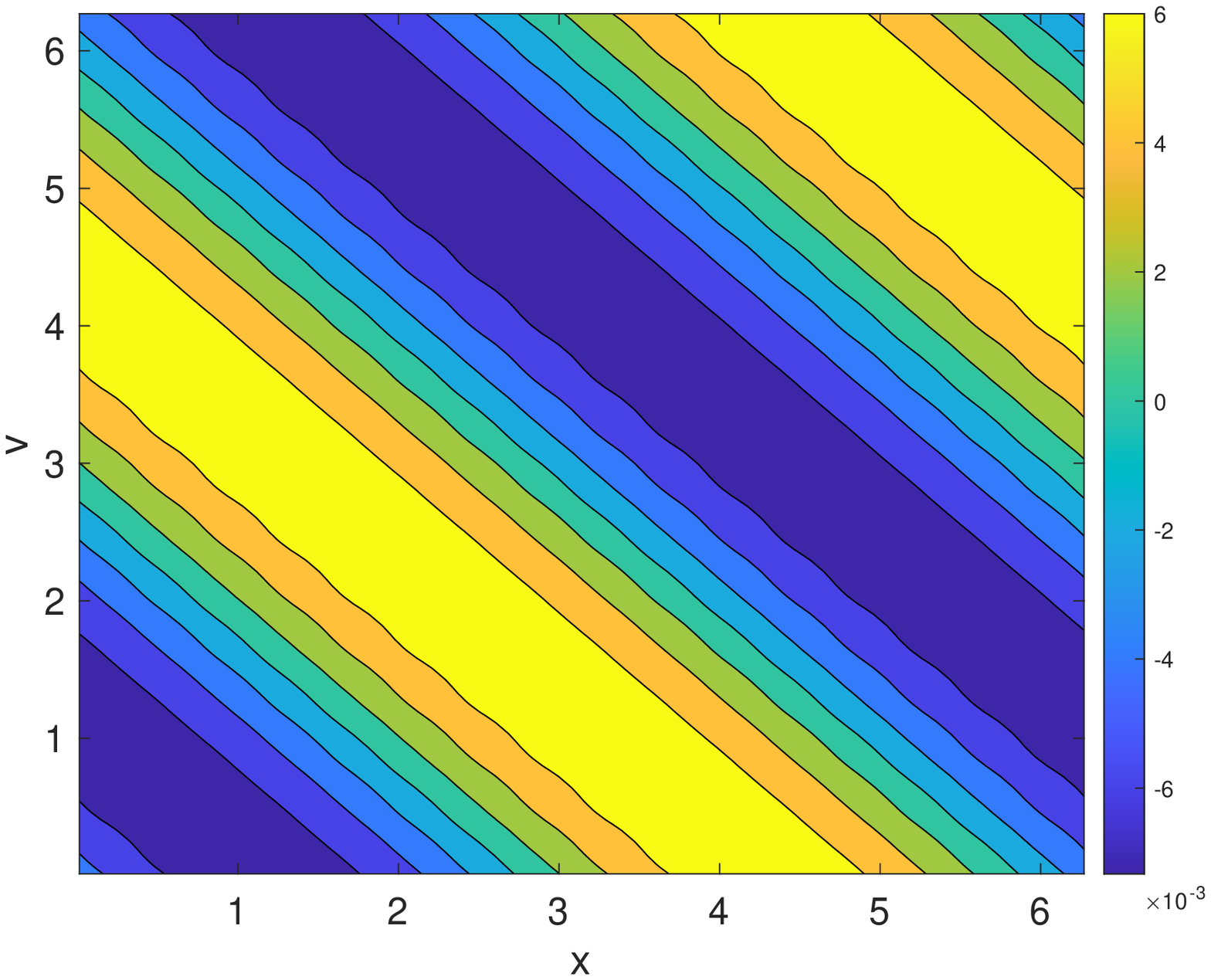}}
	\subfigure[$N_{x_1}\times N_{x_2}=32\times32$]{\includegraphics[height=40mm]{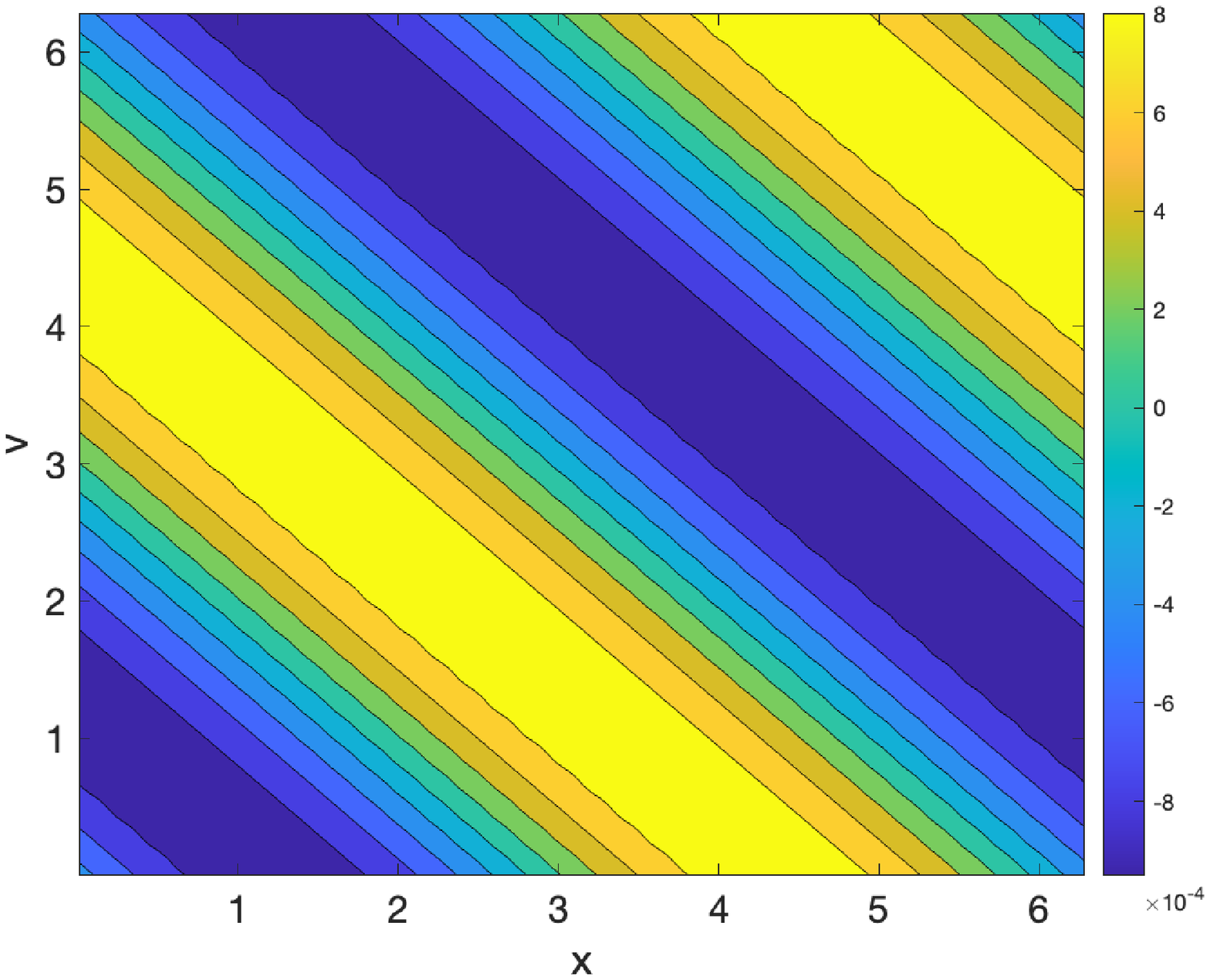}}
	\subfigure[$N_{x_1}\times N_{x_2}=64\times64$]{\includegraphics[height=40mm]{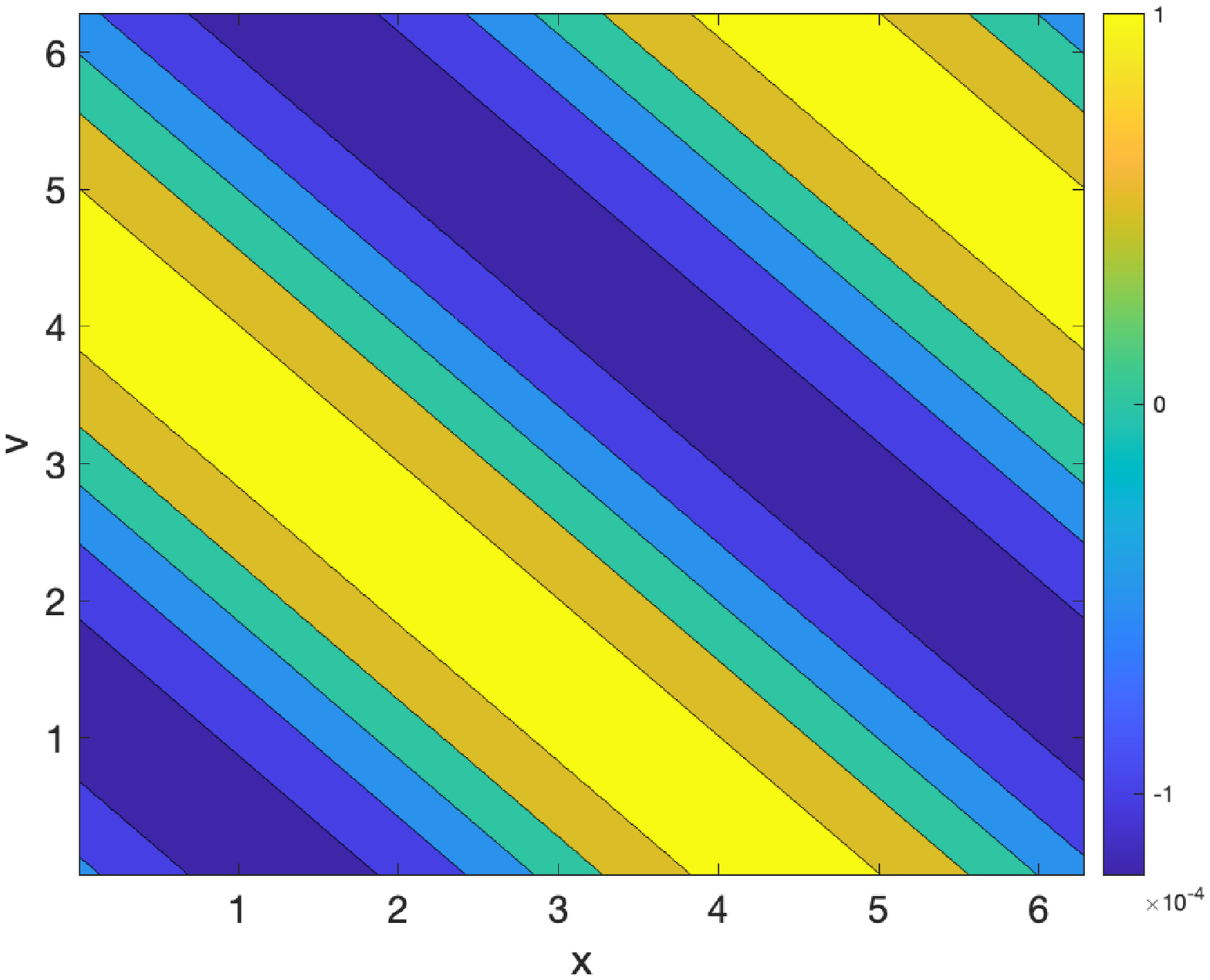}}
			\caption{Example \ref{ex:linear}. Error plots before and after post-processing at $t=1$. $k=1$. $\varepsilon=10^{-4}$.}
	\label{fig:linear_error}
\end{figure}

\begin{figure}[h!]
	\centering
		\includegraphics[height=50mm]{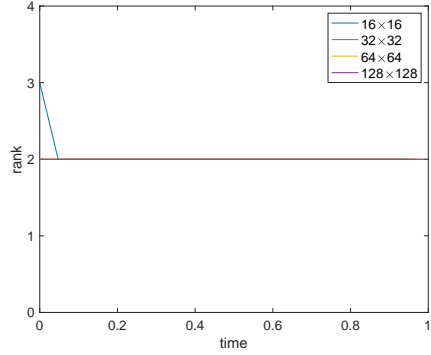}
	\caption{Example \ref{ex:linear}.  The time evolution of  ranks of the DG solutions. $k=1$. $\varepsilon=10^{-4}$.}
	\label{fig:linear_rank}
\end{figure}

\end{exa}

\subsection{1D1V Vlasov-Poisson system}
\begin{exa}\label{ex:forced}(A forced VP system \cite{de2012high}.) We simulate the following forced VP system 
\begin{align*}
\frac{\partial f}{\partial t} + vf_x + Ef_v &= \psi(x,v,t),\\
E(x,t)_x &= \rho(x,t) - \sqrt{\pi},
\end{align*}
where the source $\psi$ is defined as
$$
\psi(x, v, t)=\left(\left(\left(4 \sqrt{\pi}+2\right) v-\left(2 \pi+\sqrt{\pi}\right)\right) \sin (2 x-2 \pi t)
+\sqrt{\pi}(\frac14-v) \sin (4 x-4 \pi t)\right)\exp\left(-\frac{(4 v-1)^{2}}{4}\right) 
$$
so that the system has the exact solution
$$
\begin{aligned}
f(x, v, t) &=\left(2-\cos (2 x-2 \pi t)\right) \exp\left(-\frac{(4 v-1)^{2}}{4}\right), \\
E(x, t) &=-\frac{\sqrt{\pi}}{4} \sin (2 x-2 \pi t).
\end{aligned}
$$
Periodic boundary conditions are imposed. Note that the forced system satisfies the following the macroscopic system
\begin{align*}
\partial_{t} \rho + \bJ_x &= \frac{\sqrt{\pi}}{4}(1-4\pi)\sin(2x-2\pi t)\\
\partial_{t} \bJ +\bm{\sigma}_x&= \rho E + \frac{\sqrt{\pi}}{16}(3+ 4\sqrt{\pi} -4 \pi)\sin(2x-2\pi t) -\frac{\pi}{16}\sin(4x-4\pi t)\\
\partial_{t} e + \mathbf{Q}_x& = \frac{\sqrt{\pi}}{128}(7 + 8\sqrt{\pi}-12\pi)\sin(2x-2\pi t)-\frac{\pi}{64}\sin(4x-4\pi t) \\
& + \frac{\sqrt{\pi}}{8}\left( 2  - (1-4\pi)\cos(2x-2\pi t) \right)E.
\end{align*}
It is easily verified that the total mass, total momentum, and total energy of the system is conserved. For this example, we test the accuracy of  the proposed low rank DG method  and justify its ability  to conserve the physical invariants. In the simulation, we set the truncation threshold $\varepsilon=10^{-3}$ and set the computational domain $[-\pi,\pi]\times[-L_v,L_v]$ with $L_v=4$. The convergence study is summarized in Table \ref{tb:forced}, and $k+1$-th order of convergence is observed for both $L^2$ and $L^\infty$ errors. To showcase the flexibility of DG meshes, we perturb the uniform mesh randomly by up to 10\%. In Figures \ref{fig:forced1d_invar_k1}-\ref{fig:forced1d_invar_k2}, we report the time histories of  numerical ranks of the low rank DG solutions together with relative deviation of the total mass, total momentum and total energy for $k=1$ and $k=2$, respectively. It is observed that for $k=1$, the ranks of the numerical solution over a coarser mesh ($N_x\times N_v=16\times32$) are higher than that over a finer mesh and also increase over time, which is attributed to the large DG discretization error. For $k=2$, the ranks of the numerical solutions stay four during the time evolution. Hence, it is advantageous to employ a higher order DG discretization for this problem. Here the rank four comes from the a rank one from the exact solution, and rank three from conservative projection for mass, momentum and energy. We can observe that total mass, total momentum, and total energy are conserved up to machine precision for both $k=1$ and $k=2$ with all mesh sizes, indicating that the MaLoC property of the proposed method is independent of the degree $k$ and mesh size used.

\begin{table}[!hbp]
	\centering
	\caption{Example \ref{ex:forced}. $t=1$. Convergence study. The non-uniform meshes are obtained by randomly perturbing the element boundaries of uniform meshes up to 10\%.}
	\label{tb:forced}
	\begin{tabular}{|c|c|c|c|c|c|c|c|c|}
		\hline
	 \multirow{2}{*}{$N_x\times N_v$}  & \multicolumn{4}{|c|} {$k=1$} & \multicolumn{4}{|c|} {$k=2$} \\\cline{2-9}
		  & $L^2$ error & order & $L^\infty$ error & order & $L^2$ error & order & $L^\infty$ error & order\\\hline
$16\times32$	&	1.37E-01	&		&	1.33E-01	&		&	6.07E-03	&		&	9.73E-03	&		\\\hline
$32\times64$	&	3.83E-02	&	1.83	&	3.32E-02	&	2.01	&	9.15E-04	&	2.73	&	1.52E-03	&	2.68	\\\hline
$64\times128$	&	4.33E-03	&	3.15	&	6.31E-03	&	2.39	&	1.07E-04	&	3.10	&	1.91E-04	&	2.99	\\\hline
$128\times256$	&	1.12E-03	&	1.95	&	1.57E-03	&	2.01	&	1.23E-05	&	3.11	&	2.26E-05	&	3.08	\\\hline
	\end{tabular}
\end{table}

\begin{figure}[h!]
	\centering
		\subfigure[]{\includegraphics[height=50mm]{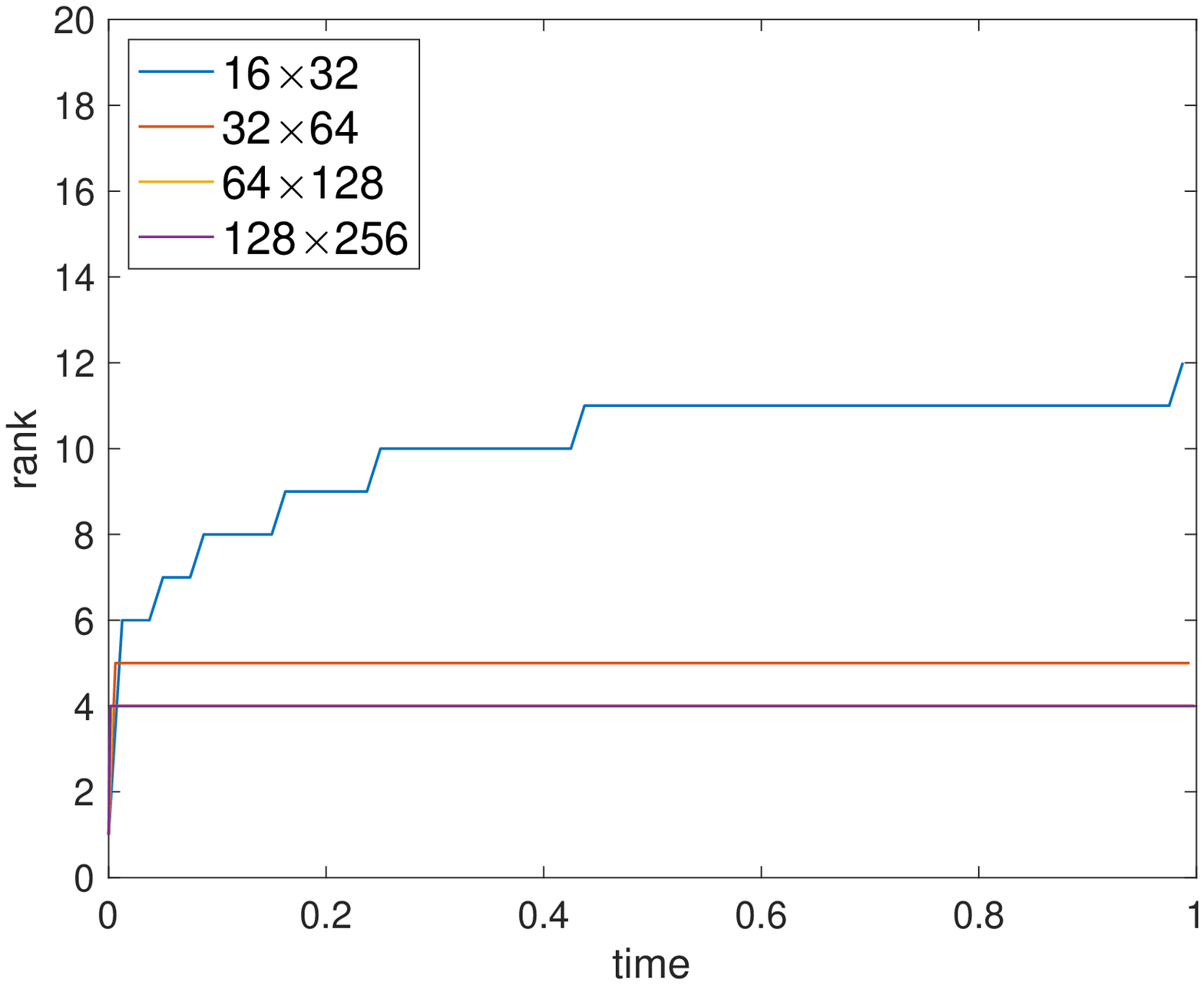}}
		\subfigure[]{\includegraphics[height=50mm]{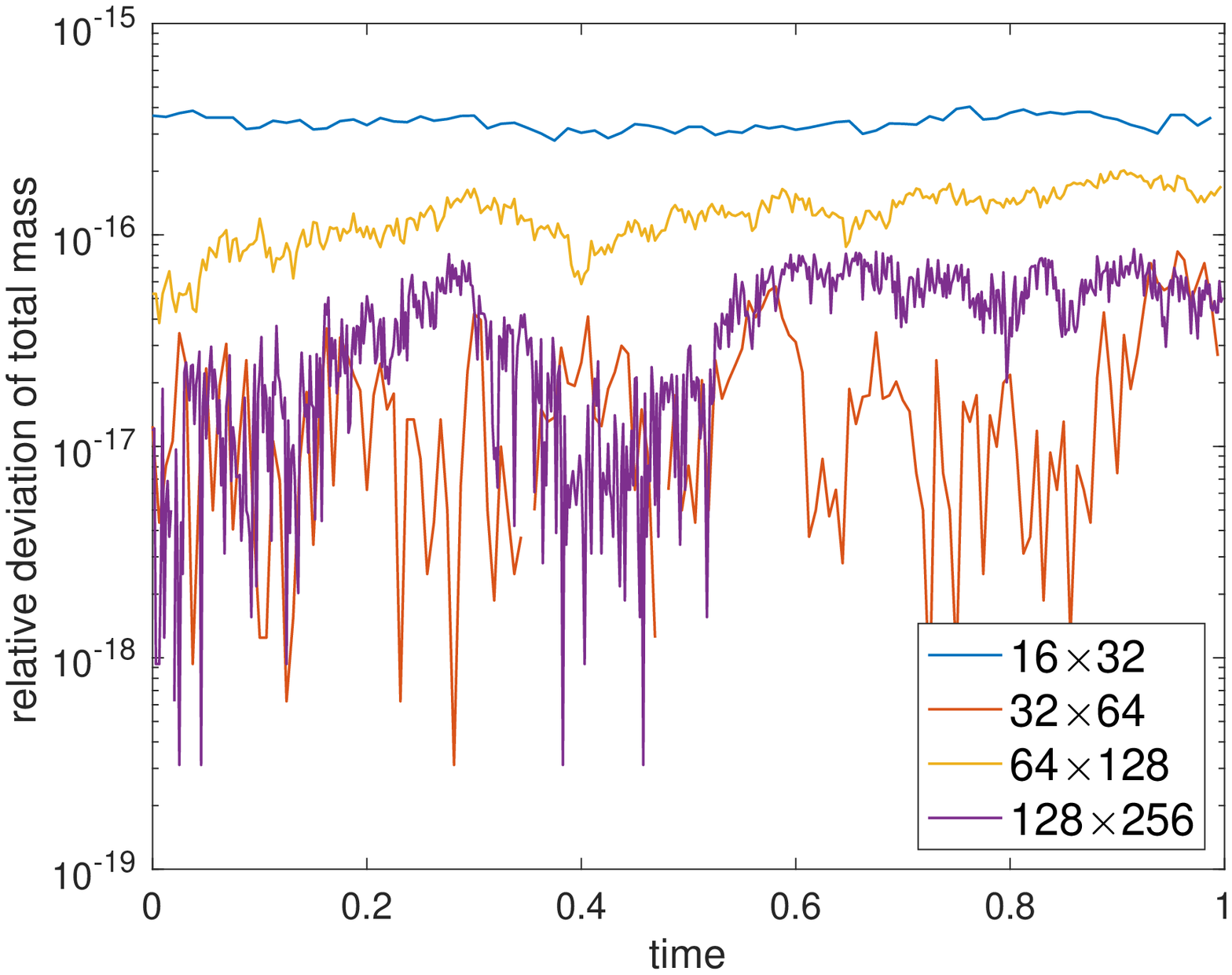}}
	       \subfigure[]{\includegraphics[height=50mm]{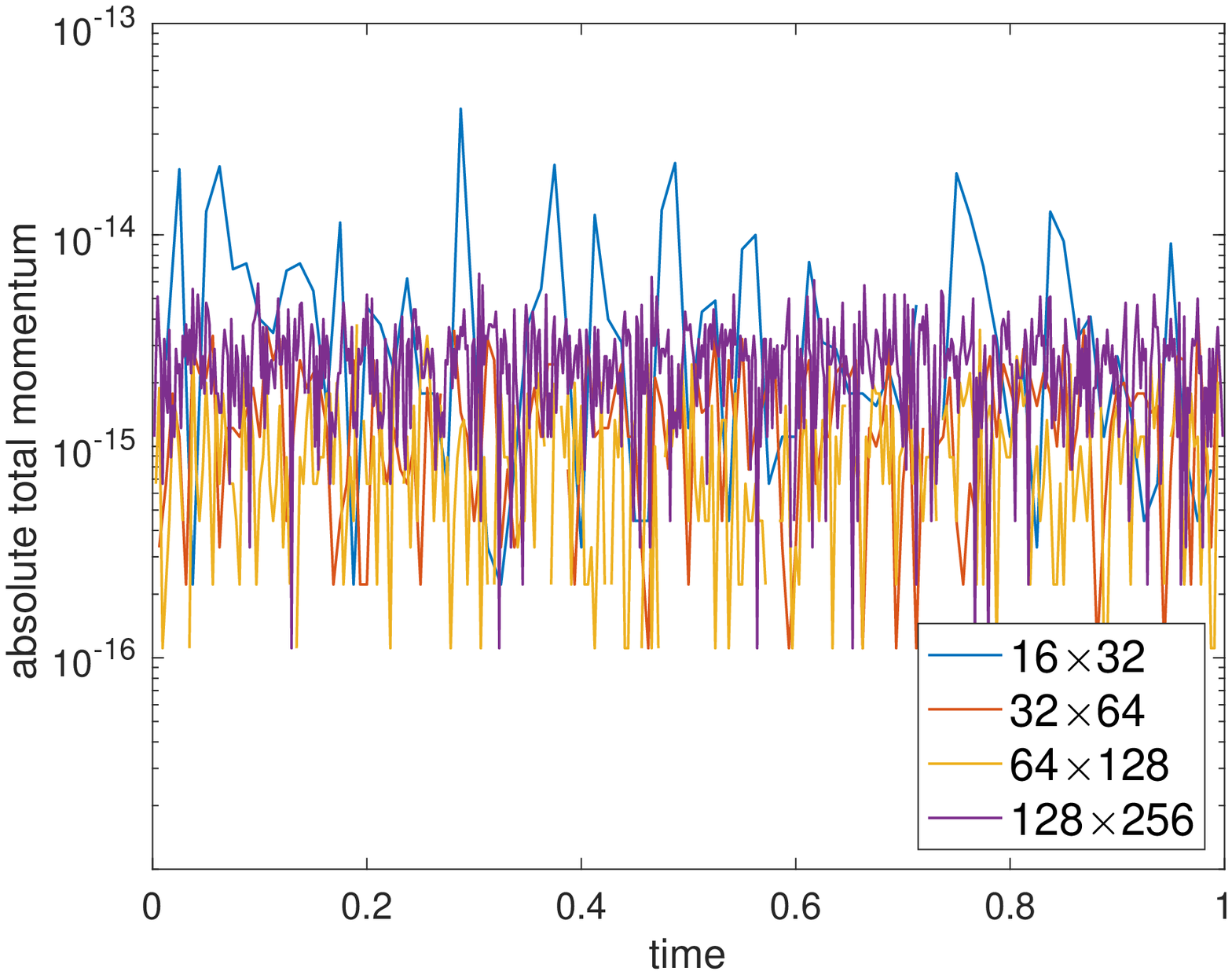}}
		\subfigure[]{\includegraphics[height=50mm]{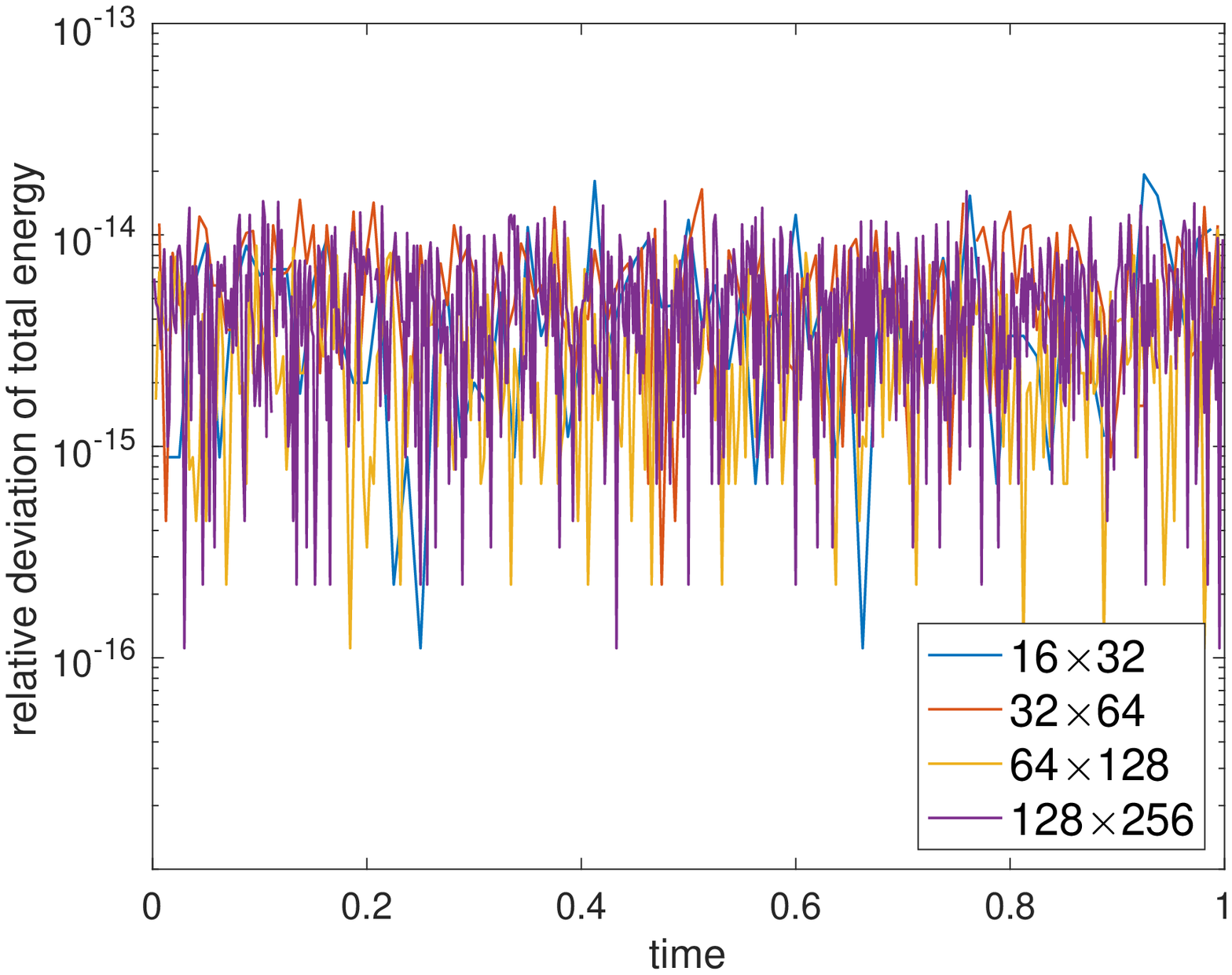}}
	\caption{Example \ref{ex:forced}.  The time evolution of  ranks of the numerical solutions (a), relative deviation of total mass (b),  total momentum (c), and total energy (d). $k=1$. $\varepsilon=10^{-3}$.}
	\label{fig:forced1d_invar_k1}
\end{figure}	

\begin{figure}[h!]
	\centering
		\subfigure[]{\includegraphics[height=50mm]{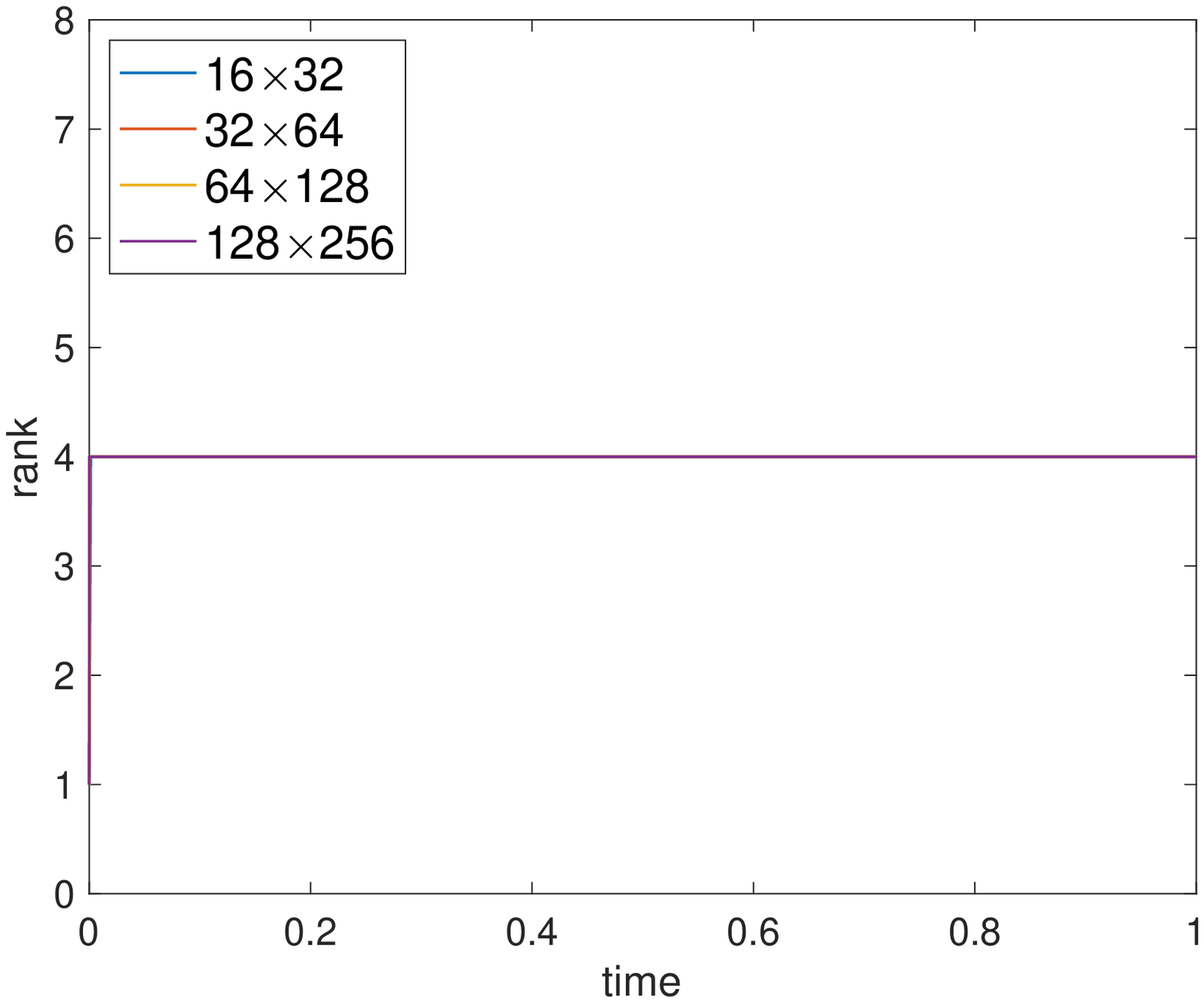}}
		\subfigure[]{\includegraphics[height=50mm]{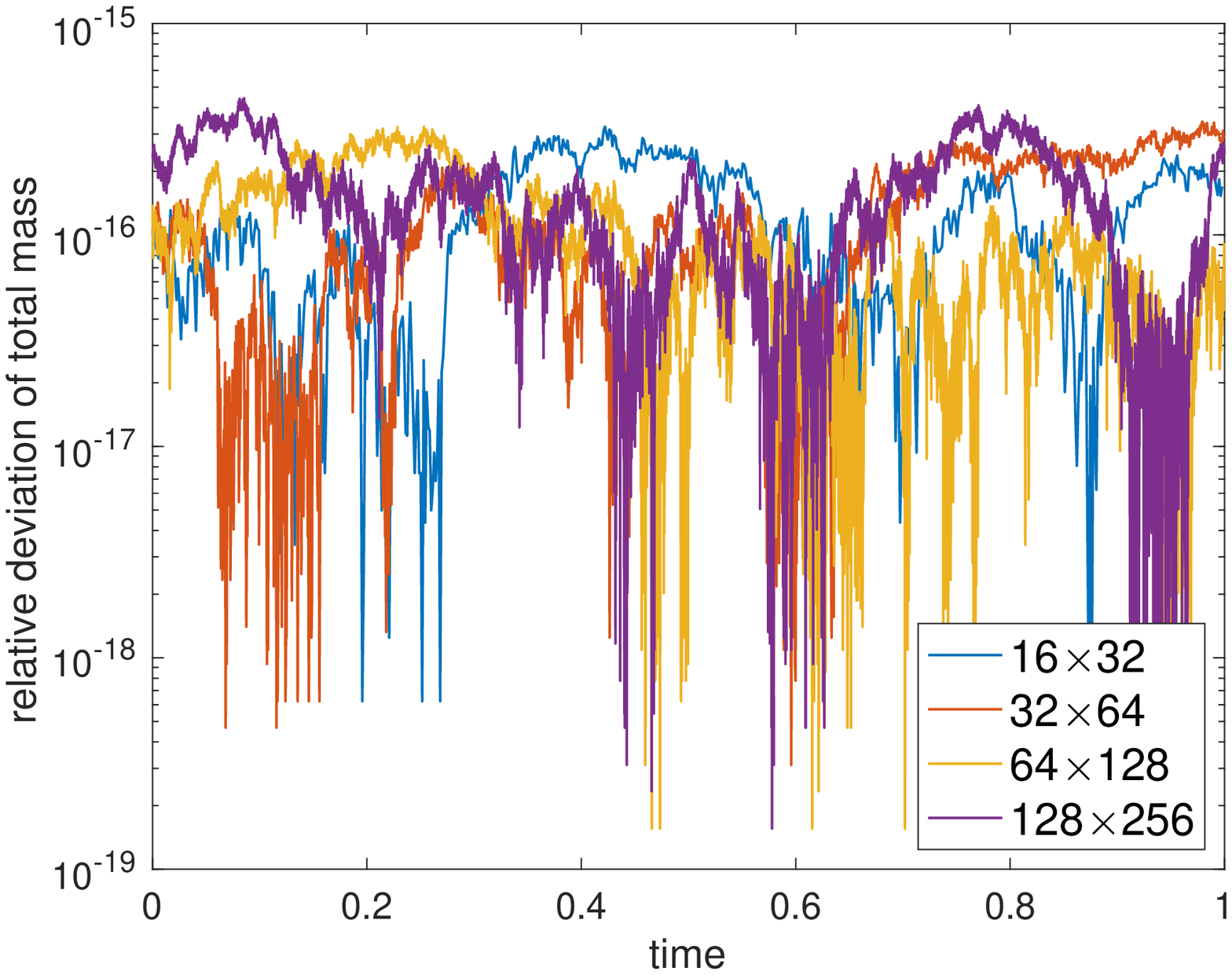}}
	       \subfigure[]{\includegraphics[height=50mm]{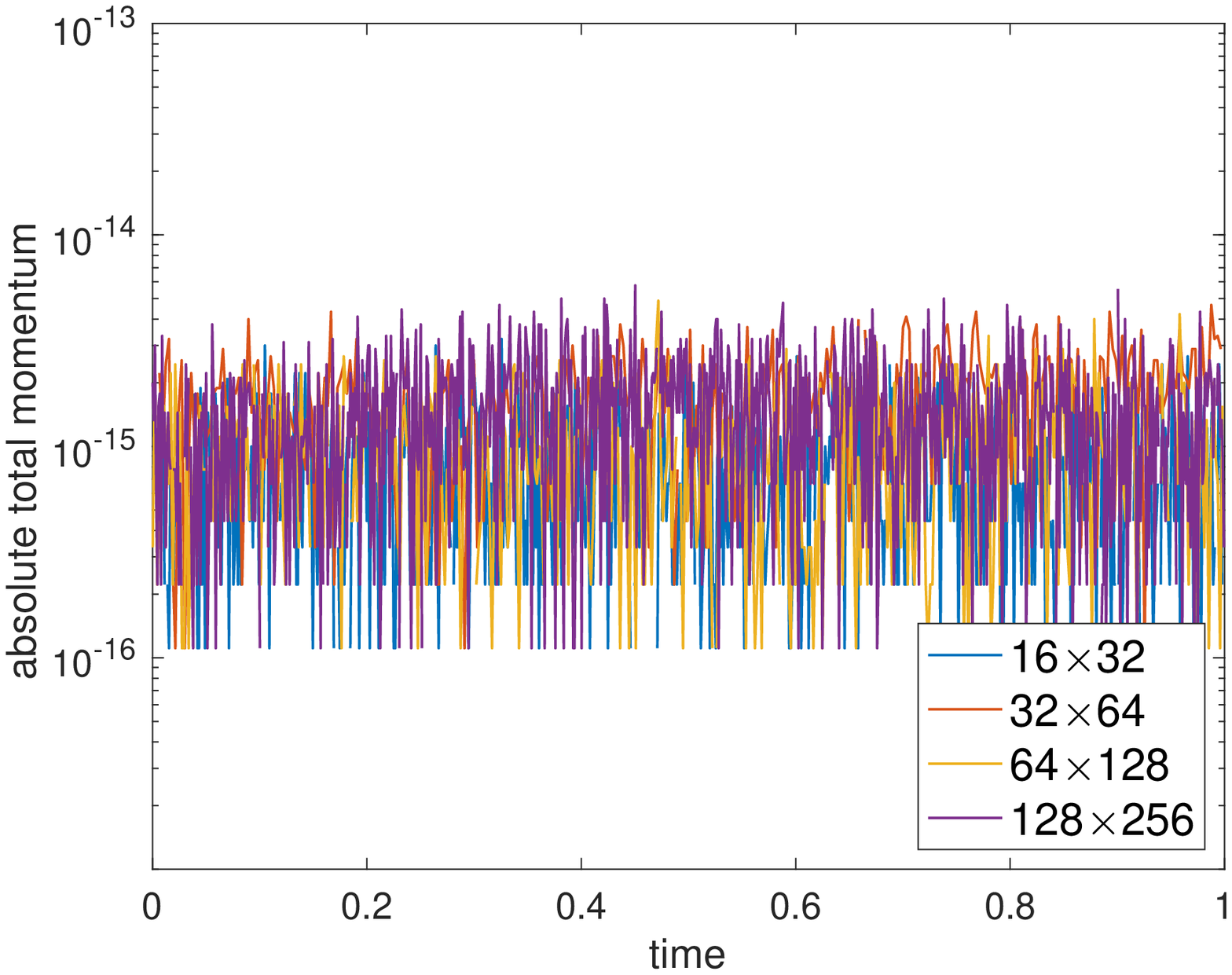}}
		\subfigure[]{\includegraphics[height=50mm]{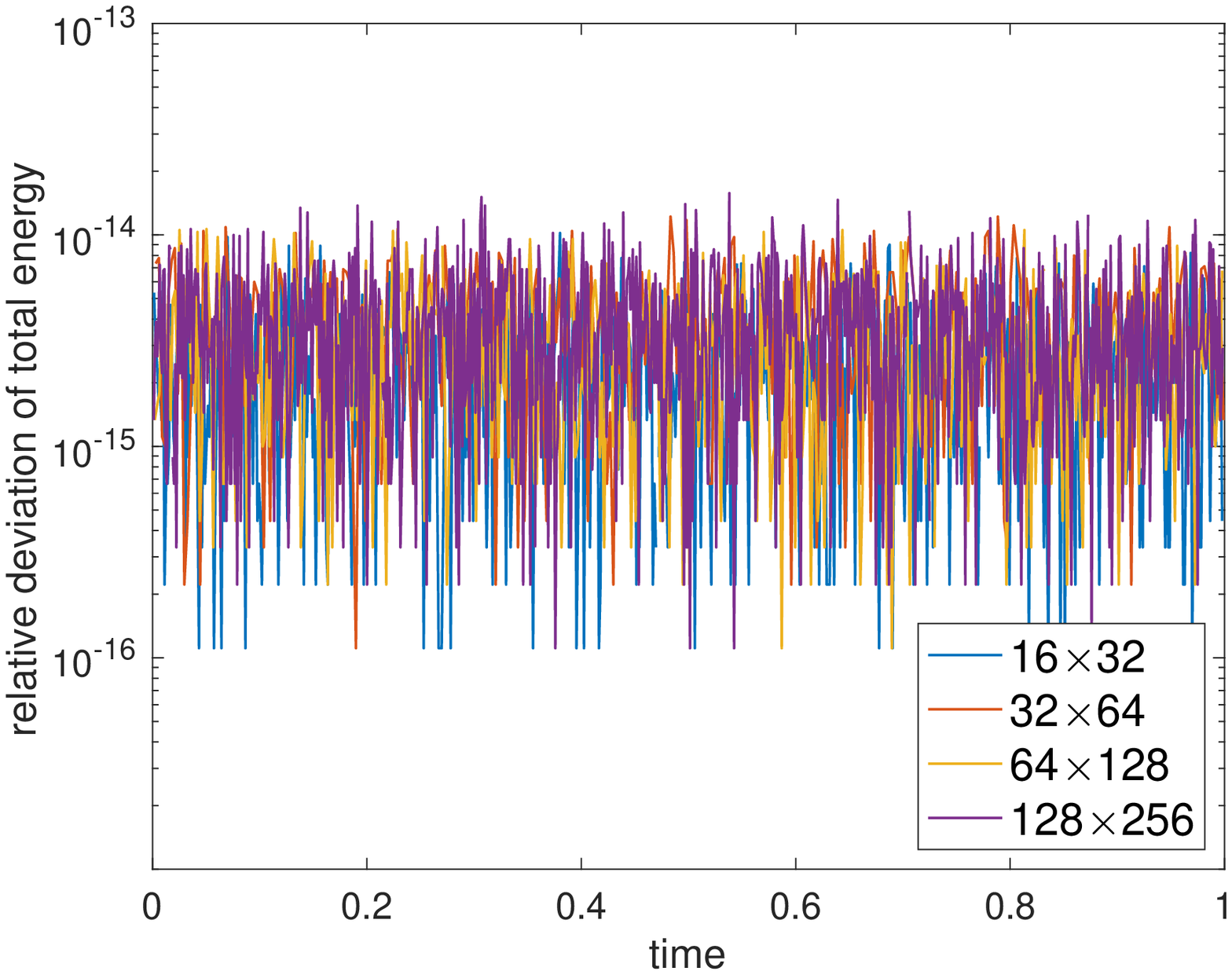}}
	\caption{Example \ref{ex:forced}.  The time evolution of  ranks of the numerical solutions (a), relative deviation of total mass (b),  total momentum (c), and total energy (d). $k=2$. $\varepsilon=10^{-3}$.}
	\label{fig:forced1d_invar_k2}
\end{figure}	

\end{exa}

\begin{exa}
	\label{ex:weak1d}(Weak Landau damping.) We simulate the weak Landau damping test with initial condition 
	\begin{equation}
		\label{eq:landau1d}
		f(x,v,t=0) =\frac{1}{\sqrt{2 \pi}} \left(1+\alpha  \cos \left(k x\right)\right)\exp\left(-\frac{v^2}{2}\right),
	\end{equation}
	where $\alpha=0.01$ and $k=0.5$. The computational domain is set to be $[0,L_x]\times[-L_v,L_v]$ with $L_x=2\pi/k$ and $L_v=6$. We set $\varepsilon=10^{-5}$ for truncation. In the simulation, we employ a set of non-uniform meshes by randomly perturbing uniform meshes up to 10\%. In Figure \ref{fig:weak1d_elec}, we report the time histories of the electric energy and numerical ranks of the low rank DG solutions for $k=1$ and $k=2$. It is observed that the low rank method is able to predict the correct damping rate of the electric energy. In addition, the method of larger $k$ over a finer mesh can better track the damping phenomenon with lower numerical ranks, justifying the computational advantages of using higher order DG discretization. In Figure \ref{fig:weak1d_invar}, we further report the  time histories of relative deviation of the total mass and total energy, together with absolute derivation of total momentum. We can see that the method is able to conserve the total mass, momentum and energy up to the machine precision.

\begin{figure}[h!]
	\centering
	\subfigure[$k=1$]{\includegraphics[height=60mm]{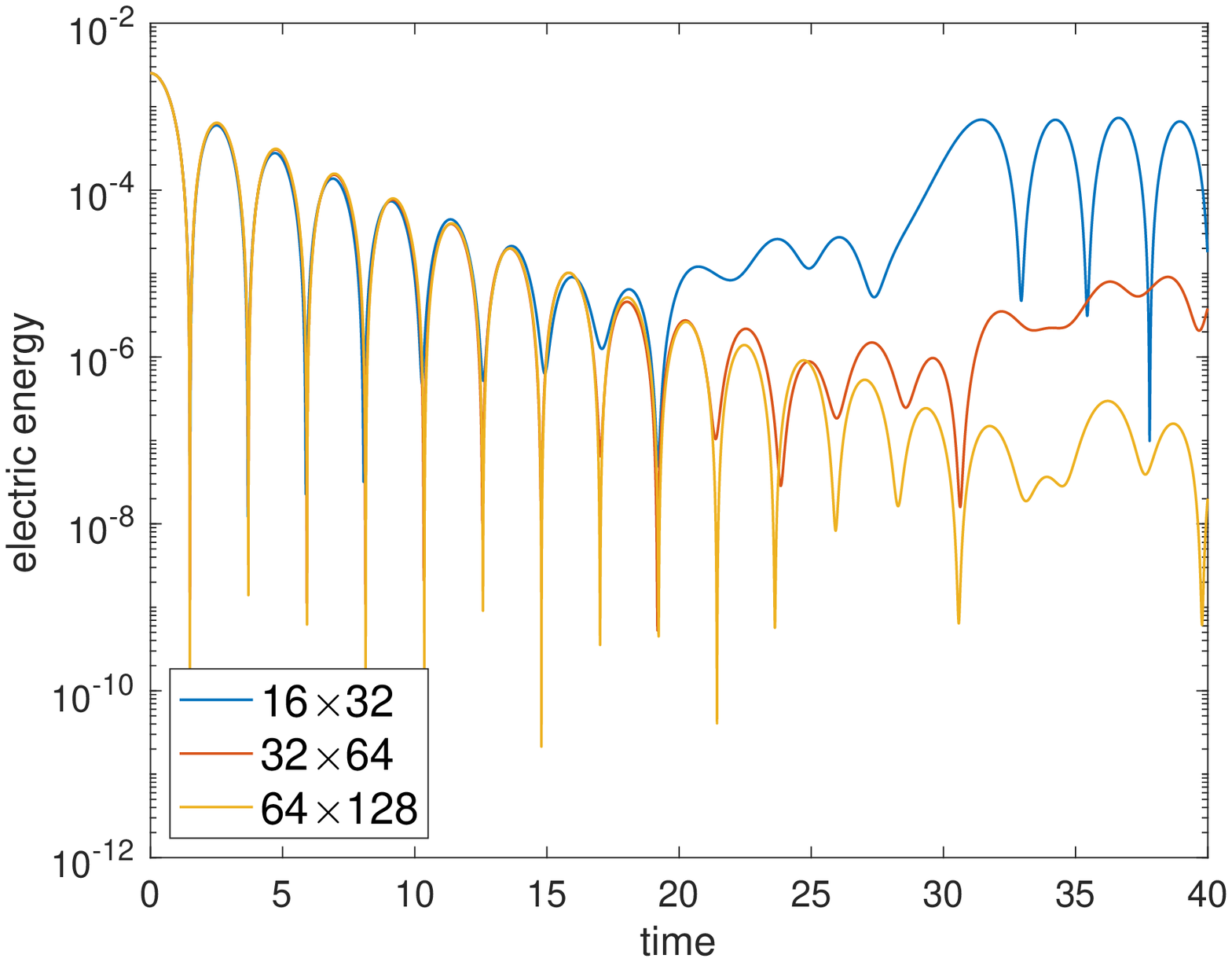}}
	\subfigure[$k=2$]{\includegraphics[height=60mm]{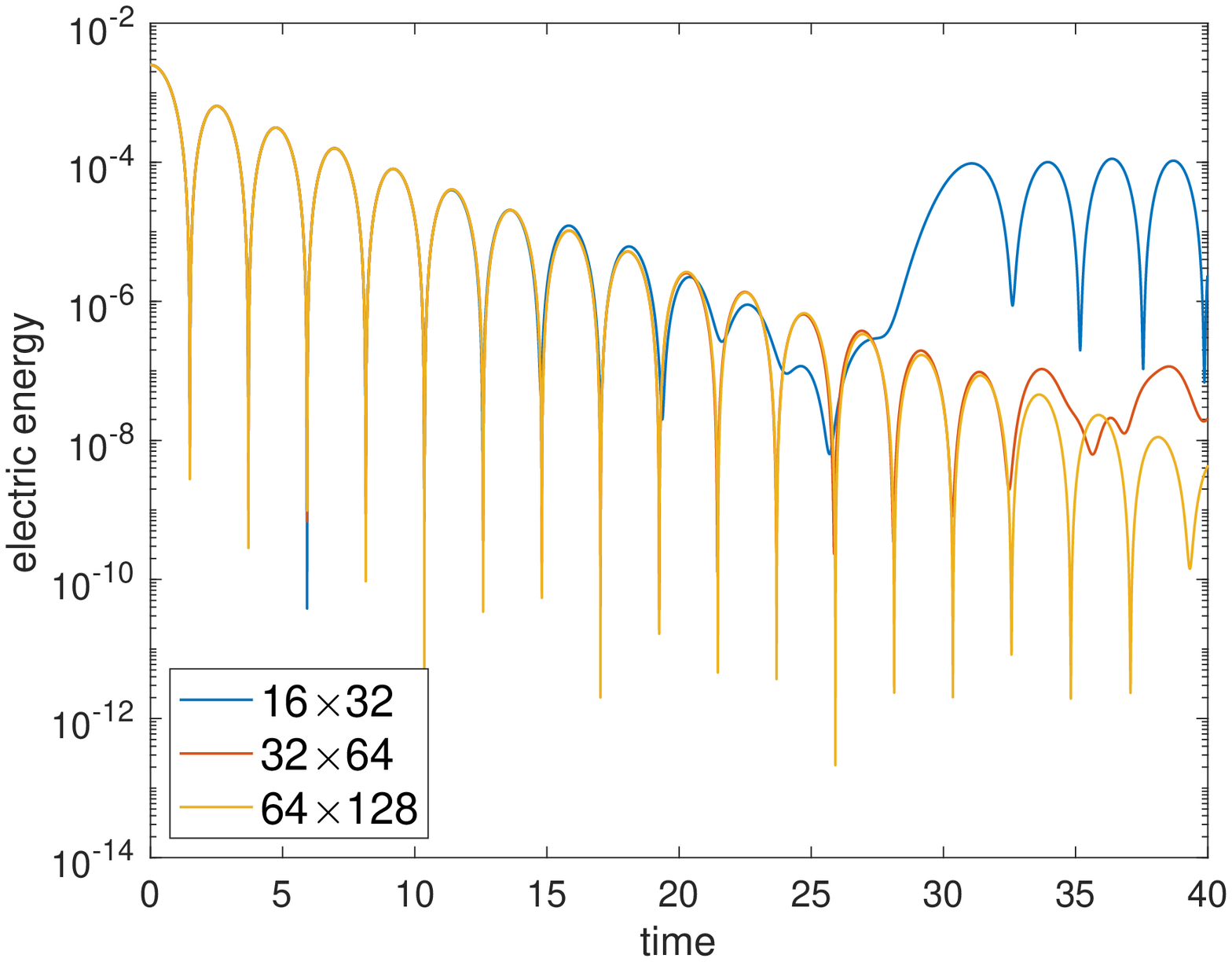}}
	\subfigure[$k=1$]{\includegraphics[height=60mm]{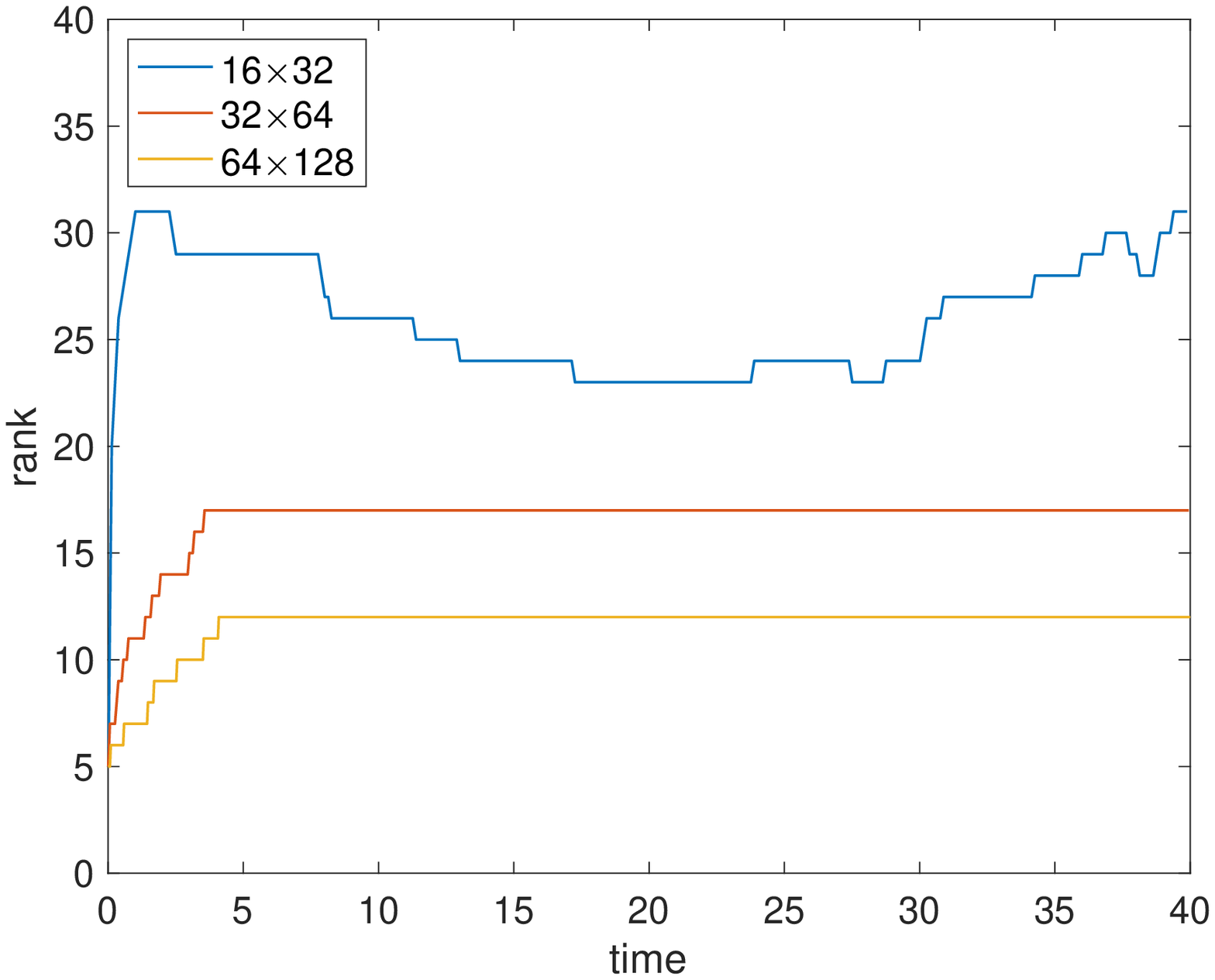}}
	\subfigure[$k=2$]{\includegraphics[height=60mm]{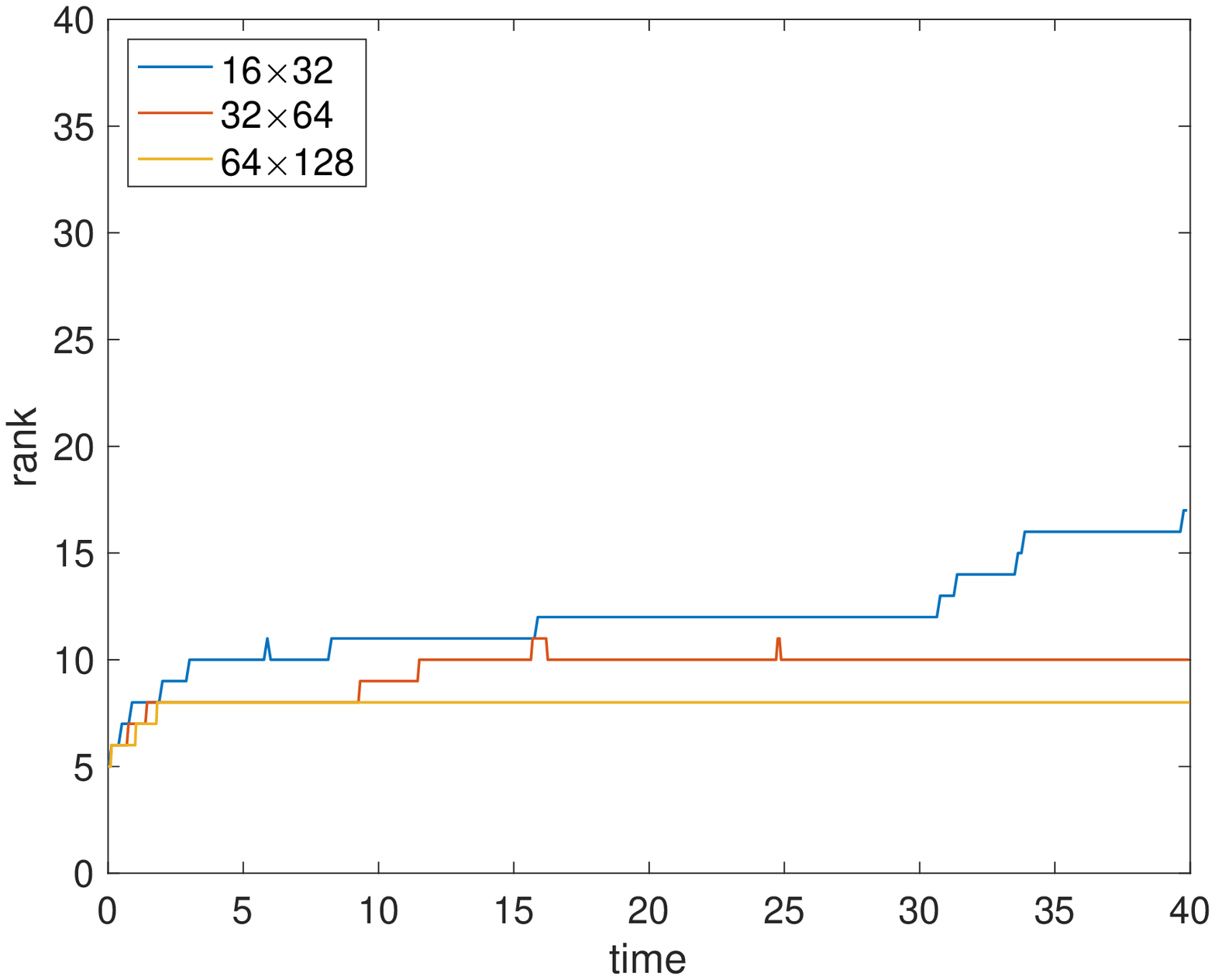}}
	\caption{Example \ref{ex:weak1d}.  The time evolution of  electric energy (a, b) and ranks of the low rank DG solutions (c, b). $\varepsilon=10^{-5}$.}
	\label{fig:weak1d_elec}
\end{figure}	

\begin{figure}[h!]
	\centering
	\subfigure[$k=1$]{\includegraphics[height=40mm]{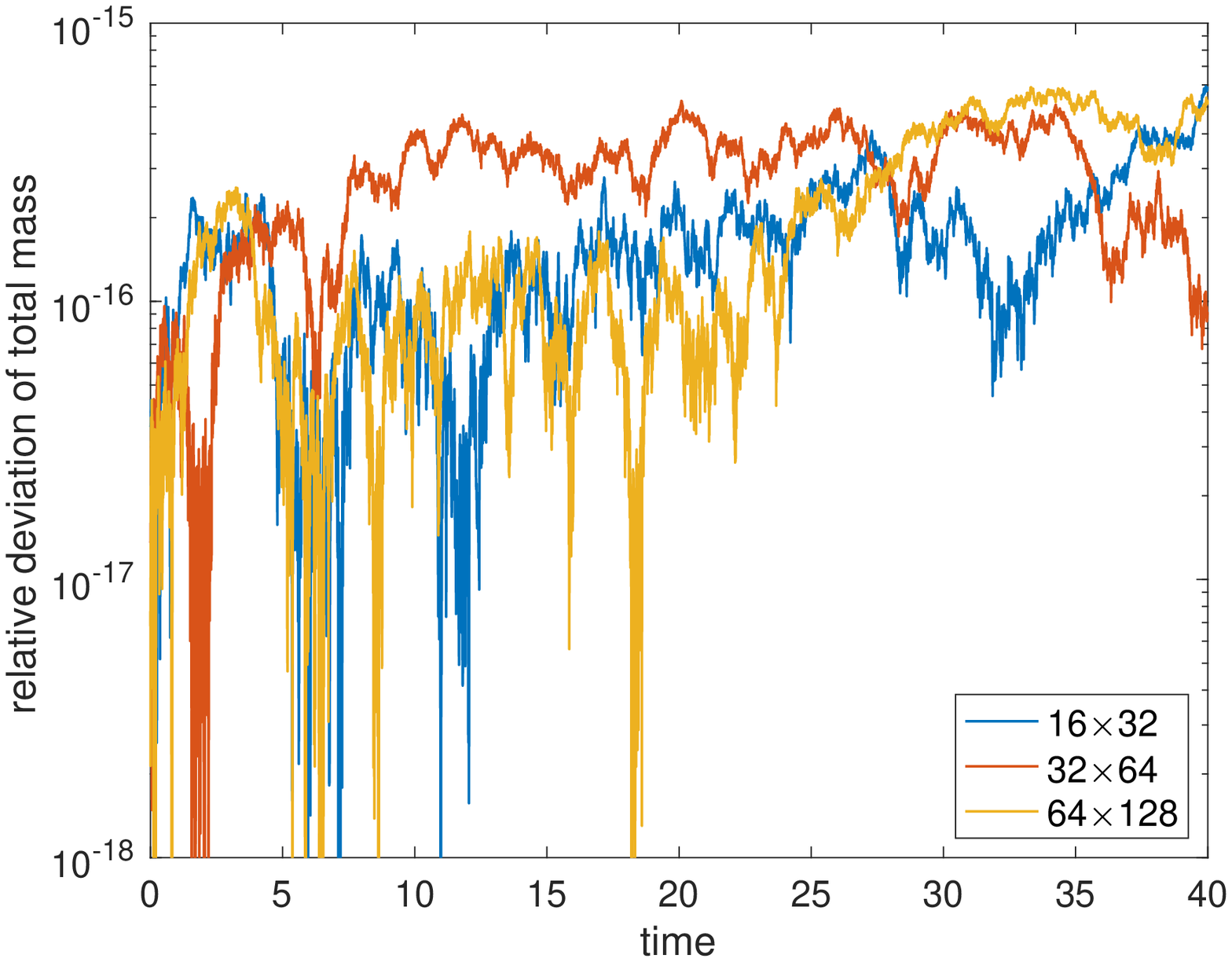}}
	\subfigure[$k=1$]{\includegraphics[height=40mm]{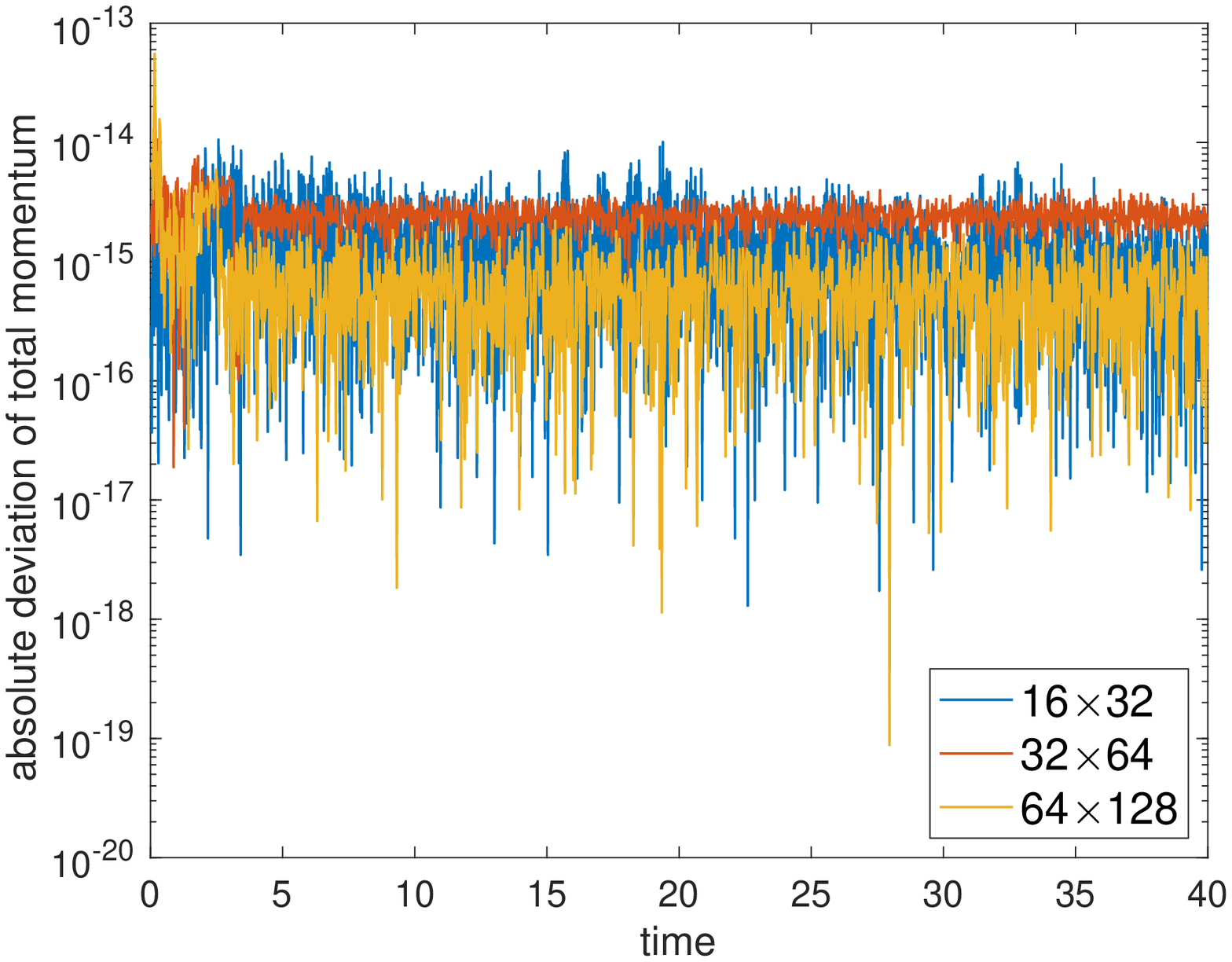}}
	\subfigure[$k=1$]{\includegraphics[height=40mm]{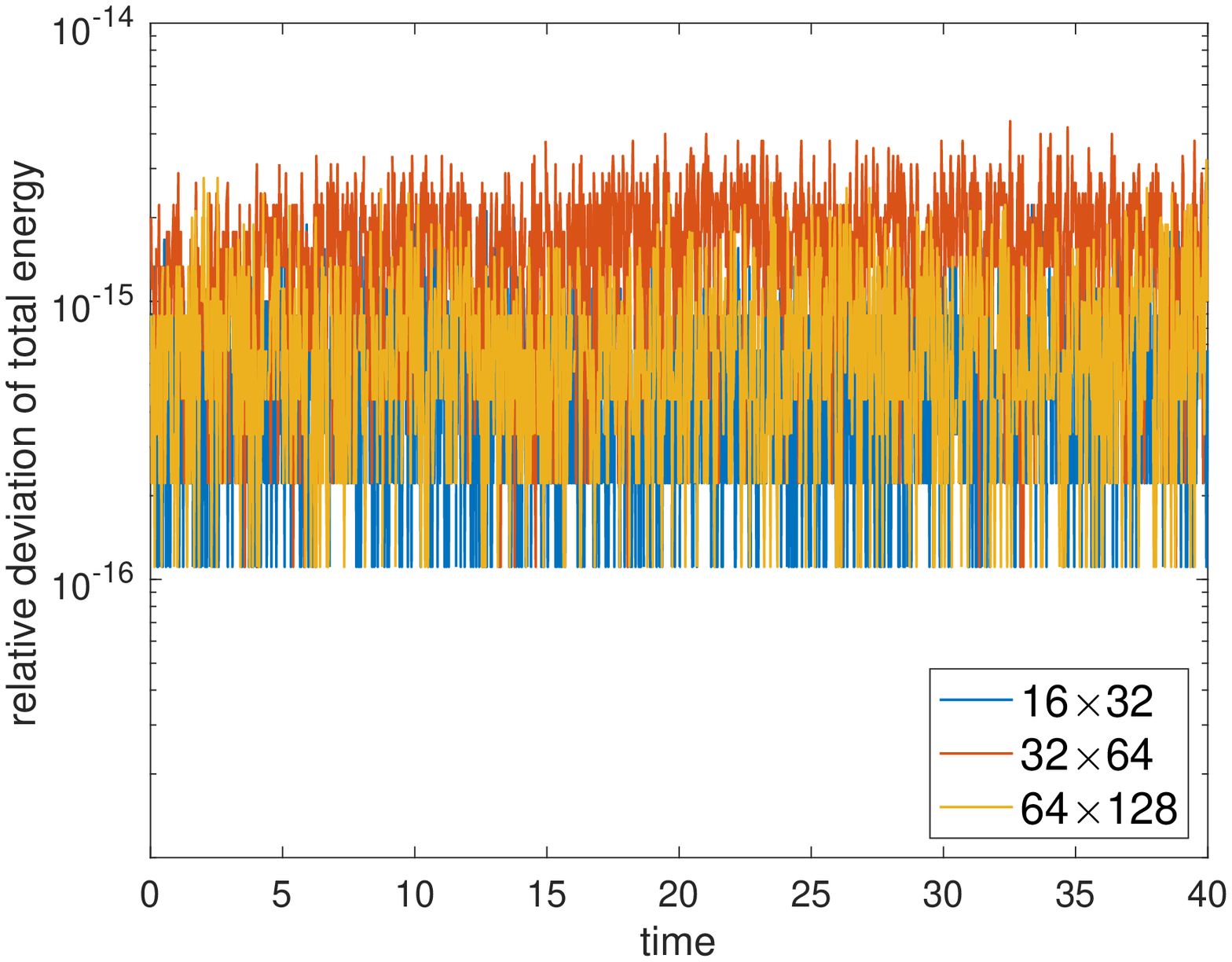}}
	\subfigure[$k=2$]{\includegraphics[height=40mm]{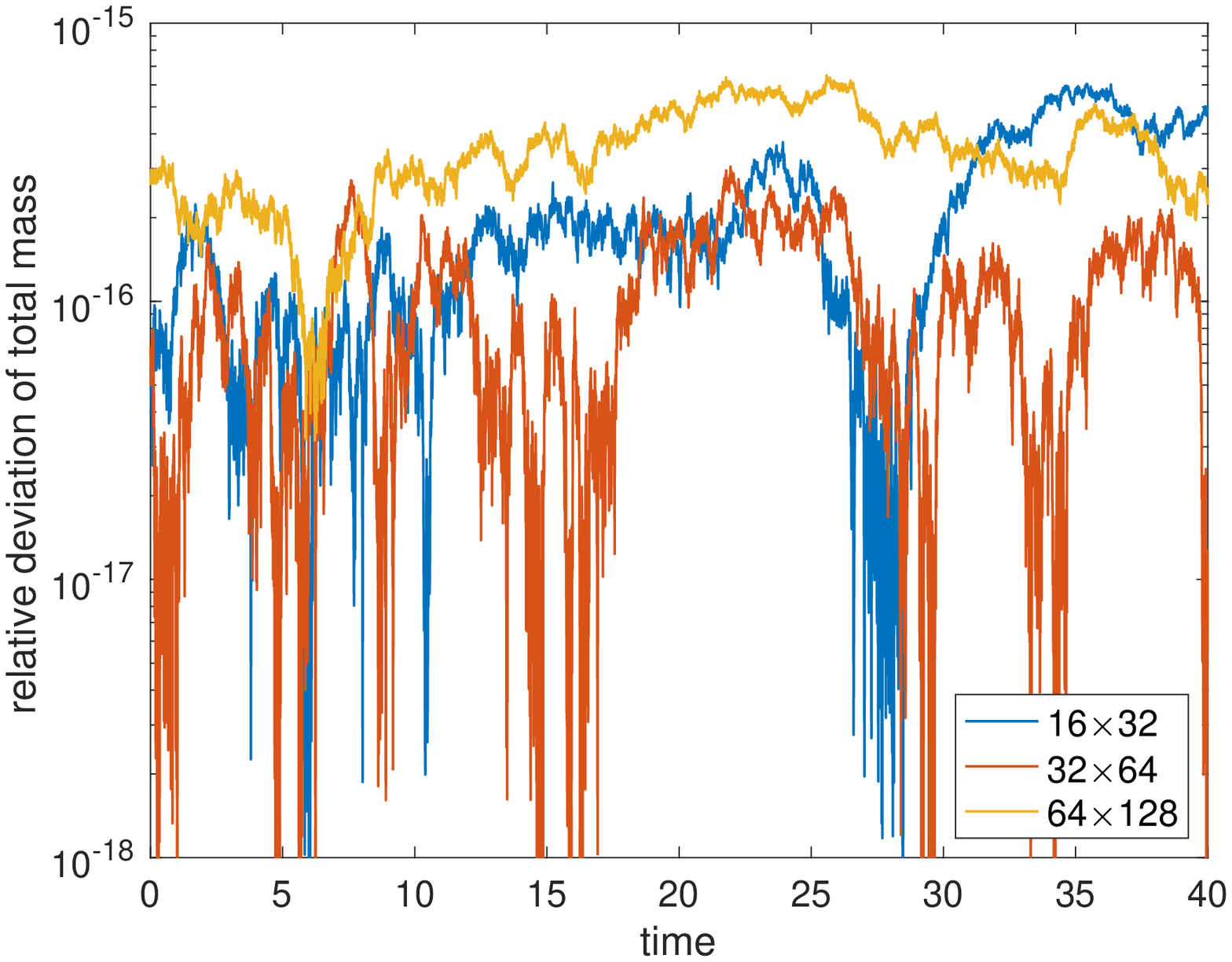}}
	\subfigure[$k=2$]{\includegraphics[height=40mm]{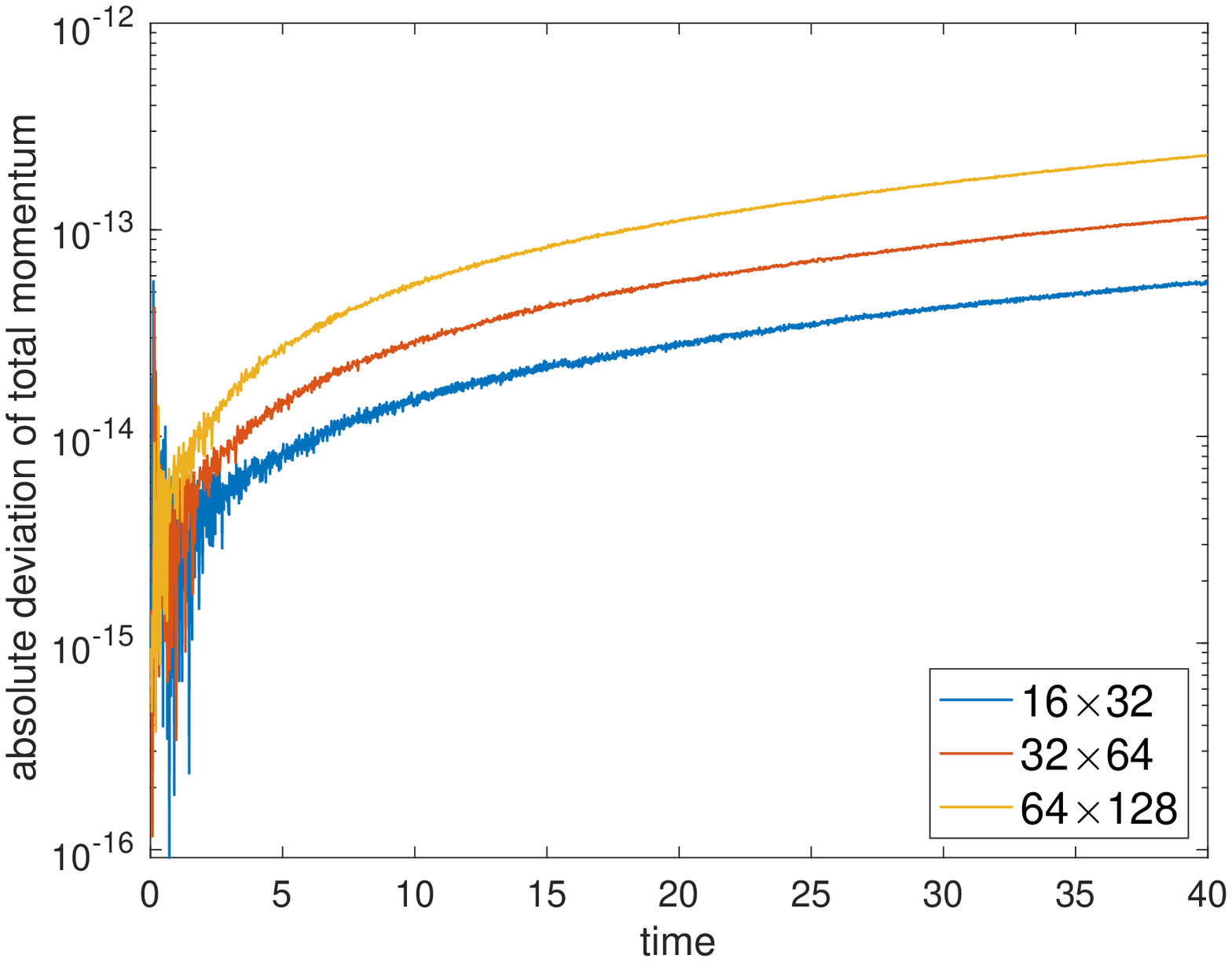}}
	\subfigure[$k=2$]{\includegraphics[height=40mm]{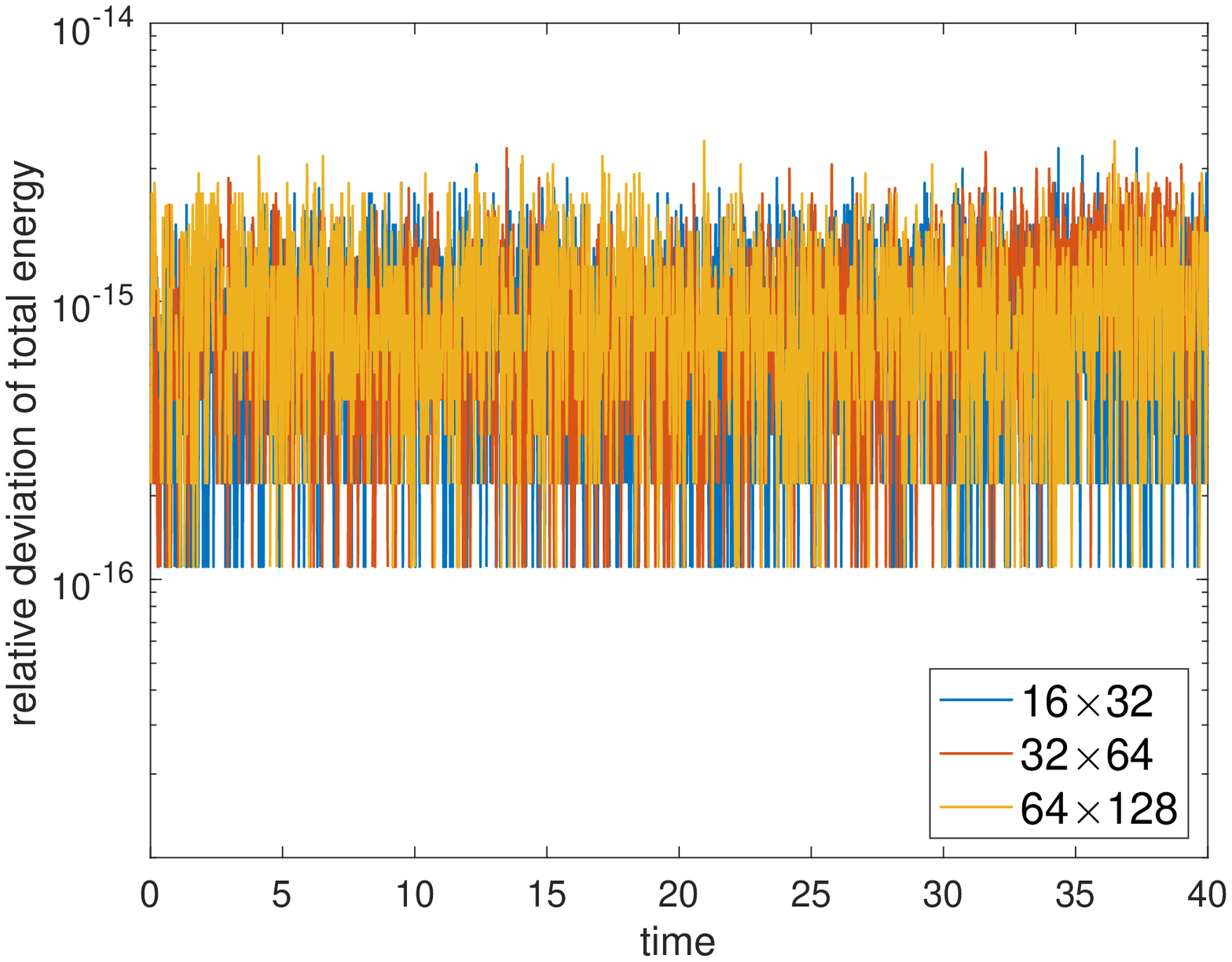}}
	\caption{Example \ref{ex:weak1d}.  The time evolution of  relative deviation of total mass (a, d), absolute deviation of total momentum (b, e), and relative deviation of total energy (c, f).  $\varepsilon=10^{-5}$.}
	\label{fig:weak1d_invar}
\end{figure}	
	
\end{exa}

\begin{exa}
	\label{ex:strong1d}(Strong Landau damping.) For this example, we simulate another benchmark problem, namely the strong Landau damping test. The initial condition is the same as \eqref{eq:landau1d} but with parameters
$\alpha=0.5$ and $k=0.5$. Unlike the previous example, due to the large perturbation the electric energy would decay at first and then start to increase until reaching saturation due to the large perturbation. The computational domain is set to be the same as the previous example. In the simulation, the truncation threshold is set to be $\varepsilon=10^{-3}$, and we employ  non-uniform meshes obtained by randomly perturbing uniform meshes up to 10\%. We summarize the simulation results in Figures \ref{fig:strong1d_elec}-\ref{fig:strong1d_invar}. 
It is observed that the proposed low rank DG method can adapt the numerical ranks to efficiently capture Vlasov dynamics. Furthermore, the method is able to conserve the physical invariants as expected up to machine precision as expected.

\begin{figure}[h!]
	\centering
	\subfigure[$k=1$]{\includegraphics[height=60mm]{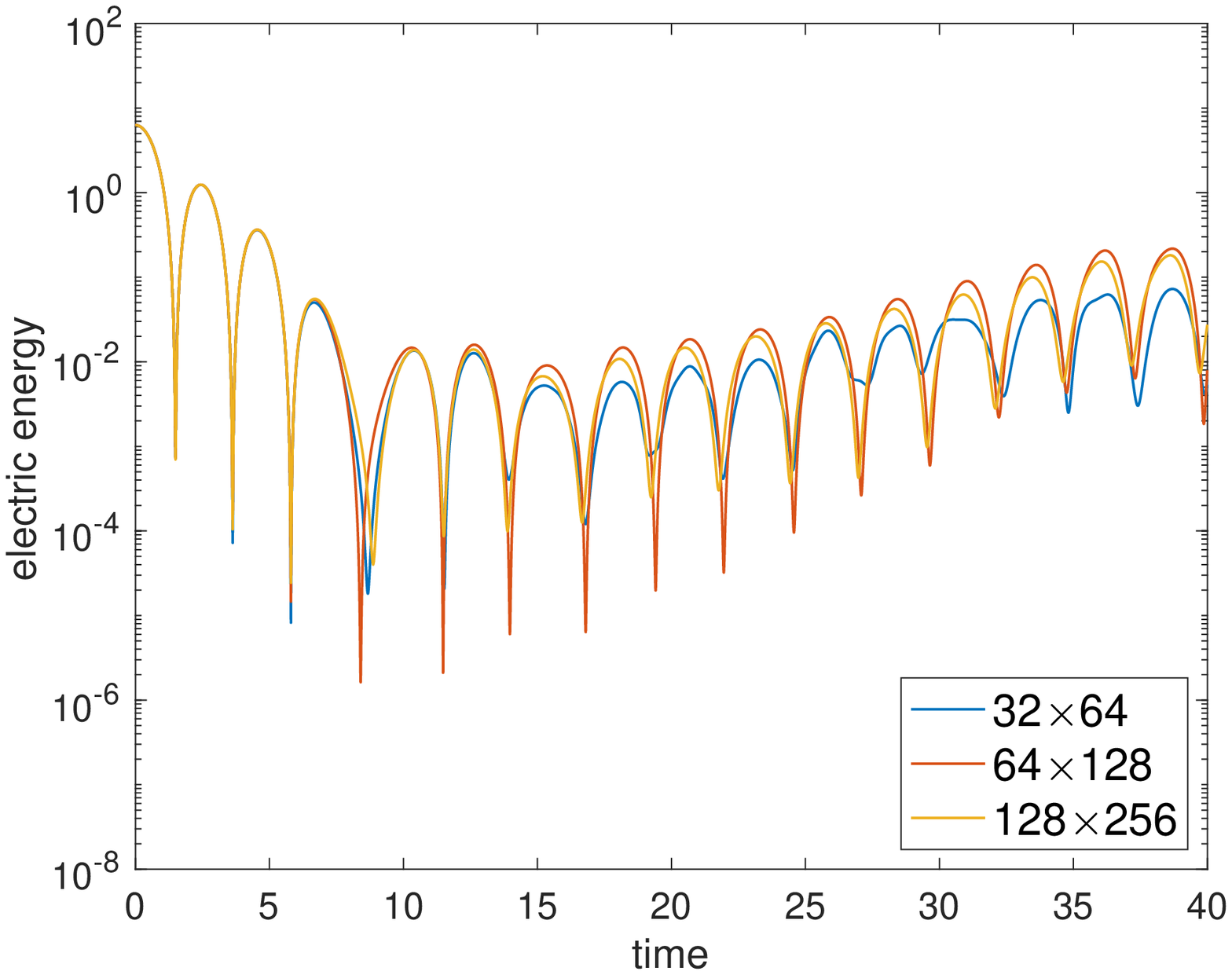}}
	\subfigure[$k=2$]{\includegraphics[height=60mm]{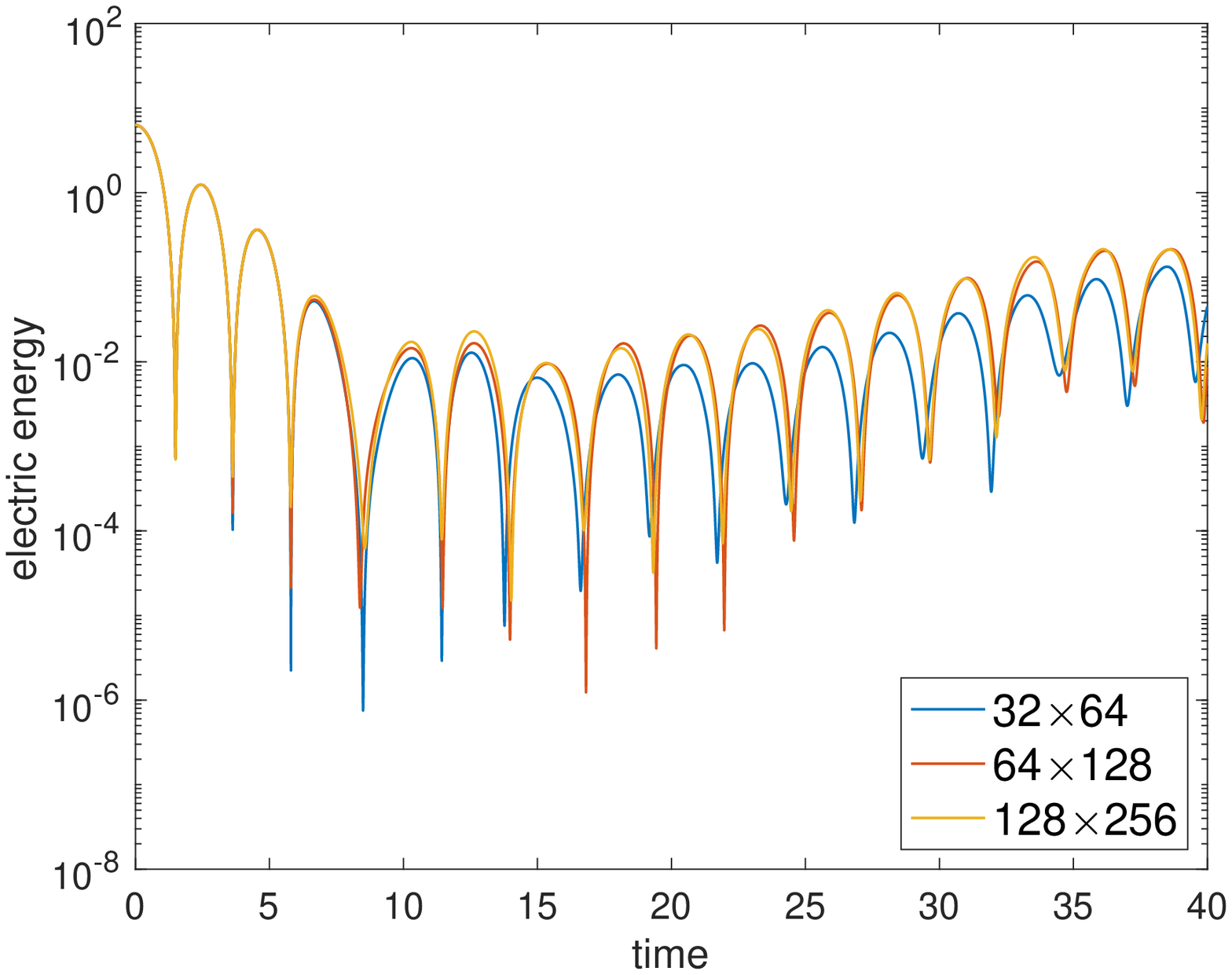}}
	\subfigure[$k=1$]{\includegraphics[height=60mm]{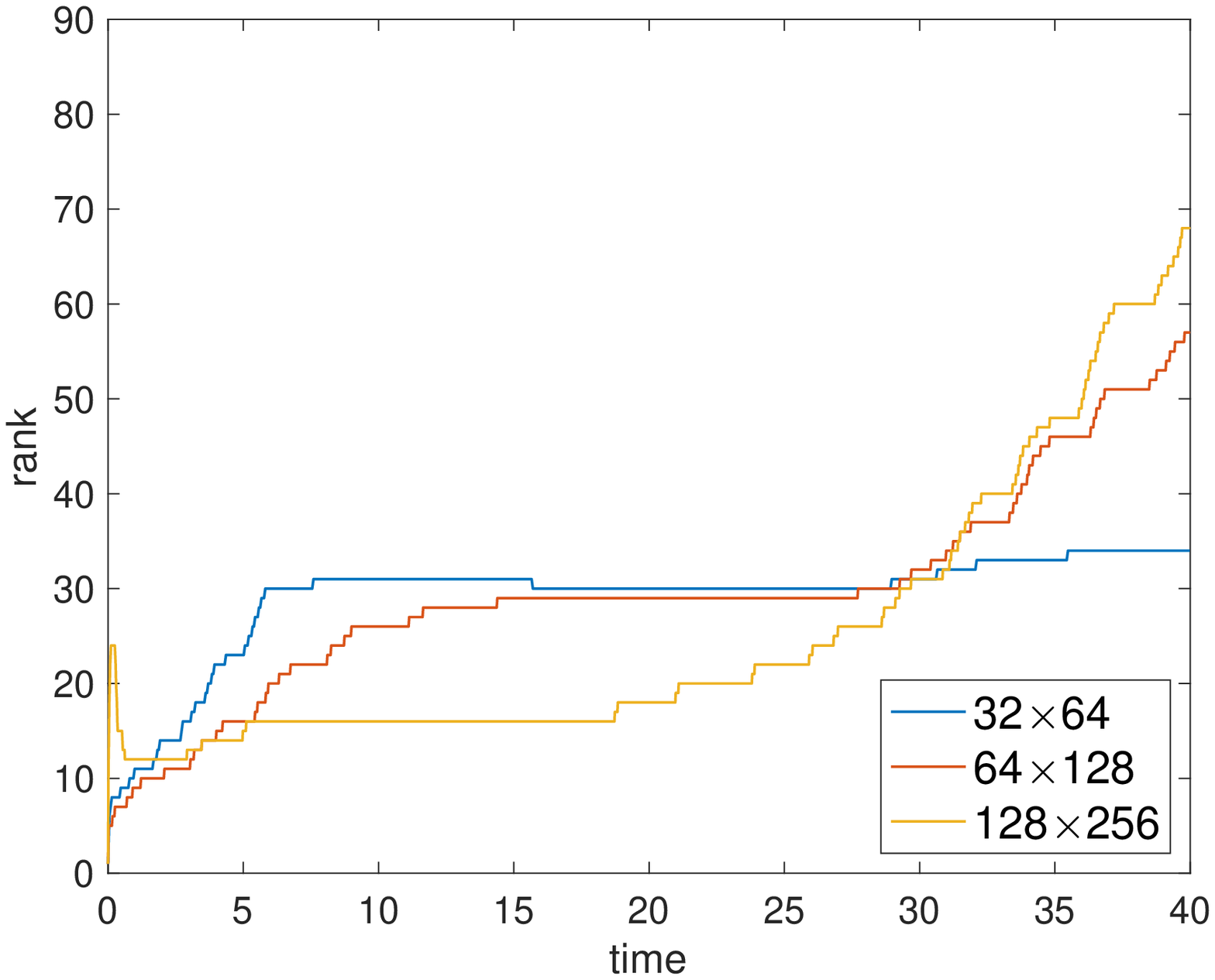}}
	\subfigure[$k=2$]{\includegraphics[height=60mm]{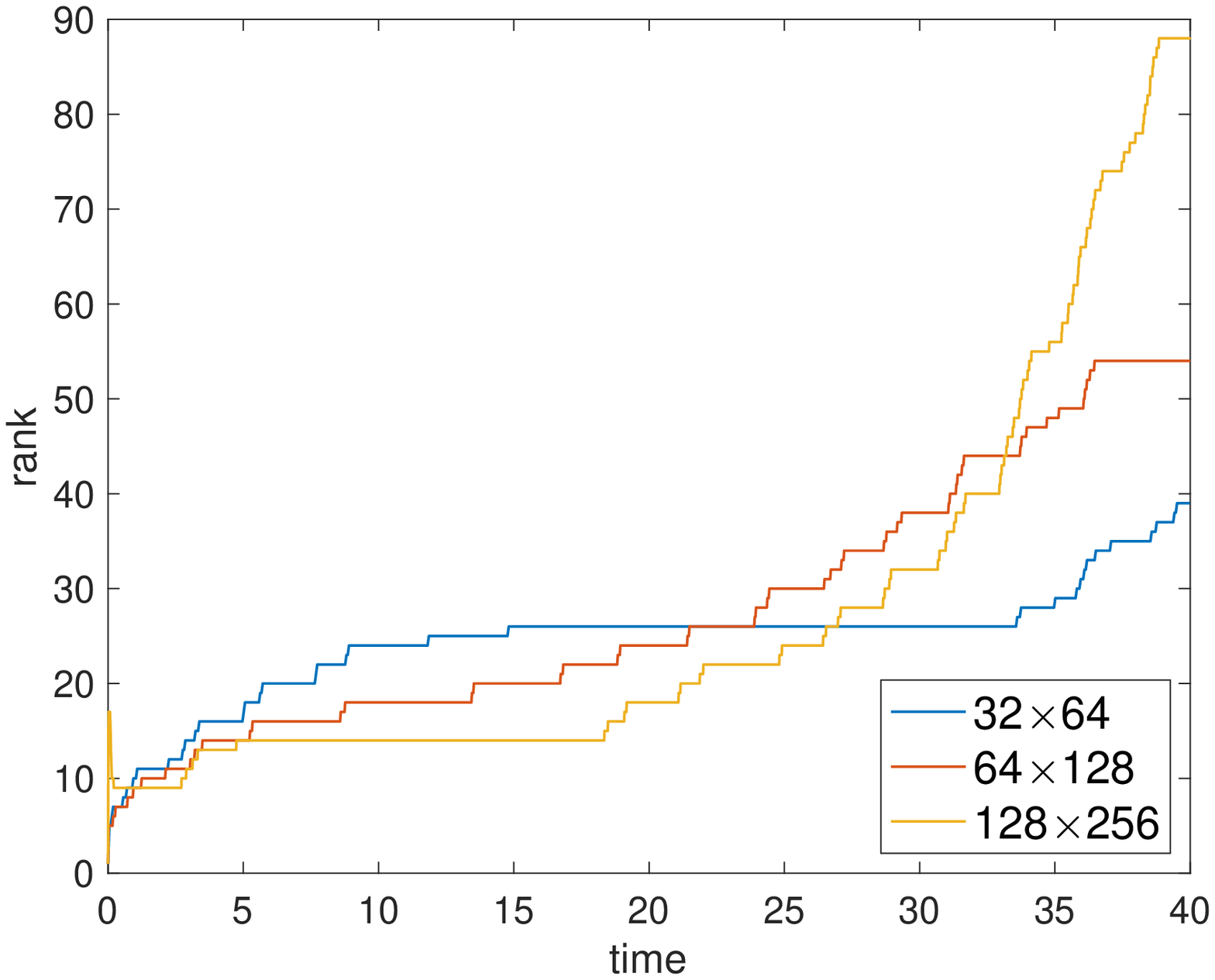}}
	\caption{Example \ref{ex:strong1d}.  The time evolution of  electric energy (a, b) and ranks of the low rank DG solutions (c, b). $\varepsilon=10^{-3}$.}
	\label{fig:strong1d_elec}
\end{figure}	

\begin{figure}[h!]
	\centering
	\subfigure[$k=1$]{\includegraphics[height=40mm]{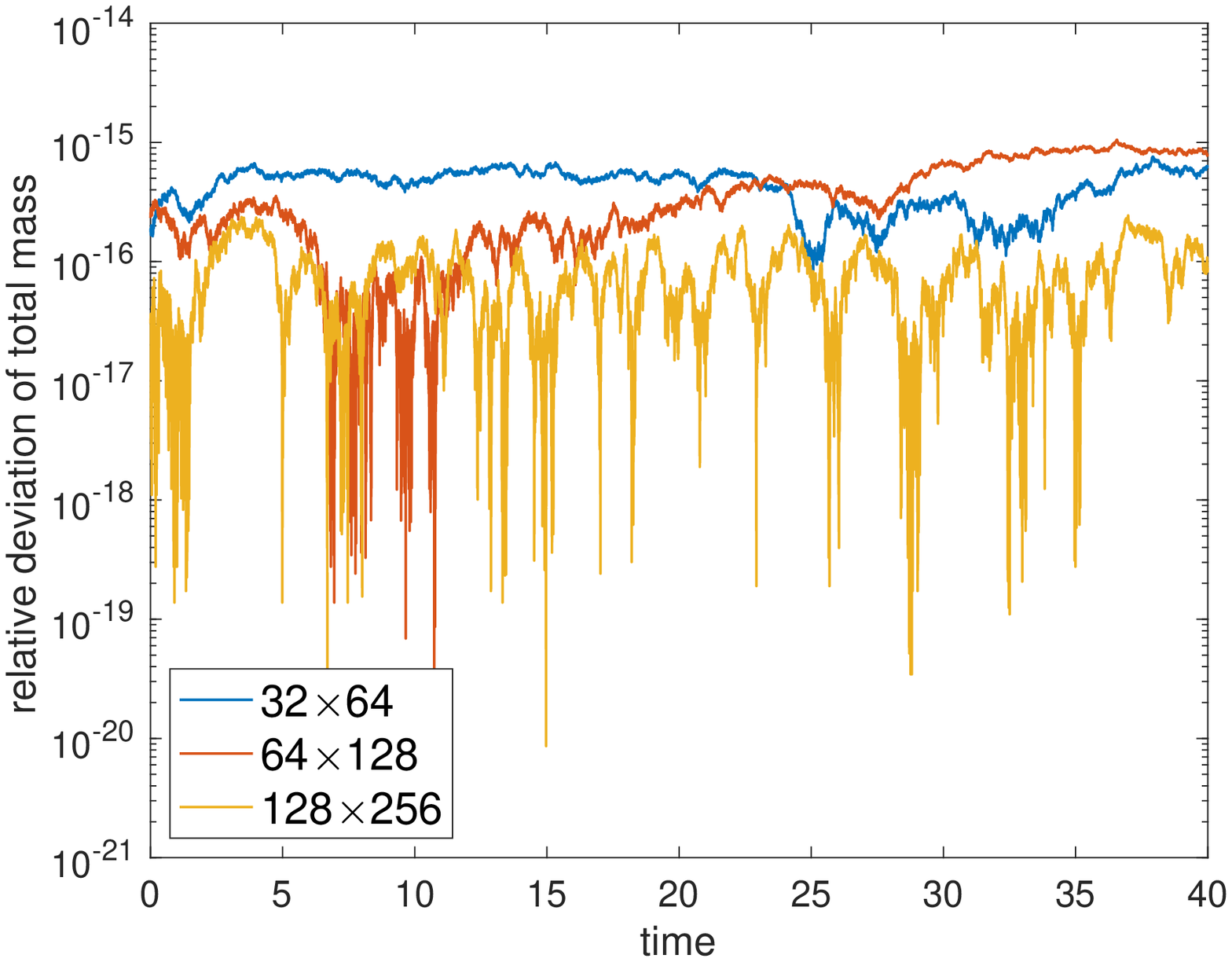}}
	\subfigure[$k=1$]{\includegraphics[height=40mm]{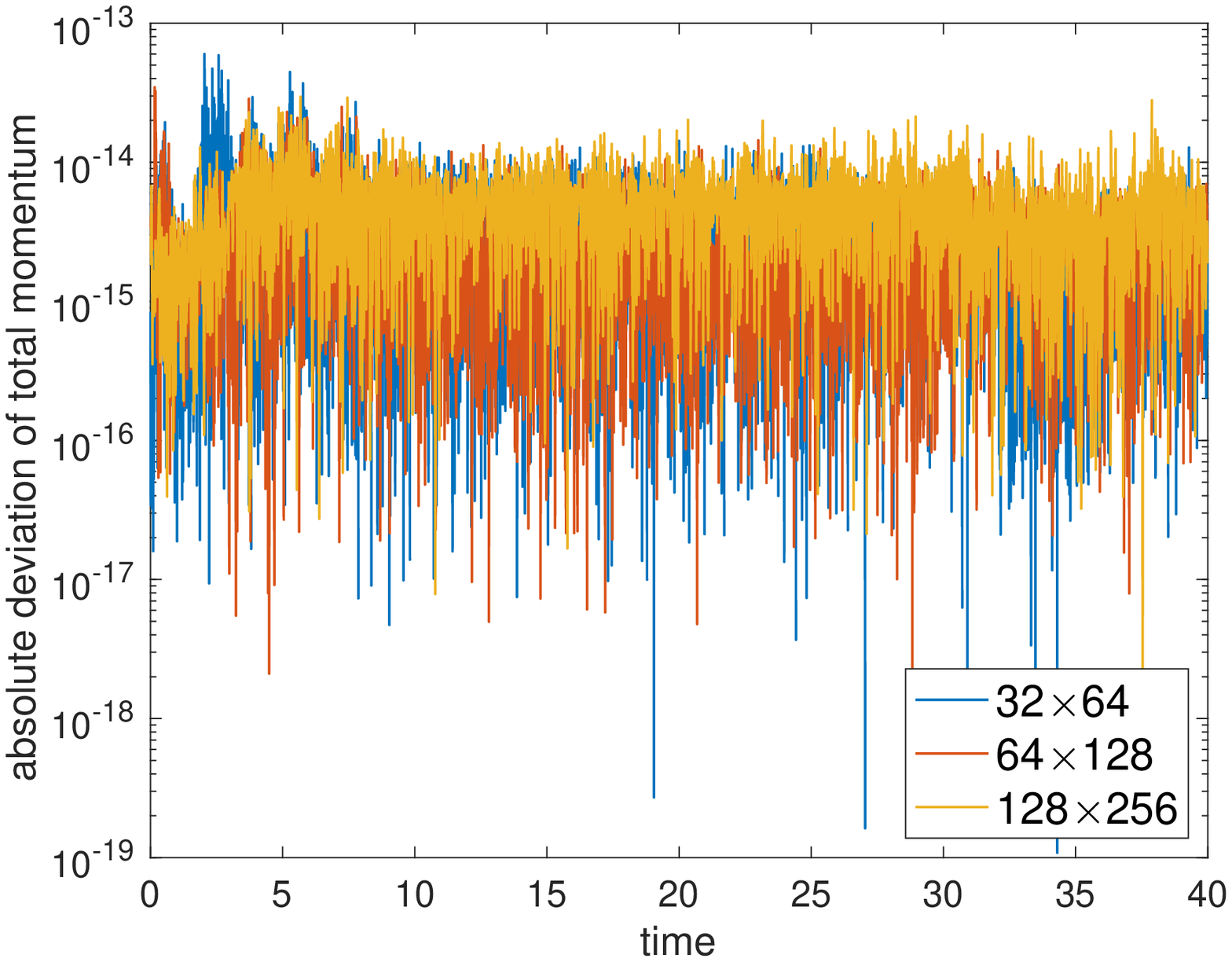}}
	\subfigure[$k=1$]{\includegraphics[height=40mm]{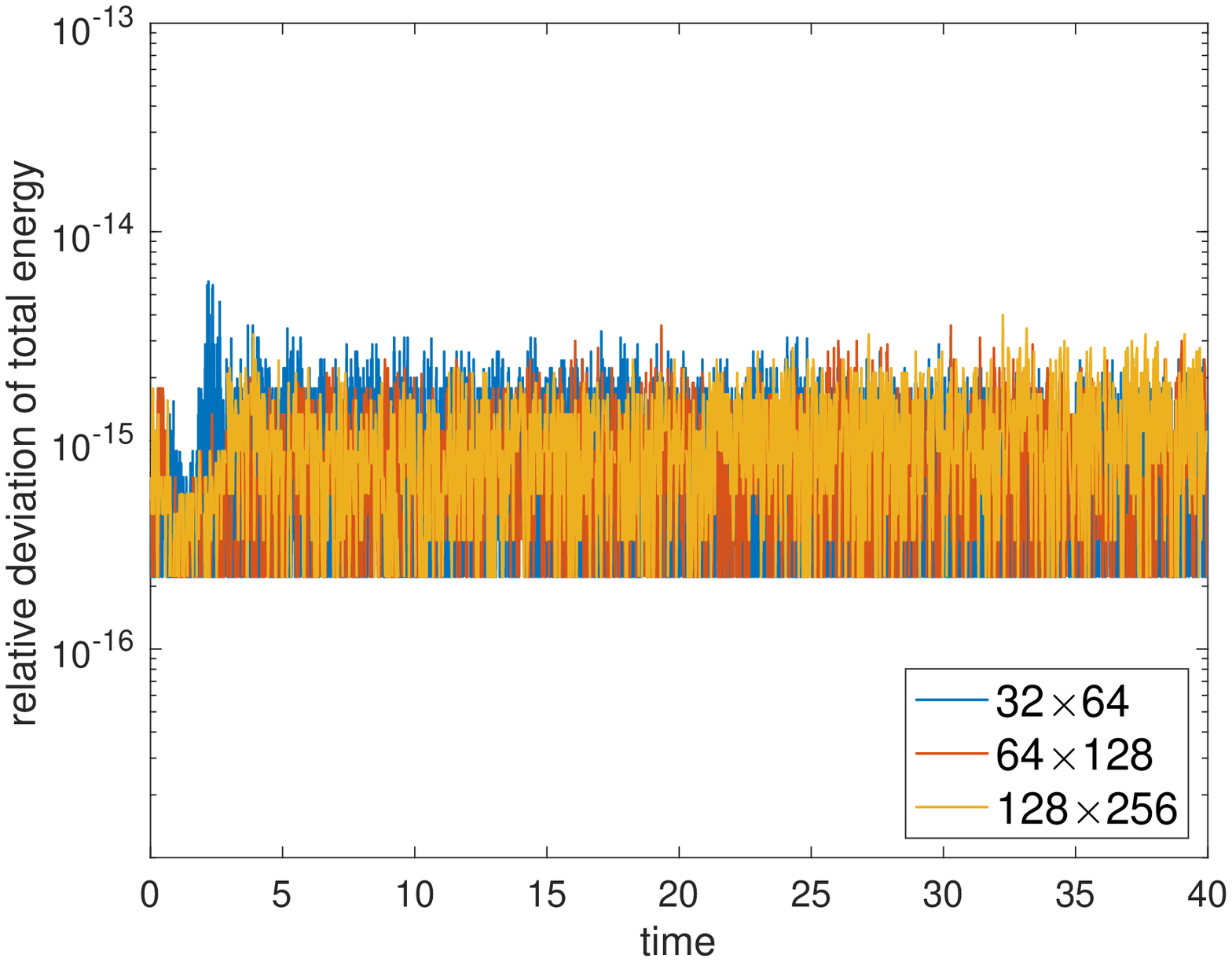}}
	\subfigure[$k=2$]{\includegraphics[height=40mm]{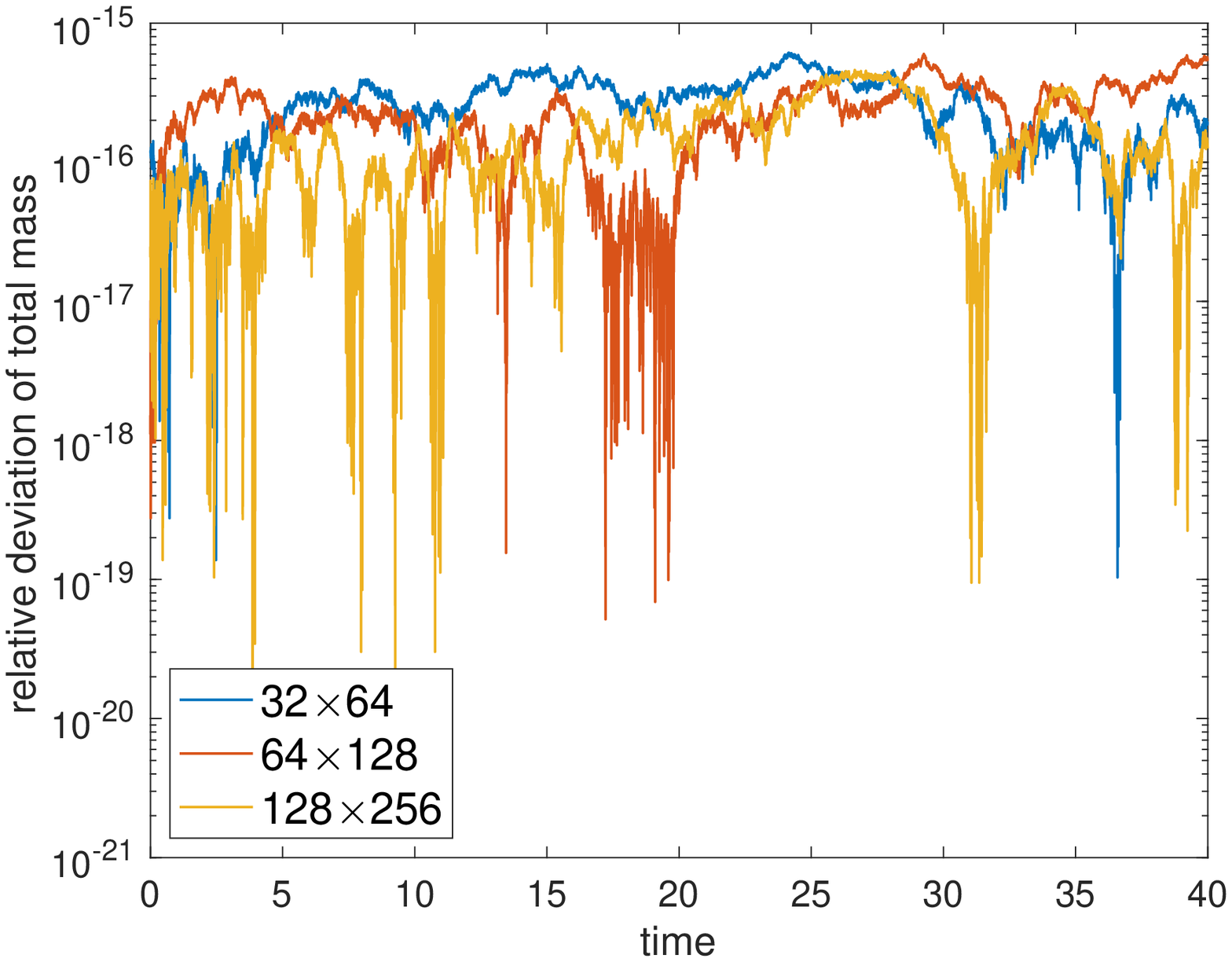}}
	\subfigure[$k=2$]{\includegraphics[height=40mm]{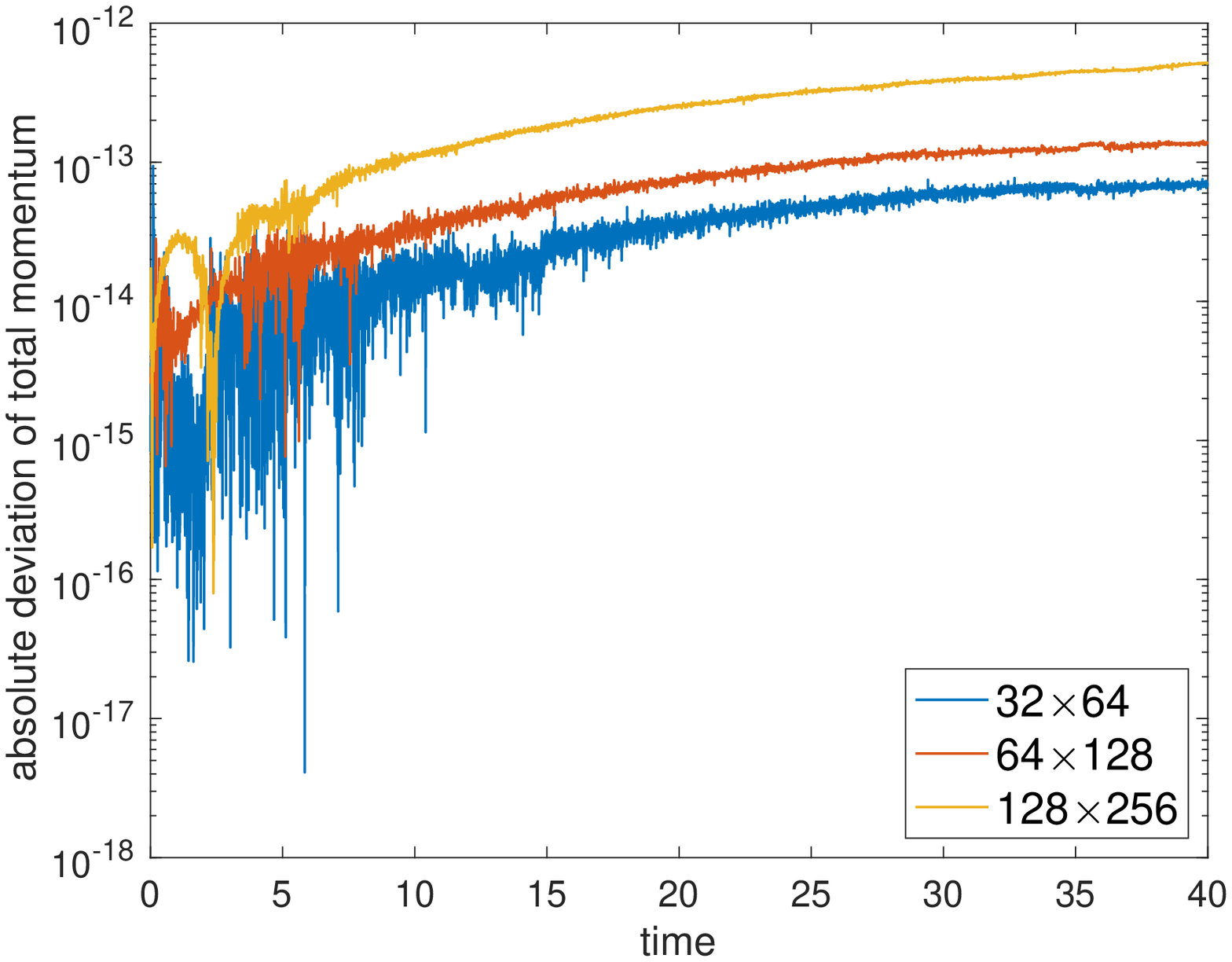}}
	\subfigure[$k=2$]{\includegraphics[height=40mm]{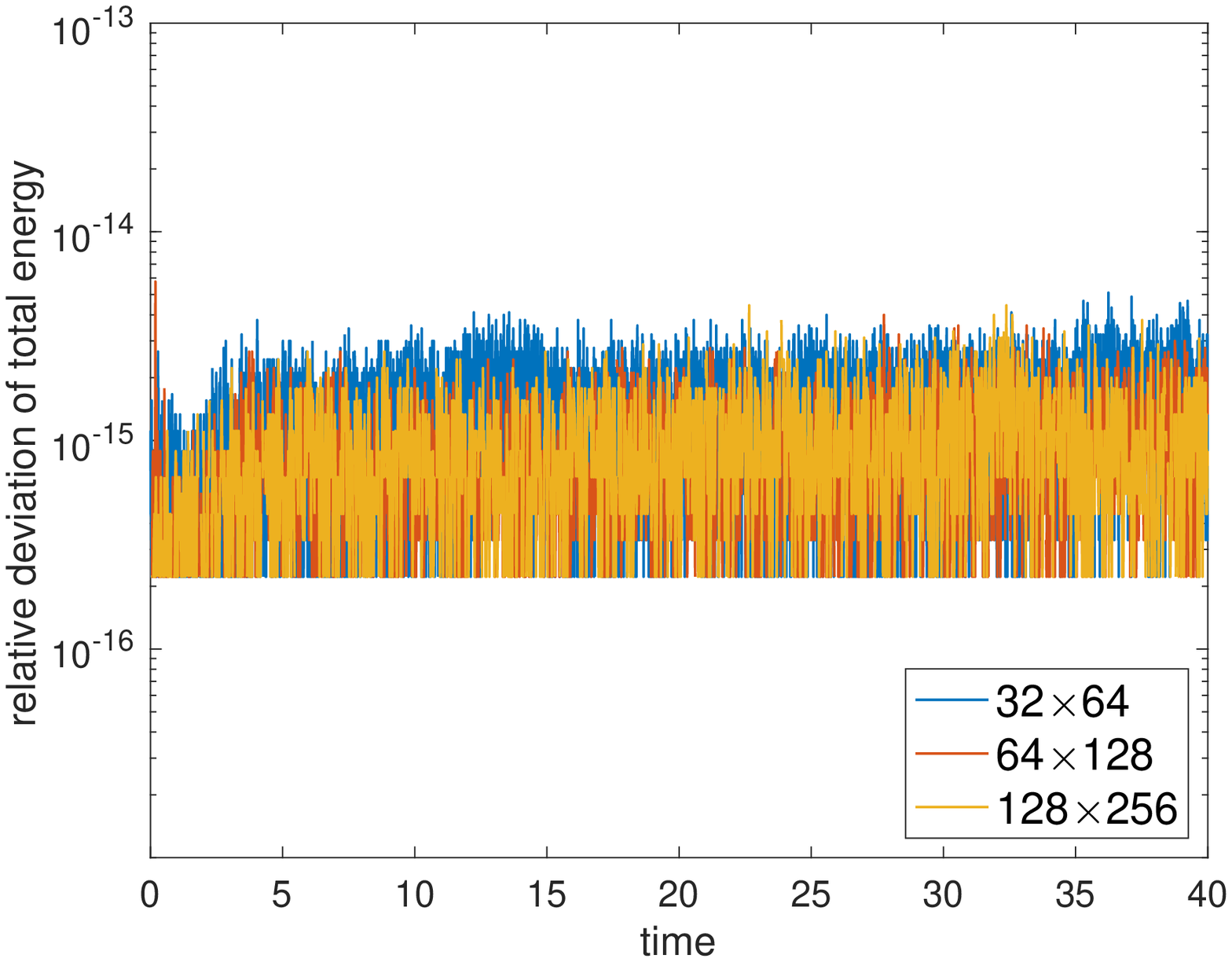}}
	\caption{Example \ref{ex:strong1d}.  The time evolution of  relative deviation of total mass (a, d), absolute deviation of total momentum (b, e), and relative deviation of total energy (c, f).  $\varepsilon=10^{-3}$.}
	\label{fig:strong1d_invar}
\end{figure}
	
\end{exa}

 \begin{exa}\label{ex:bumpontail} (Bump on tail.) As the last 1D1V test, we simulate the bump-on-tail problem with the initial condition
 \begin{equation}
\label{eq:bump1d}
f(x,v,t=0) = \left(1+\alpha  \cos \left(k x\right)\right)\left(n_p\exp\left(-\frac{v^2}{2}\right) +n_b\exp\left(-\frac{(v-u)^2}{2v_{t}}\right) \right),
\end{equation}
where $\alpha=0.04$, $k=0.3$, $n_{p}=\frac{9}{10 \sqrt{2 \pi}}$, $n_{b}=\frac{2}{10 \sqrt{2 \pi}}$, $u=4.5$, $v_{t}=0.5$. The weight function $w(v) = \exp(-\frac{v^2}{7})$ is chosen. The domain is set to be $[0,L_x]\times[-L_v,L_v]$ with $L_x=2\pi/k$ and $L_v=13$, and the truncation threshold is chosen as $\varepsilon=10^{-5}$.  We simulate the problem up to $t=30$ and plot the contours of the low rank DG solutions with a set of non-uniform meshes obtained by perturbing uniform meshes by 10\%. The results are consistent with those reported in the literature, and a method with larger $k$ and over a finer mesh can provide better resolution as expected. In Figure \ref{fig:bump1d_elec}-\ref{fig:bump1d_invar}, we report the time histories of electric energy, numerical ranks, together with relative derivation of total mass, total momentum and total energy. The observation is similar to the strong Landau damping test that the filamentation structures are well captured by the proposed method with rank adaptivity, and the physical invariants are conserved up to the machine precision.


 \begin{figure}[h!]
	\centering
       \subfigure[$k=1$, $N_x\times N_v=16\times32$]{\includegraphics[height=40mm]{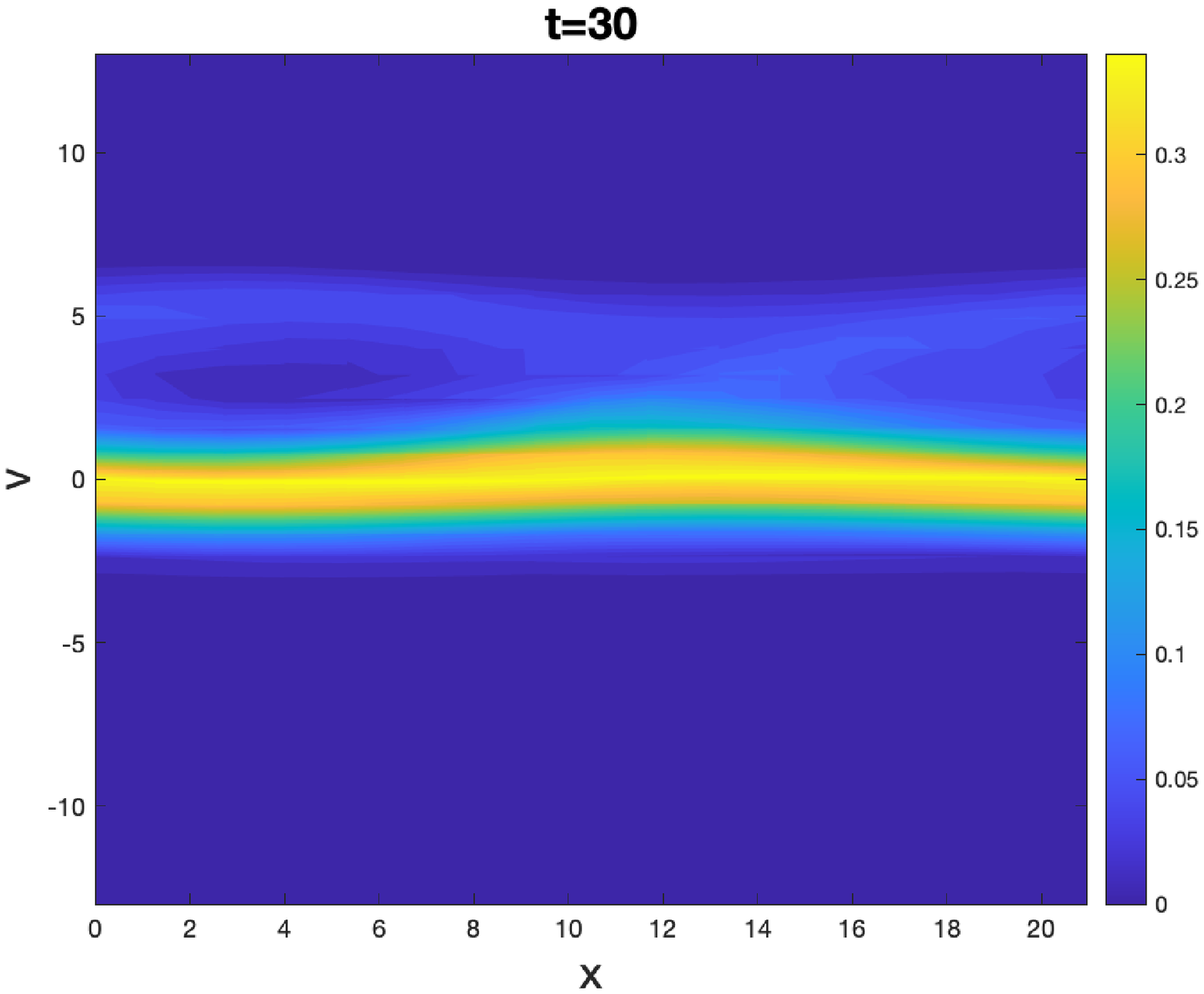}}
	\subfigure[$k=1$, $N_x\times N_v=32\times64$]{\includegraphics[height=40mm]{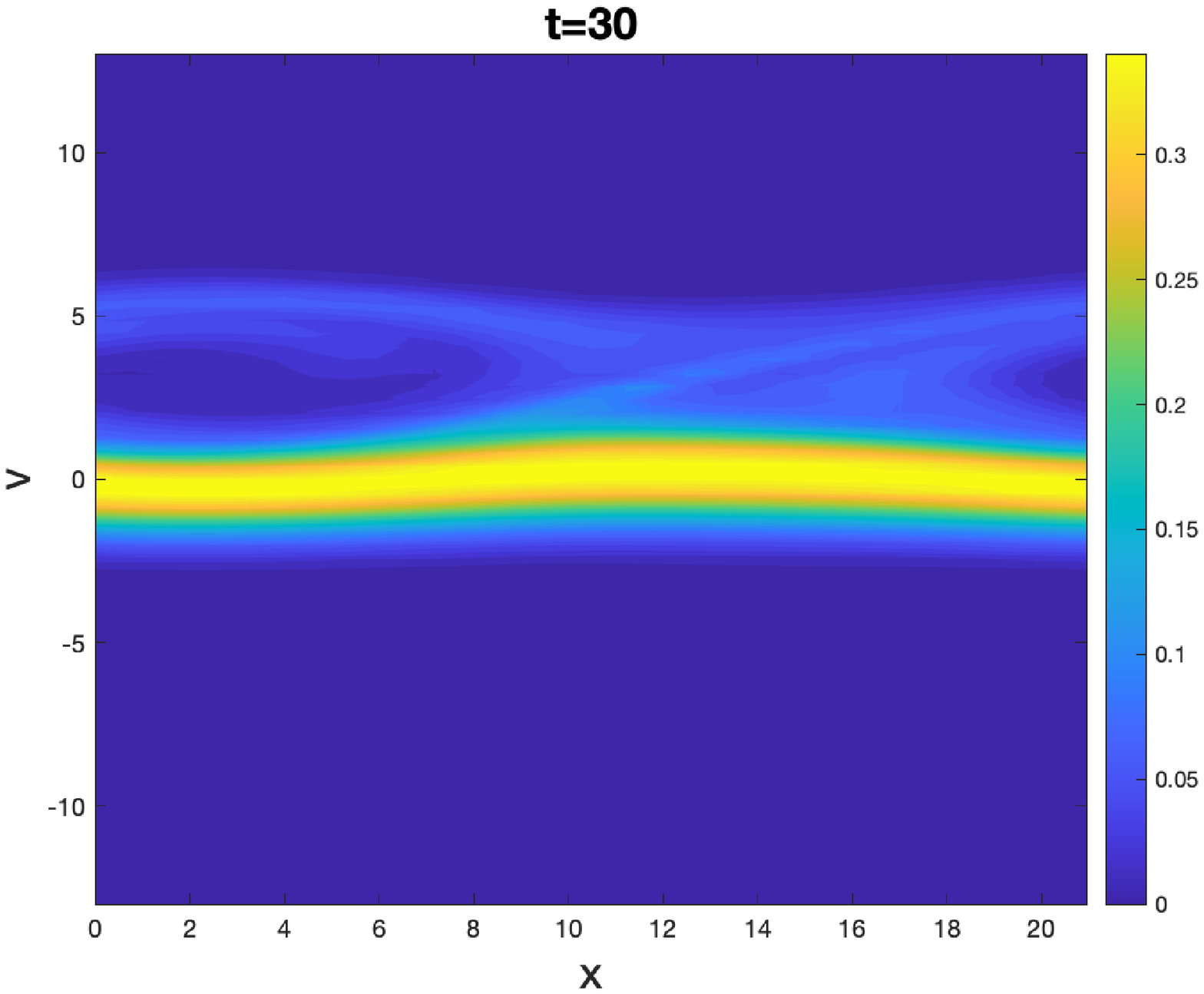}}
	\subfigure[$k=1$, $N_x\times N_v=64\times128$]{\includegraphics[height=40mm]{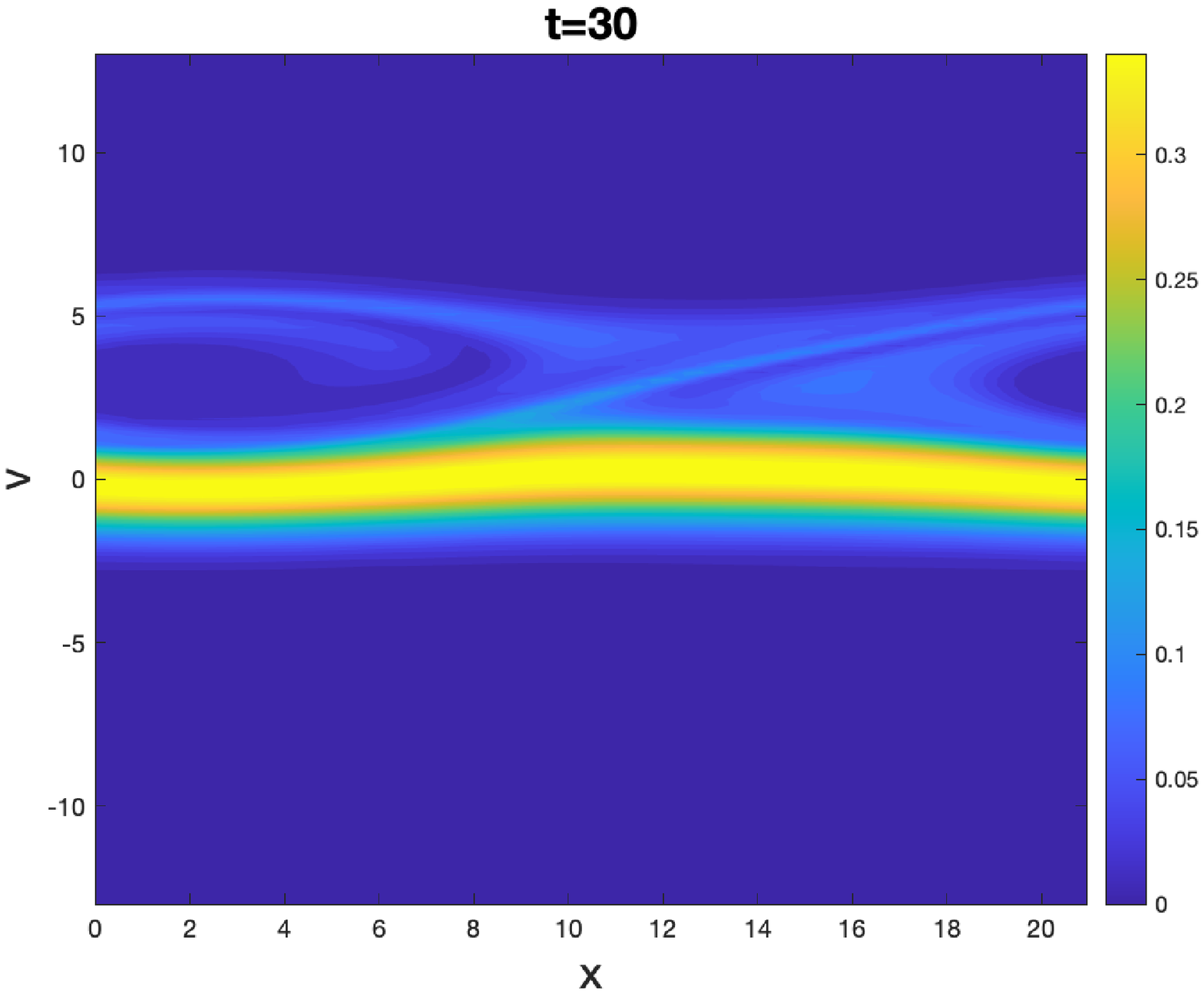}}
       \subfigure[$k=2$, $N_x\times N_v=16\times32$]{\includegraphics[height=40mm]{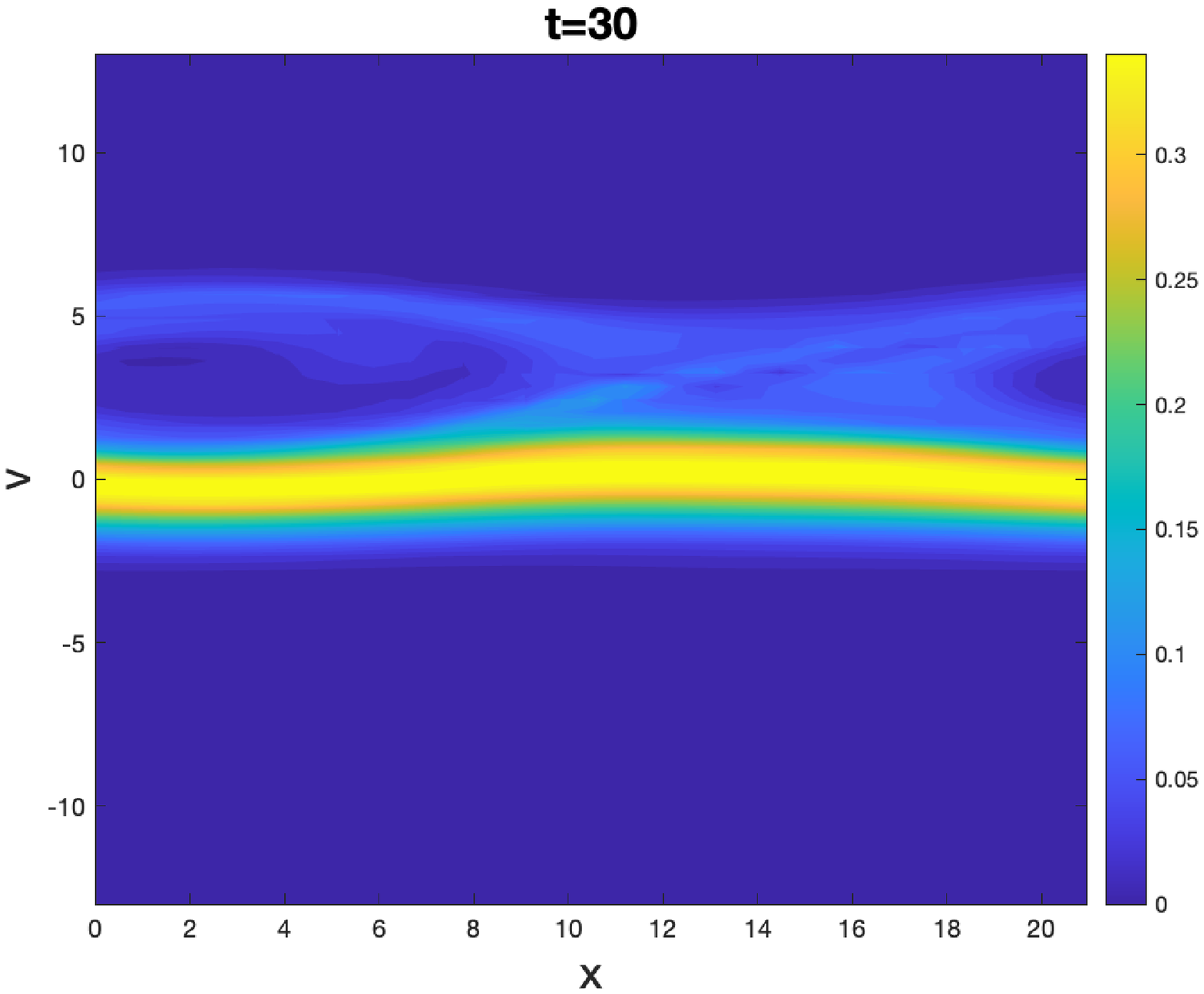}}
	\subfigure[$k=2$, $N_x\times N_v=32\times64$]{\includegraphics[height=40mm]{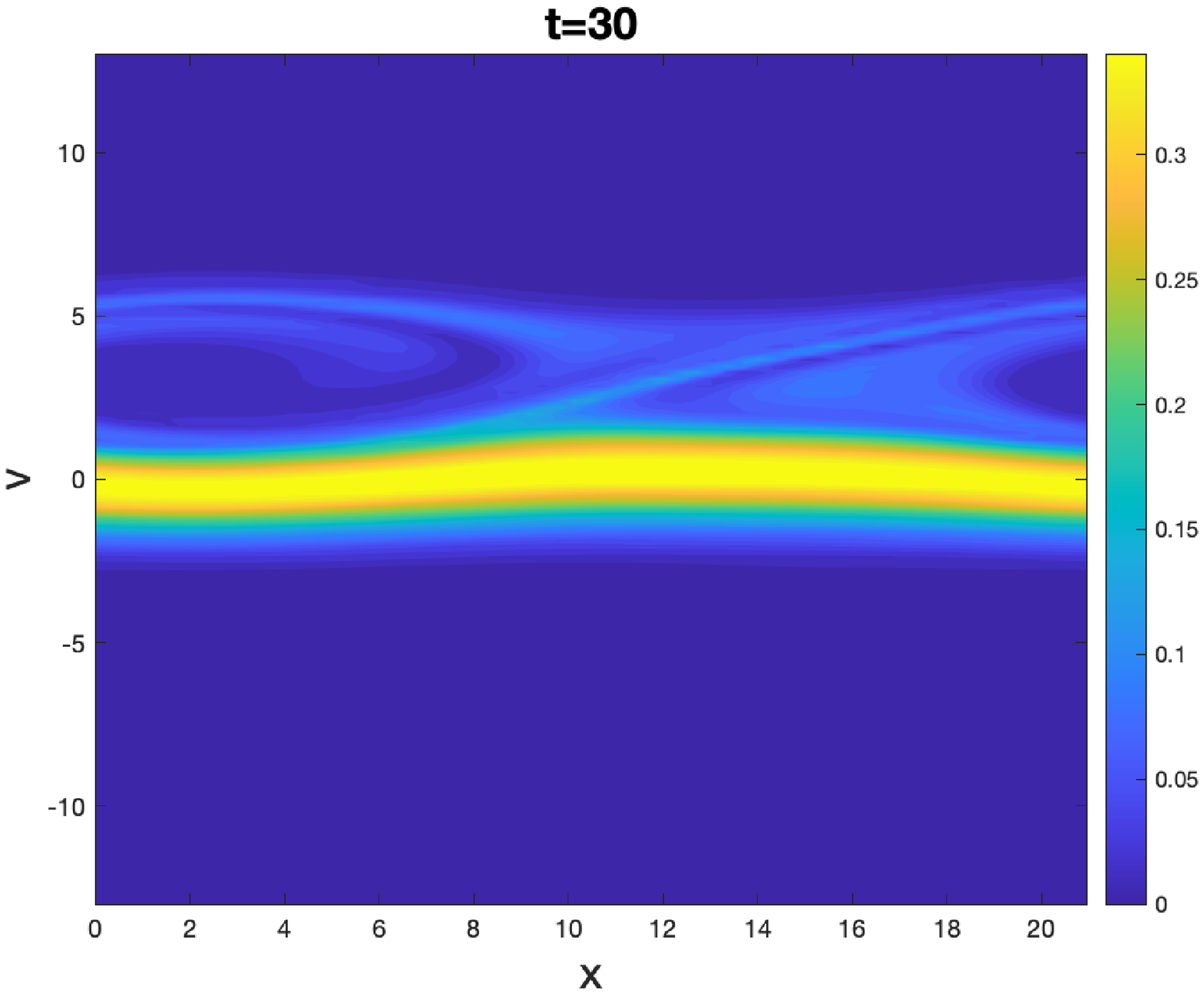}}
	\subfigure[$k=2$, $N_x\times N_v=64\times128$]{\includegraphics[height=40mm]{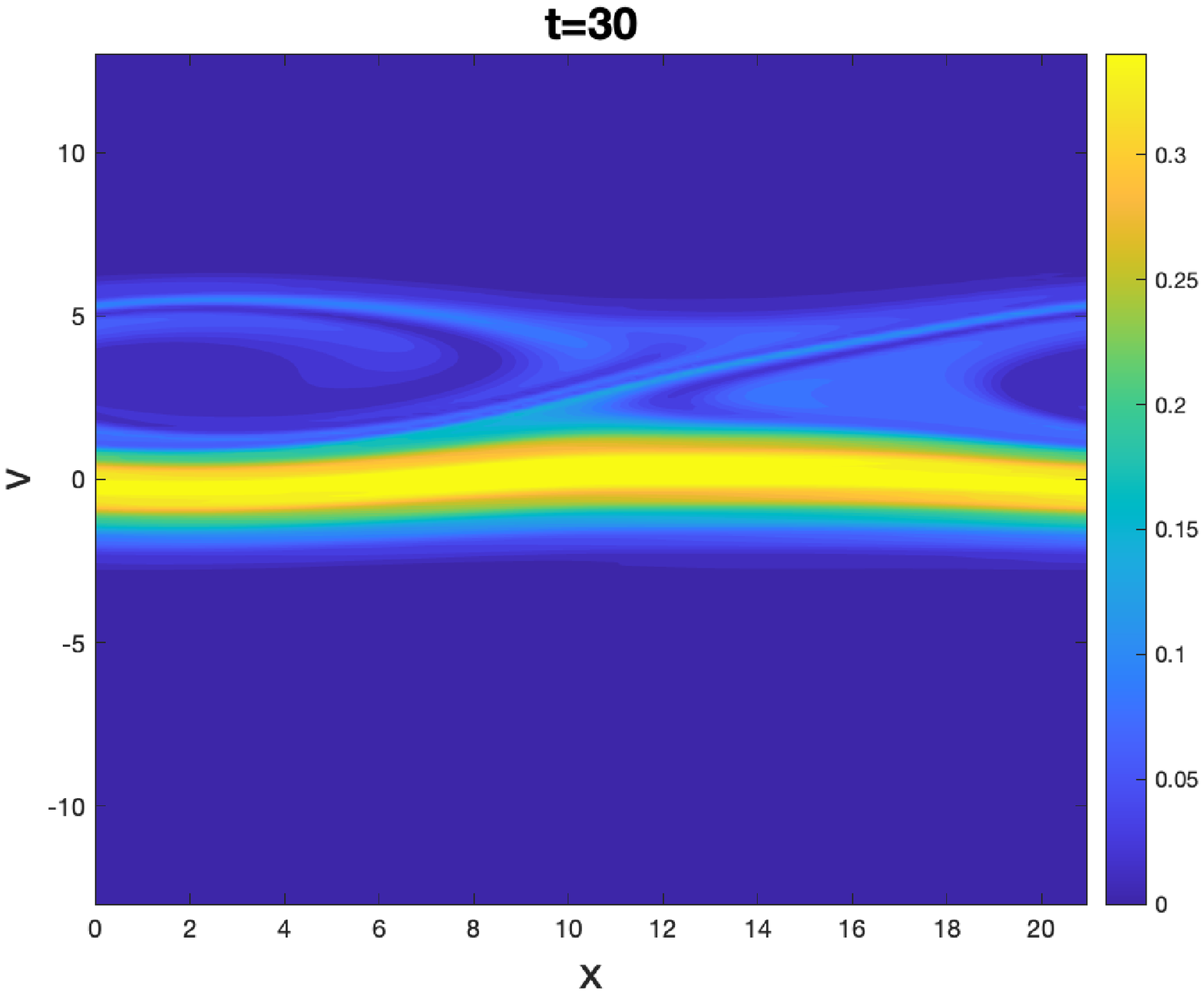}}
			\caption{Example \ref{ex:bumpontail}. Contour plots of the low rank DG solutions at $t=30$. $\varepsilon=10^{-5}$.}
	\label{fig:bumpontail_contour}
\end{figure}

\begin{figure}[h!]
	\centering
	\subfigure[$k=1$]{\includegraphics[height=60mm]{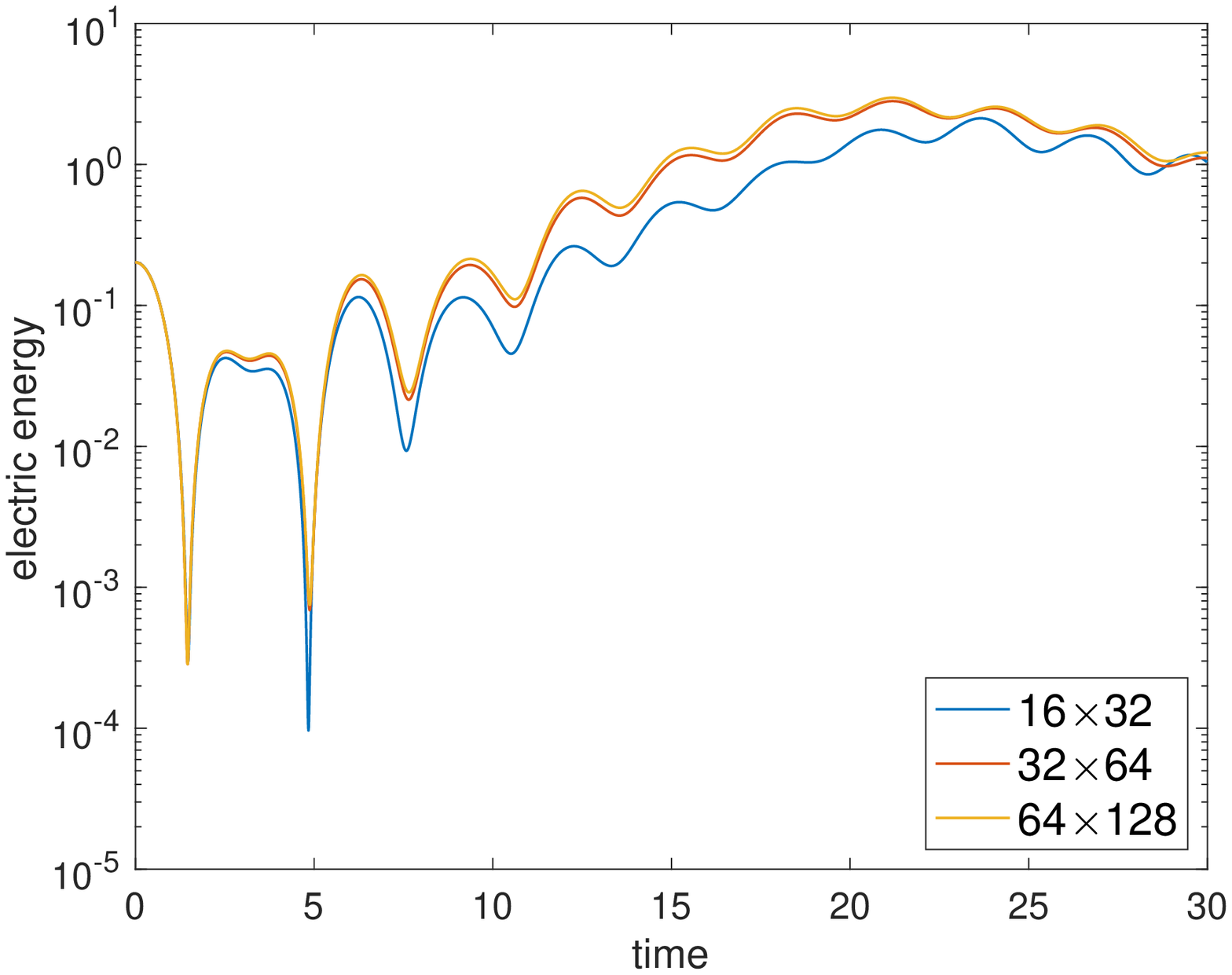}}
	\subfigure[$k=2$]{\includegraphics[height=60mm]{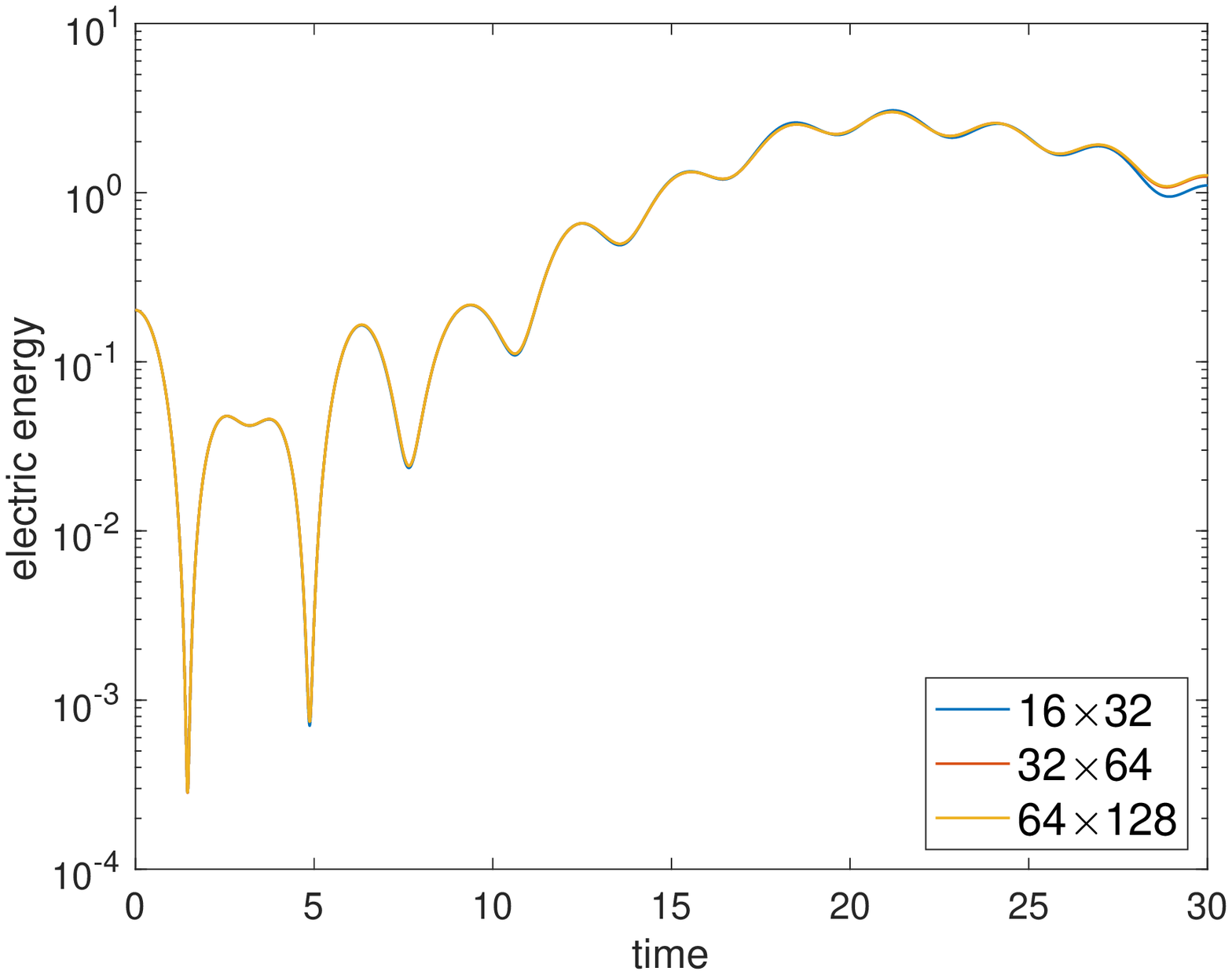}}
	\subfigure[$k=1$]{\includegraphics[height=60mm]{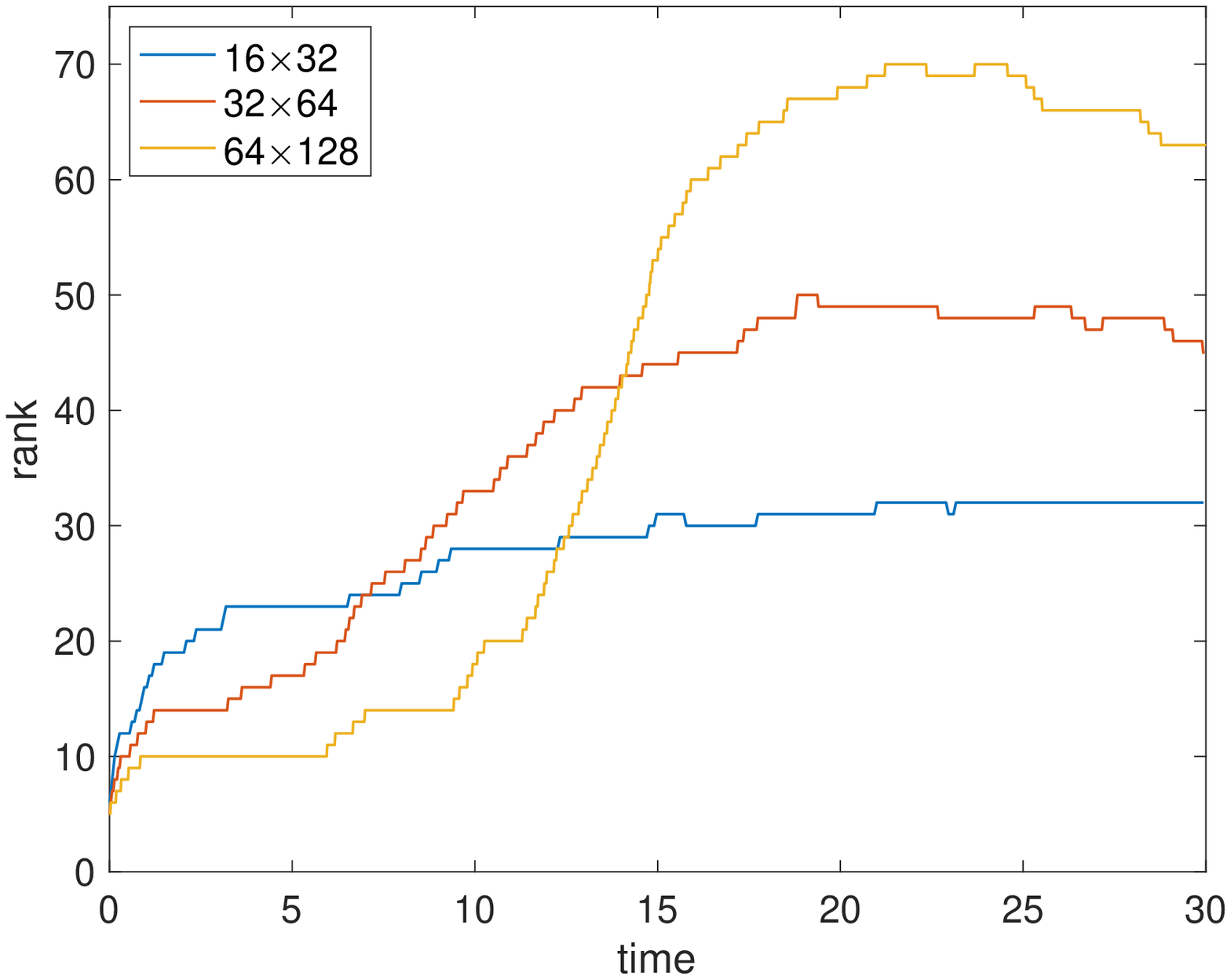}}
	\subfigure[$k=2$]{\includegraphics[height=60mm]{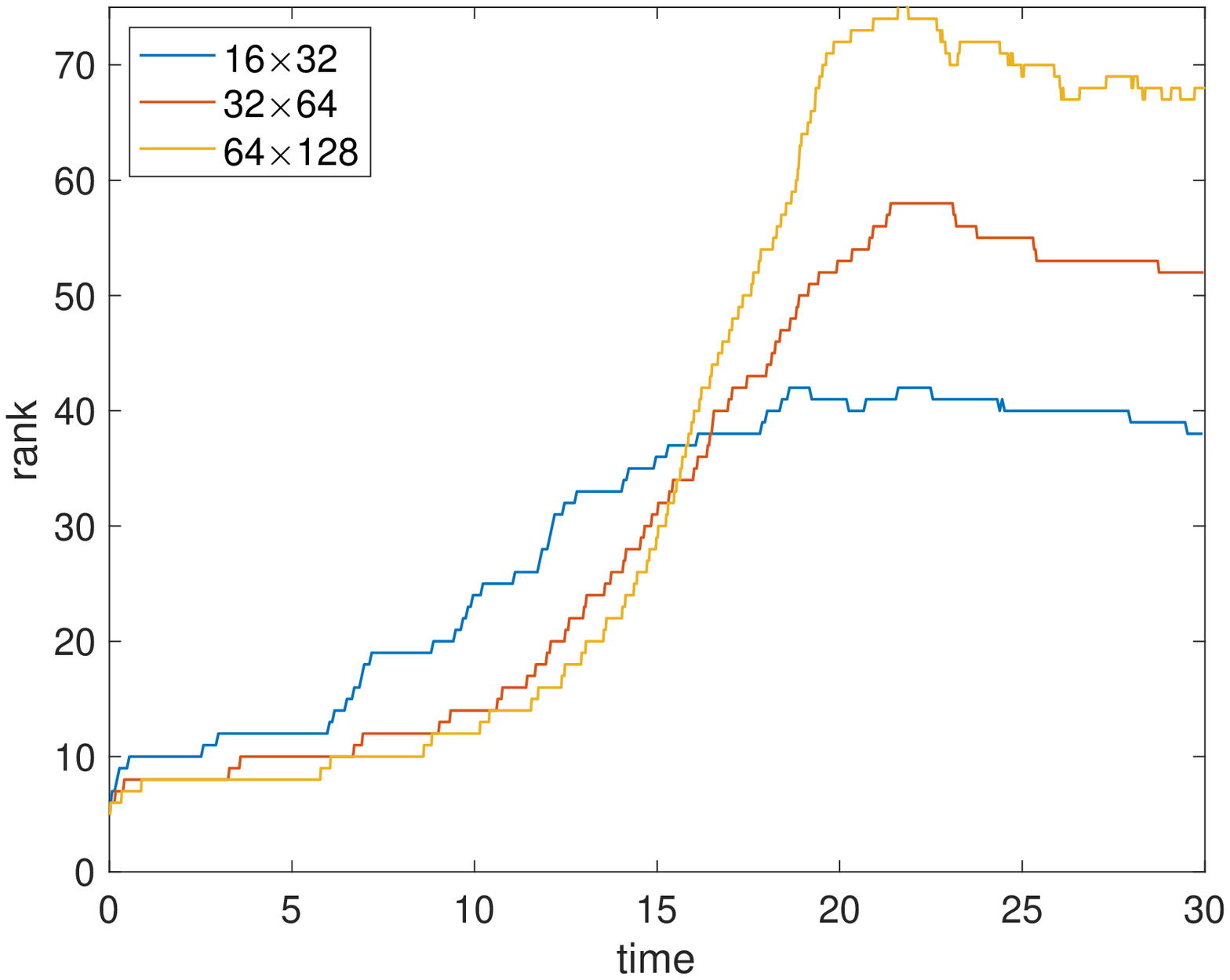}}
	\caption{Example \ref{ex:bumpontail}.  The time evolution of  electric energy (a, b) and ranks of the low rank DG solutions (c, b). $\varepsilon=10^{-5}$.}
	\label{fig:bump1d_elec}
\end{figure}	

\begin{figure}[h!]
	\centering
	\subfigure[$k=1$]{\includegraphics[height=40mm]{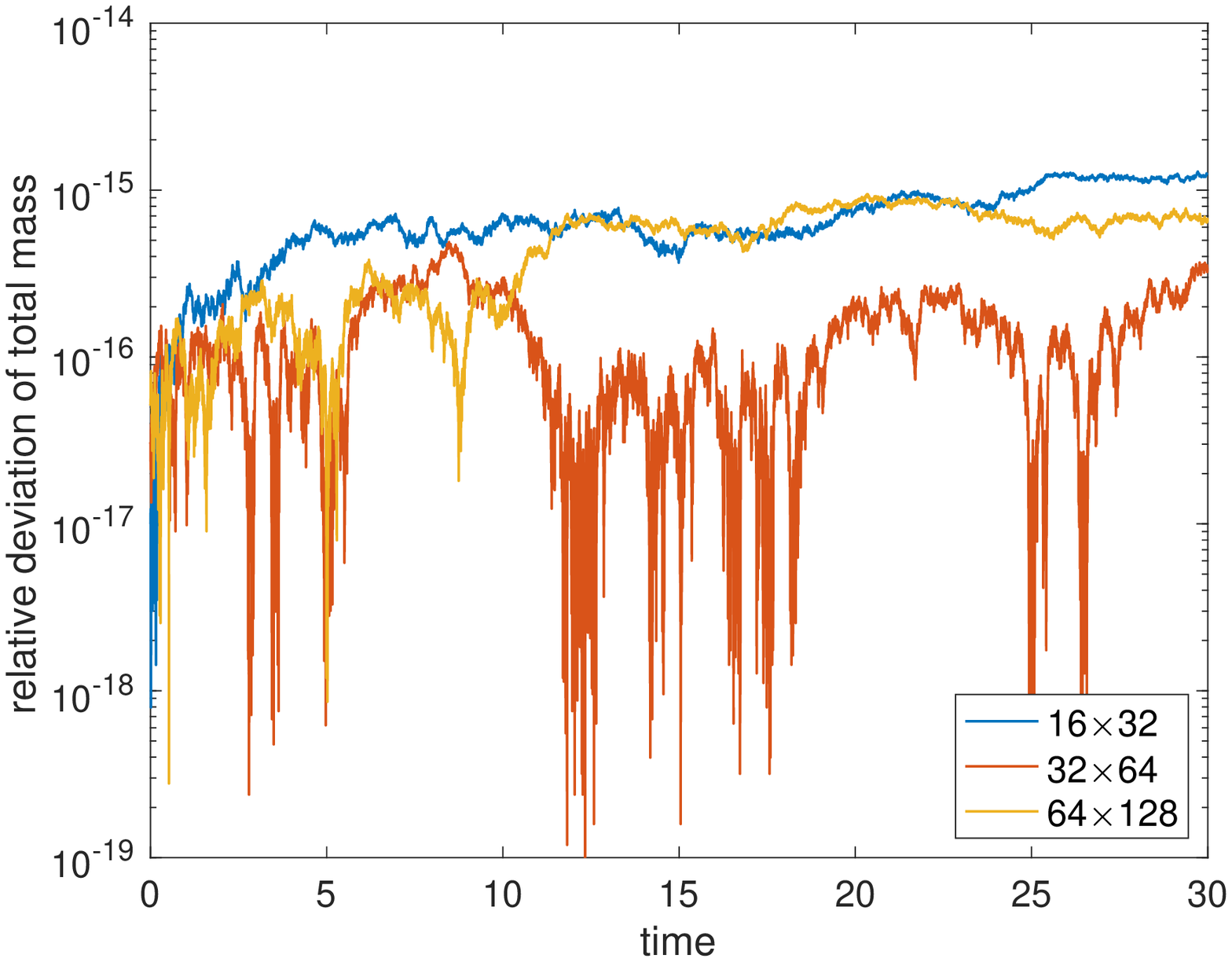}}
	\subfigure[$k=1$]{\includegraphics[height=40mm]{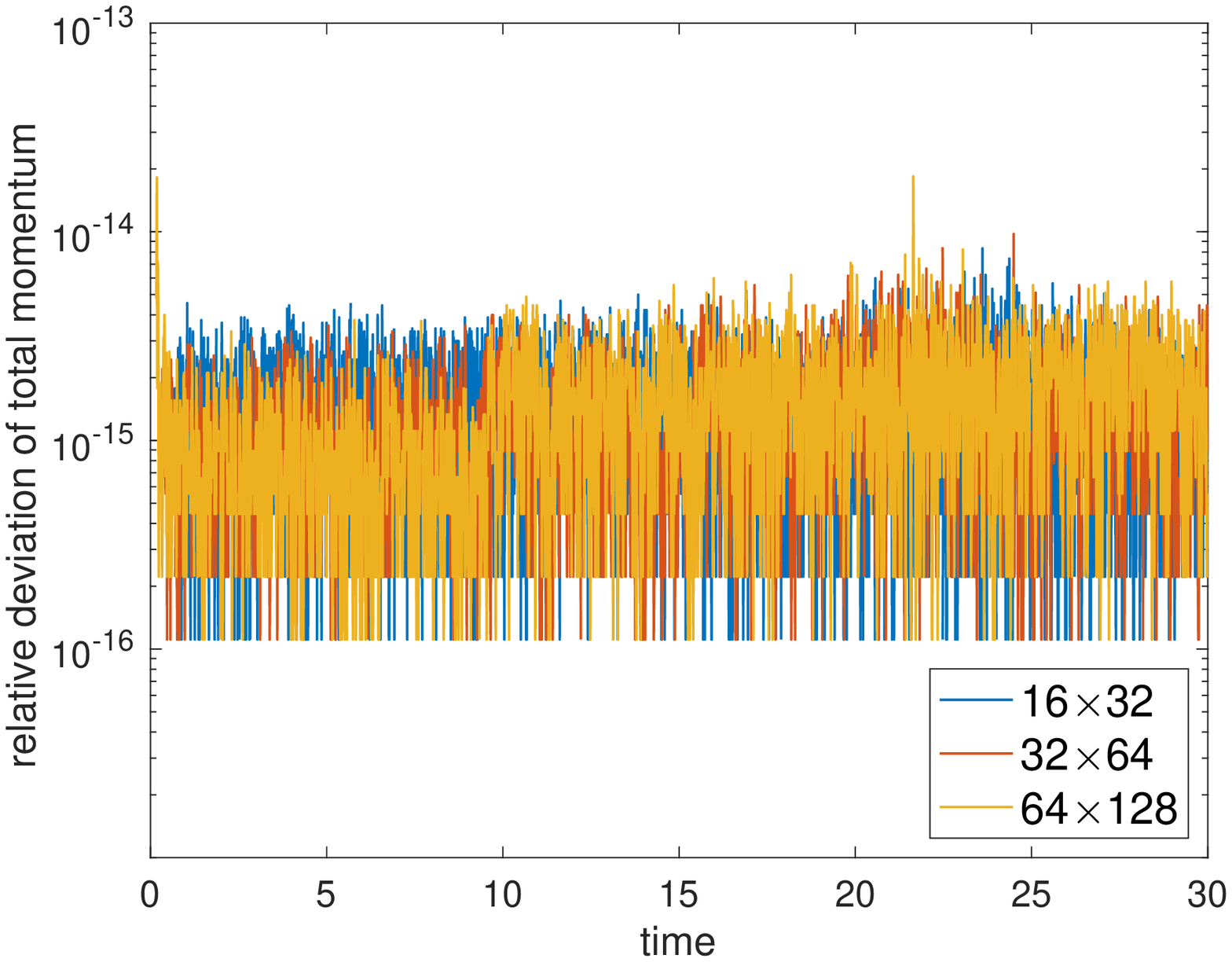}}
	\subfigure[$k=1$]{\includegraphics[height=40mm]{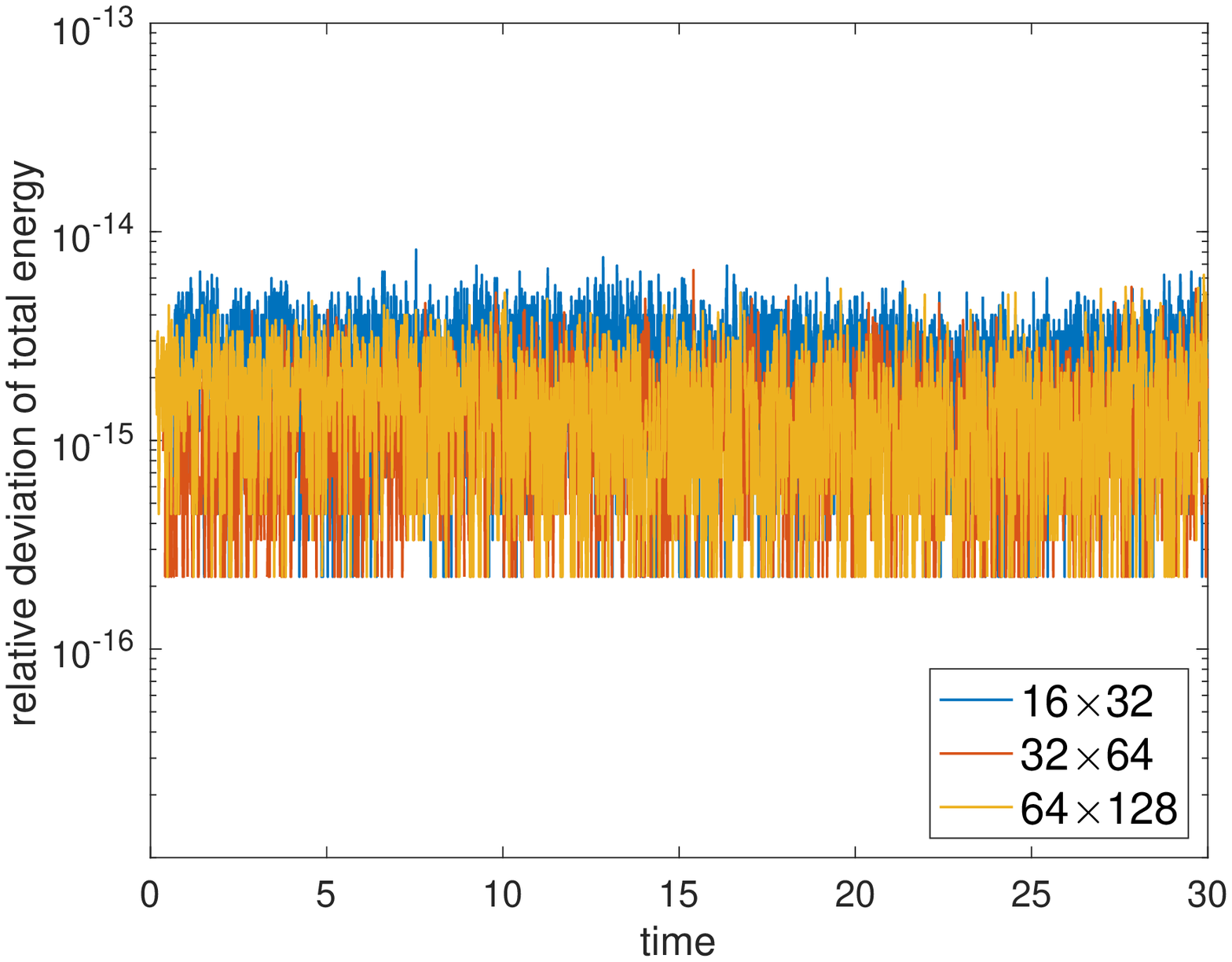}}
	\subfigure[$k=2$]{\includegraphics[height=40mm]{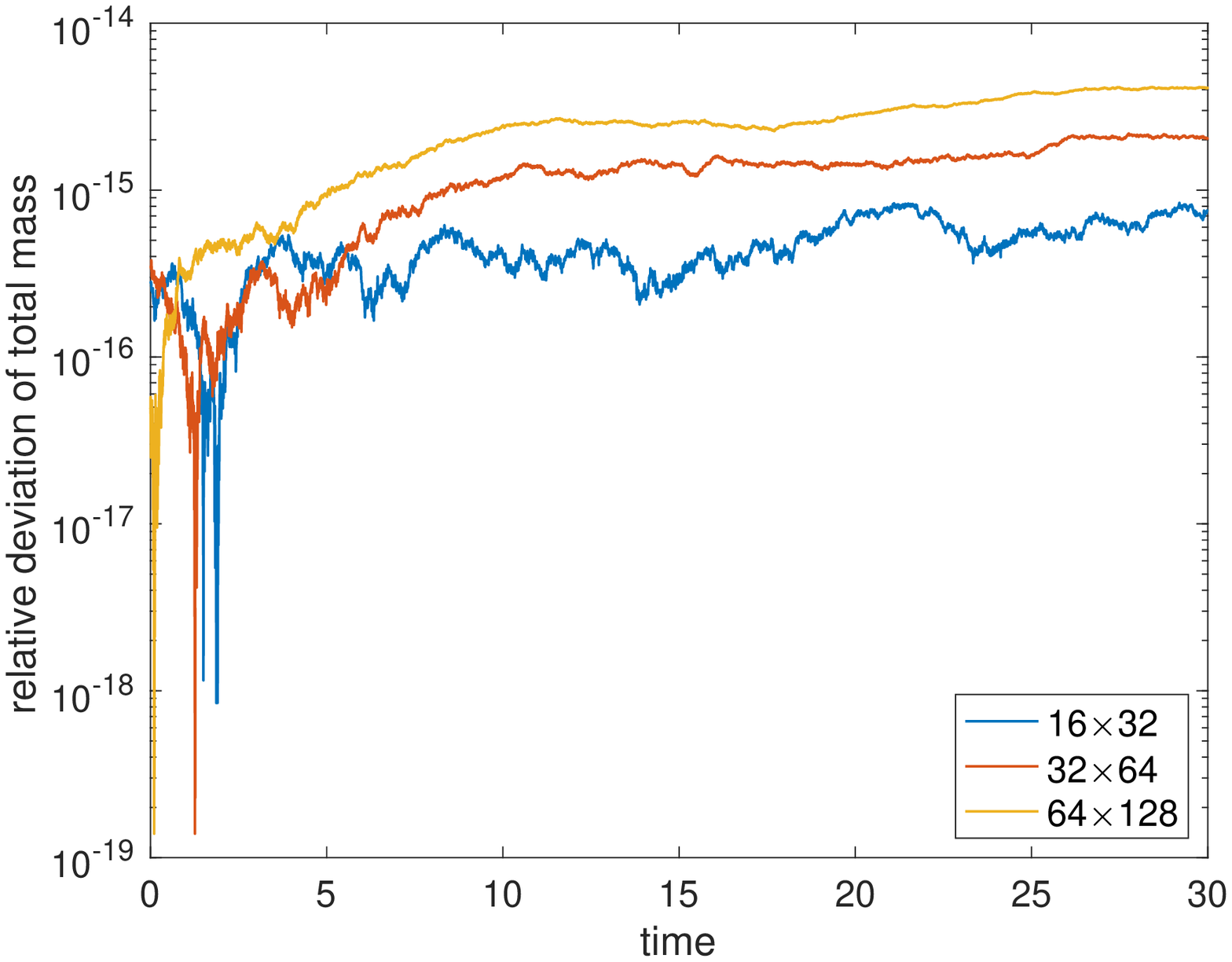}}
	\subfigure[$k=2$]{\includegraphics[height=40mm]{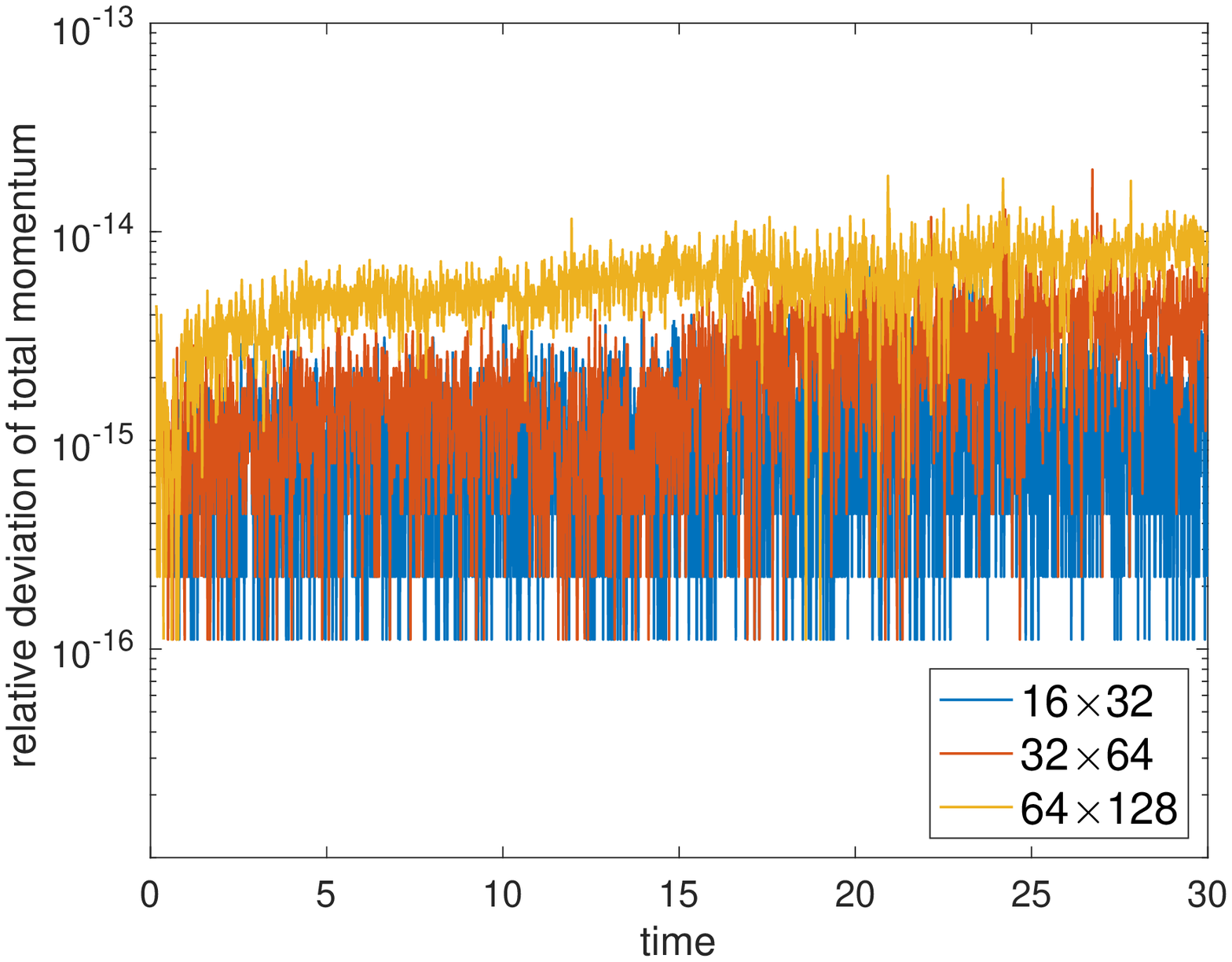}}
	\subfigure[$k=2$]{\includegraphics[height=40mm]{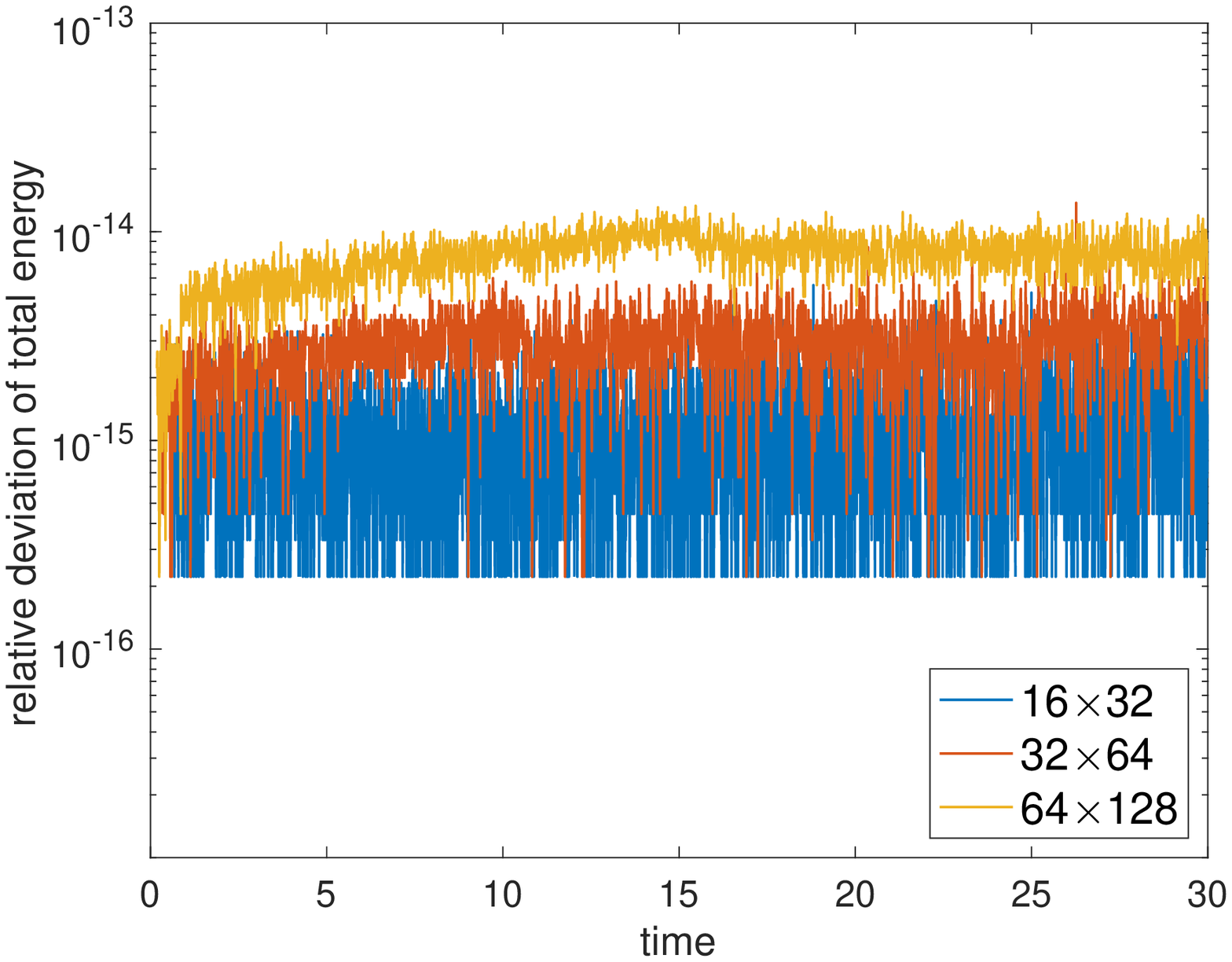}}
	\caption{Example \ref{ex:bumpontail}.  The time evolution of  relative deviation of total mass (a, d),  total momentum (b, e), and rtotal energy (c, f).  $\varepsilon=10^{-5}$.}
	\label{fig:bump1d_invar}
\end{figure}

\end{exa}

\subsection{2D2V Vlasov-Poisson system}

\begin{exa} \label{ex:weak2d} (Weak Landau damping.) We consider the 2D2V version of weak Landau damping. The initial  
condition is 
\begin{equation}
	\label{eq:weak}
	f(\bx,\bv,t=0) =\frac{1}{(2 \pi)^{d / 2}} \left(1+\alpha \sum_{m=1}^{d} \cos \left(k x_{m}\right)\right)\exp\left(-\frac{|\bv|^2}{2}\right),
\end{equation}
where $d=2$, $\alpha=0.01$, and $k=0.5$. We set the computation domain as $[0,L_x]^2\times[-L_v,L_v]^2$, where $L_x=\frac{2\pi}{k}$ and $L_v=6$, and the truncation threshold $\varepsilon=10^{-4}$. As with the 1D1V case, the electric energy will decay exponentially fast over time. To mitigate the curse of the dimensionality, we represent the four dimensional solution in the third order HT tensor format, without further decomposition in the spatial directions.   
 In Figure \ref{fig:weak2d_elec_con_k1}-\ref{fig:weak2d_elec_con_k2}, we report the time evolution of the electric energy, hierarchical ranks of the numerical solution,  relative deviation of total mass and energy together with absolute deviation of total momentum $J_1$ and $J_2$. It is known that the solution processes low rank structures on phase space, and hence we expect the proposed low rank DG method can efficiently avoid the curse of dimensionality. The CPU cost for the simulation with meshes $16^2\times32^2$, $32^2\times64^2$, $64^2\times128^2$ is for 550s, 1092s, 3047s for $k=1$ and 556s, 1143s, and 4435s for  $k=2$ with serial implementation. Furthermore, the LoMaC low rank DG method can conserve the physical invariants up to machine precision.

\begin{figure}[h!]
	\centering
	\subfigure[]{\includegraphics[height=40mm]{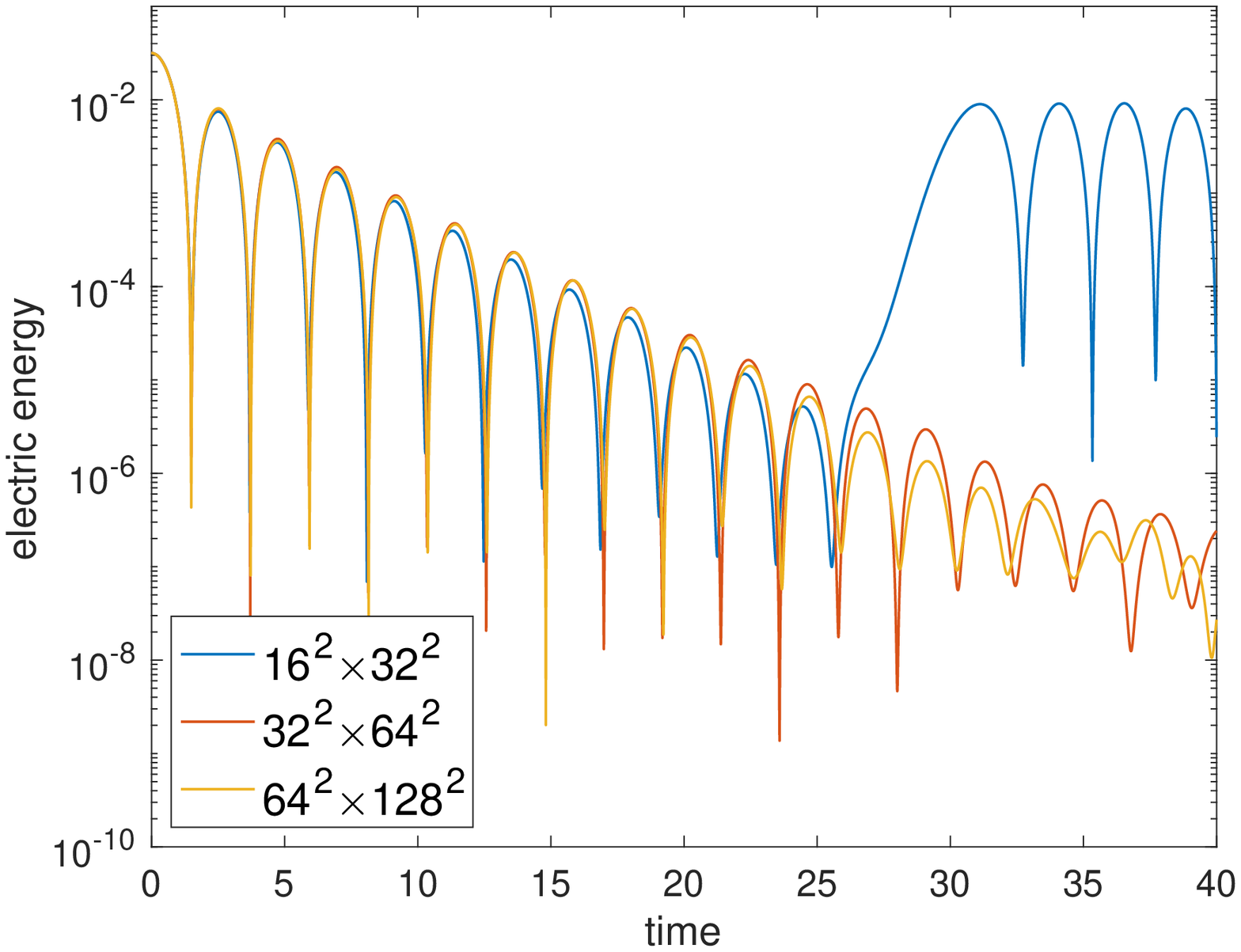}}
		\subfigure[]{\includegraphics[height=40mm]{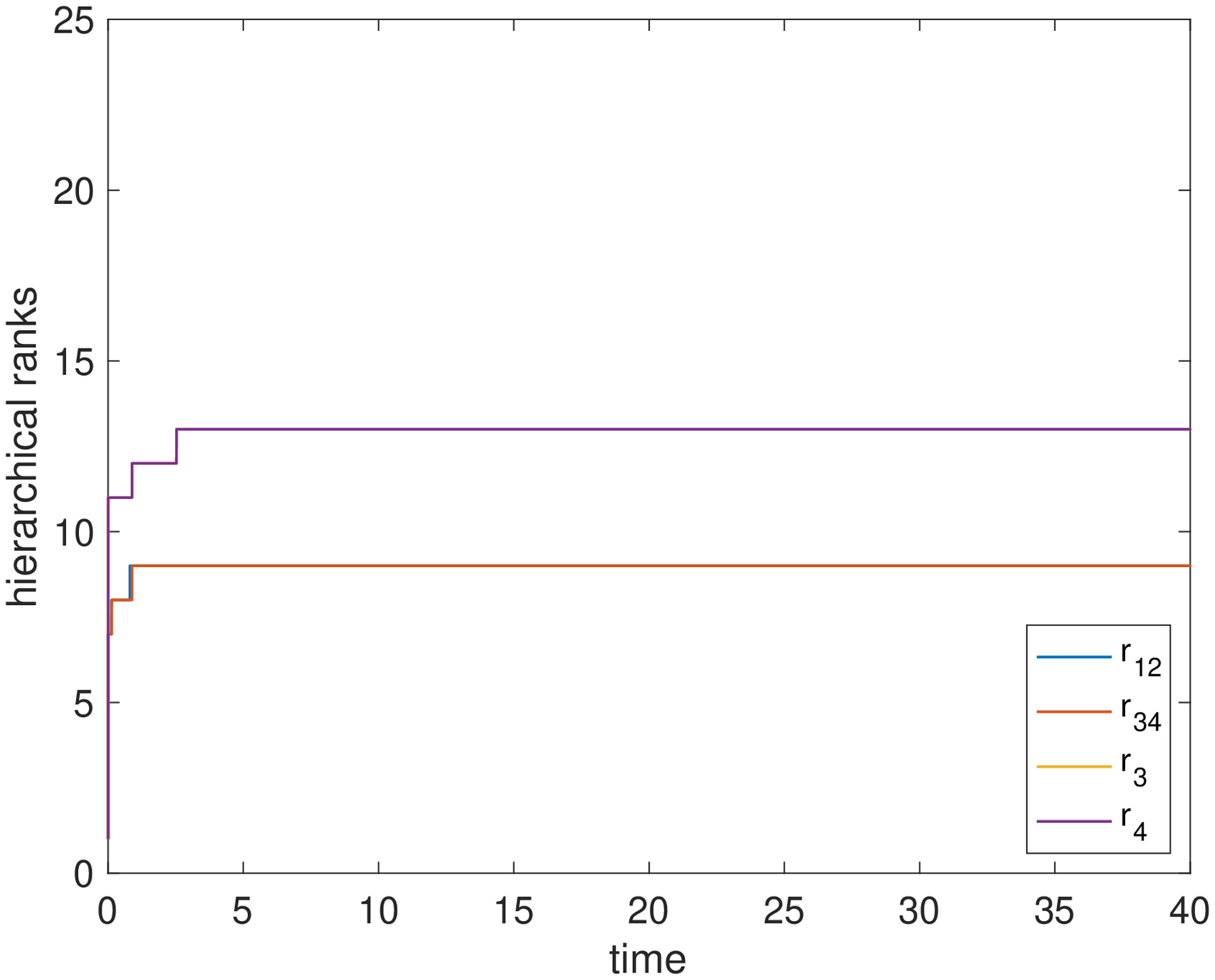}}
		\subfigure[]{\includegraphics[height=40mm]{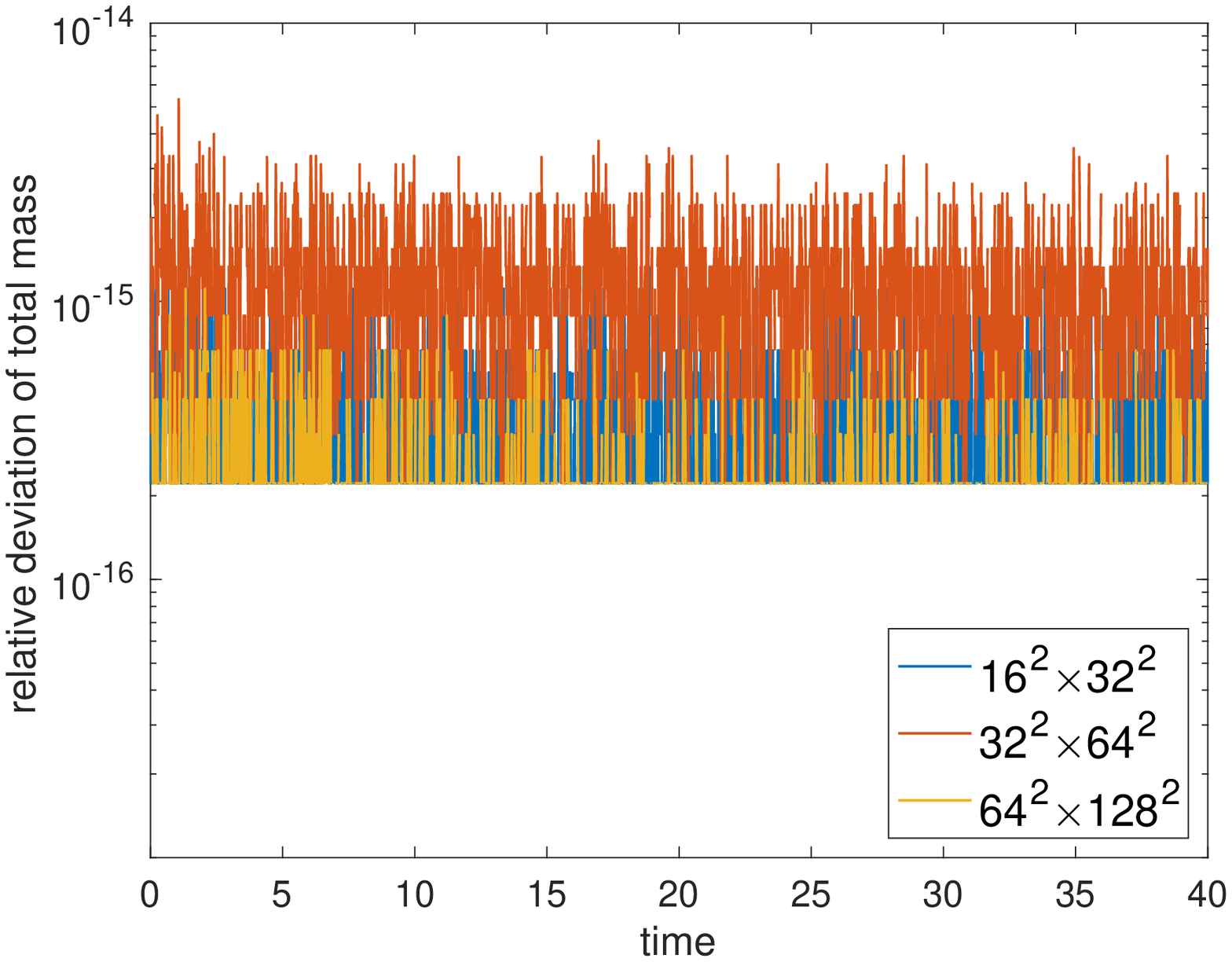}}
		\subfigure[]{\includegraphics[height=40mm]{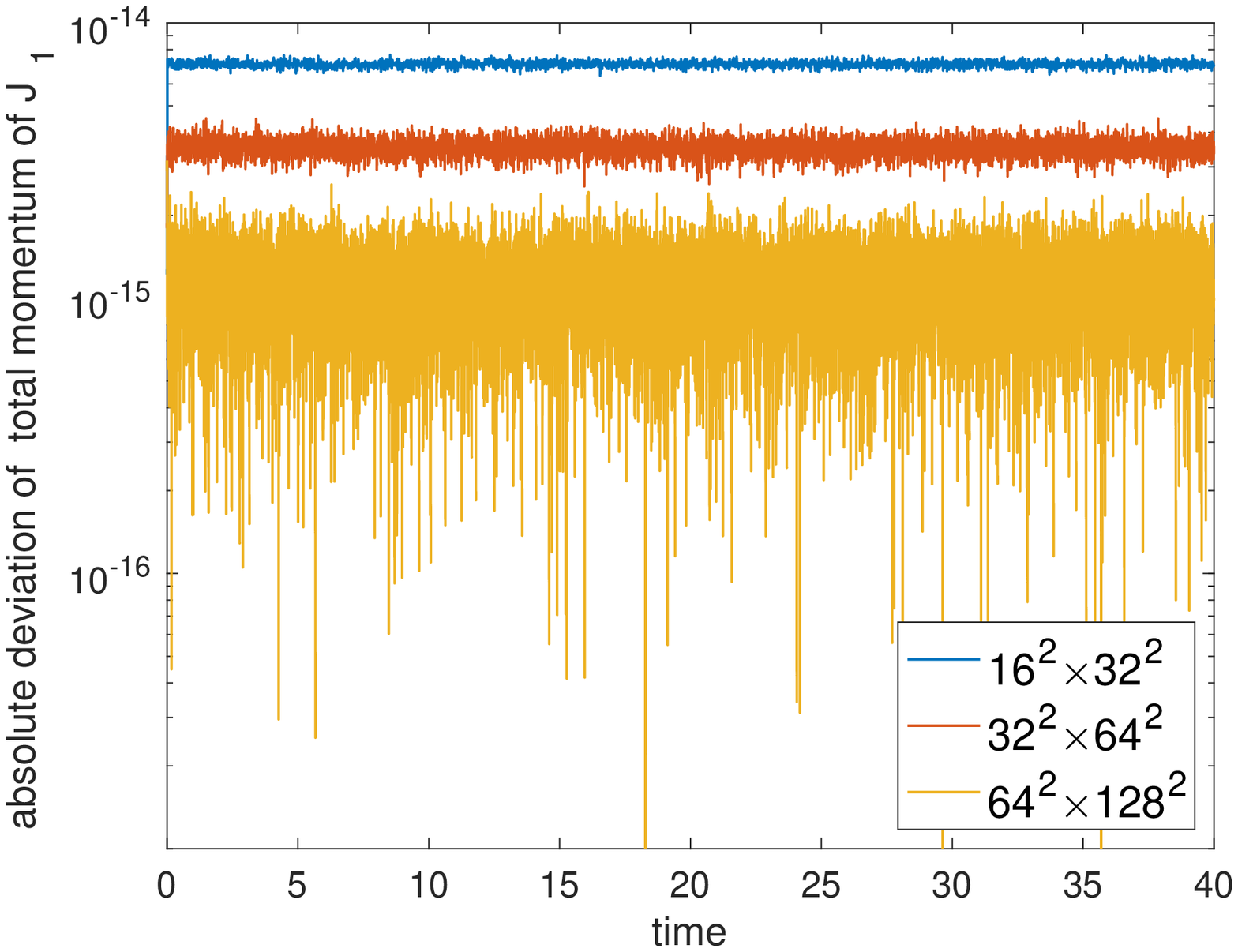}}
			\subfigure[]{\includegraphics[height=40mm]{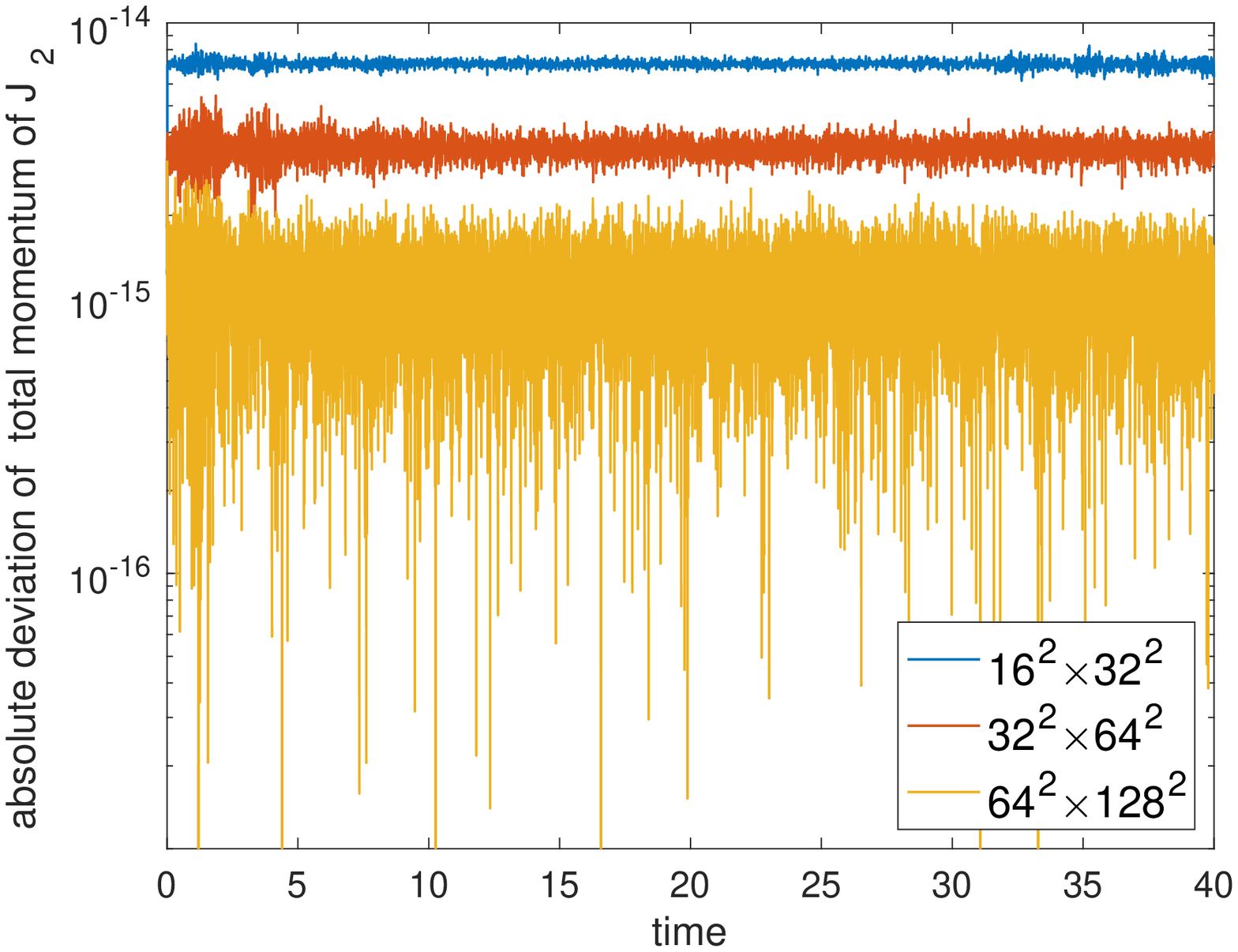}}
		\subfigure[]{\includegraphics[height=40mm]{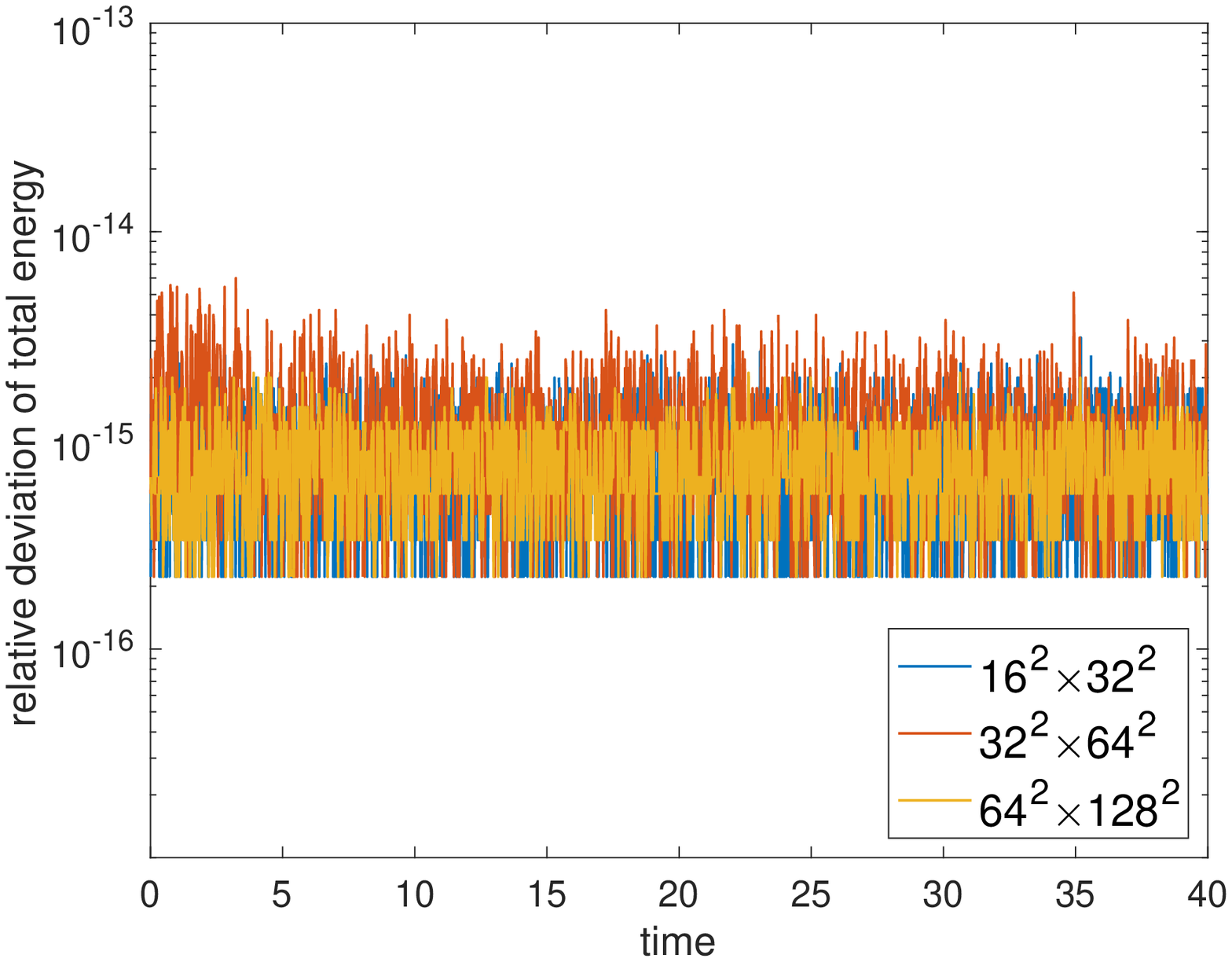}}
	\caption{Example \ref{ex:weak2d}. The time evolution of  electric energy (a), hierarchical ranks of the numerical solution of mesh size $N^2_x\times N^2_v=64^2\times128^2$ (b), relative deviation of total mass (c), absolute total momentum $J_1$ (d), absolute total momentum $J_2$ (e), and relative deviation of total energy (f). $\varepsilon=10^{-4}$. $k=1$. }
	\label{fig:weak2d_elec_con_k1}
\end{figure}

\begin{figure}[h!]
	\centering
	\subfigure[]{\includegraphics[height=40mm]{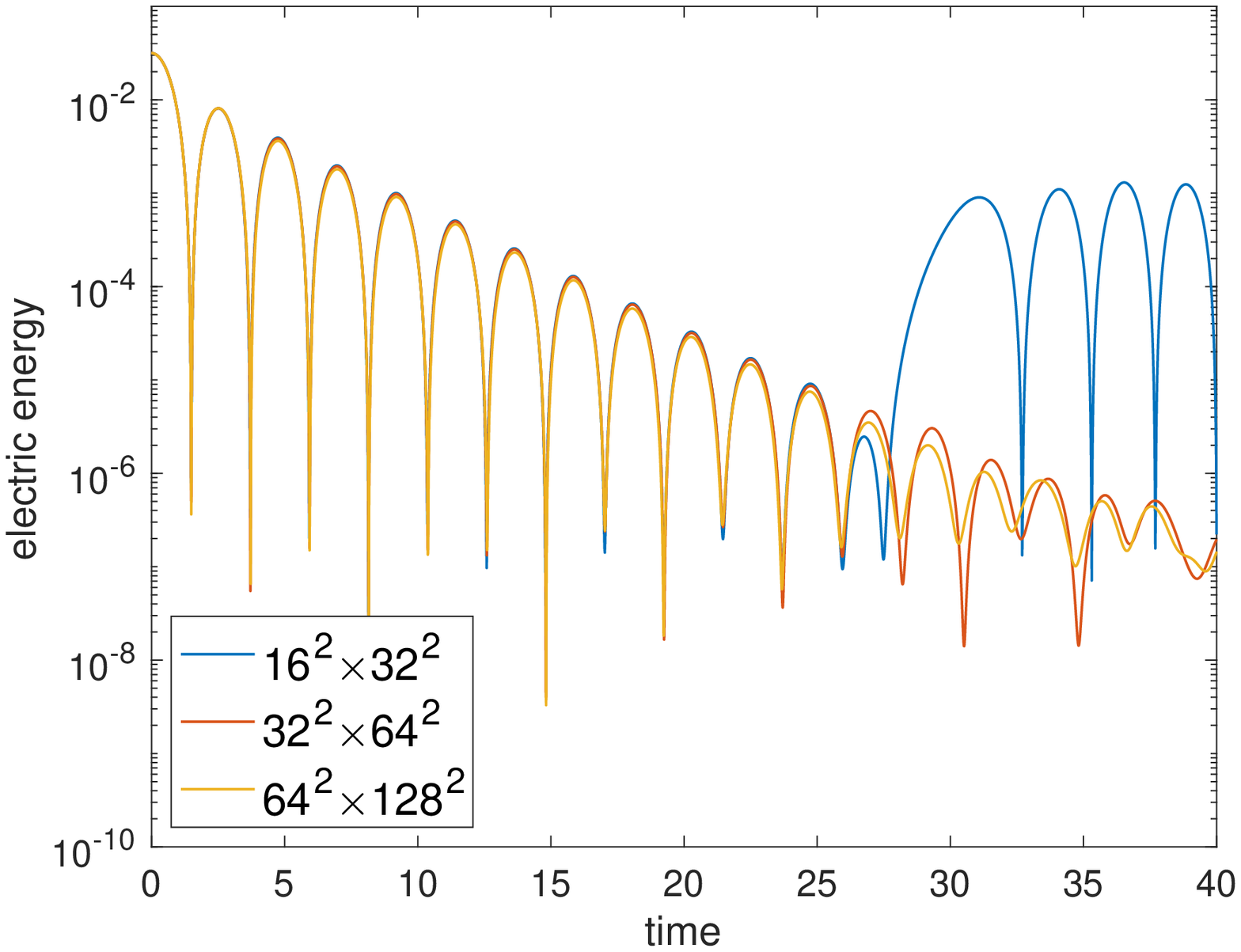}}
		\subfigure[]{\includegraphics[height=40mm]{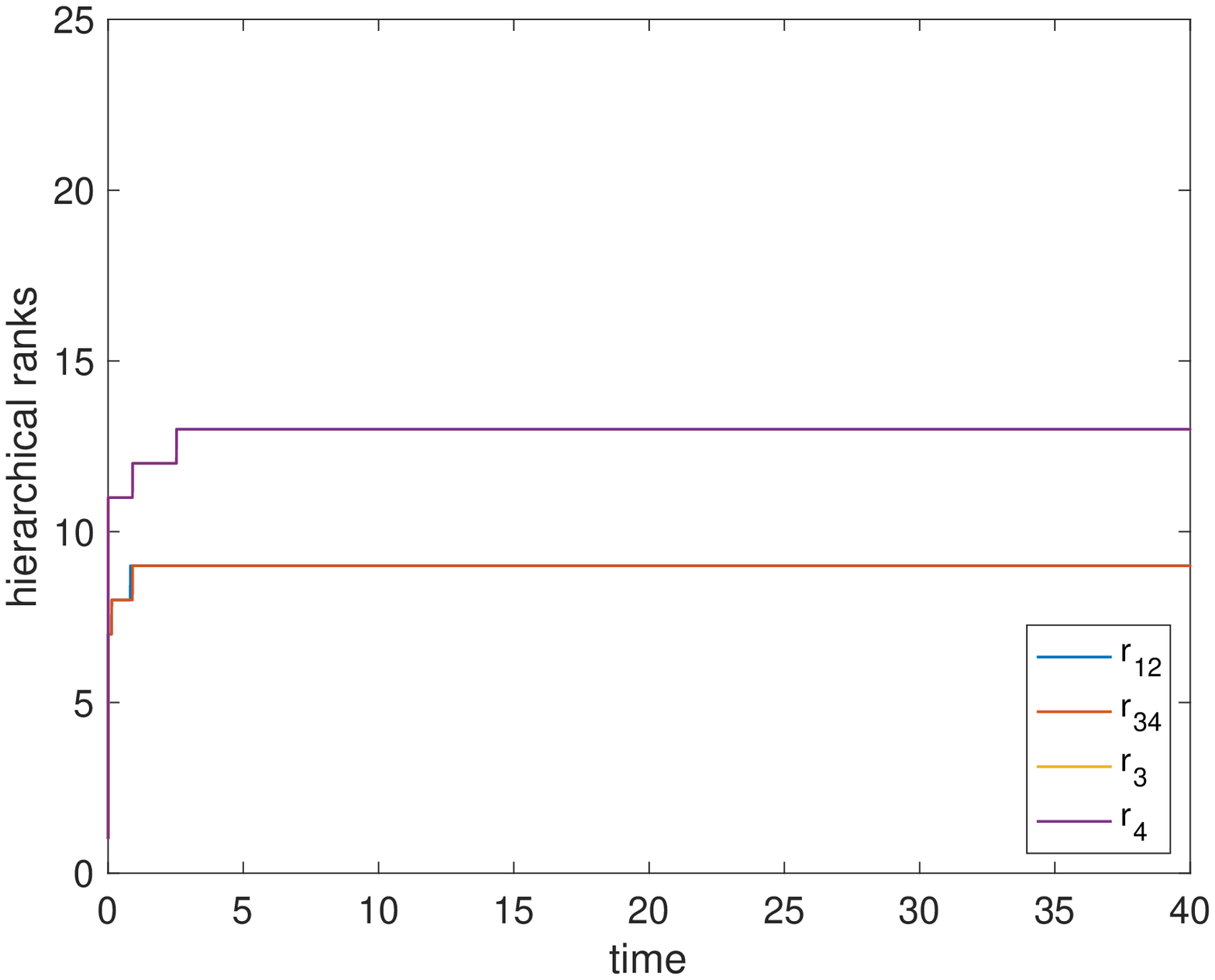}}
		\subfigure[]{\includegraphics[height=40mm]{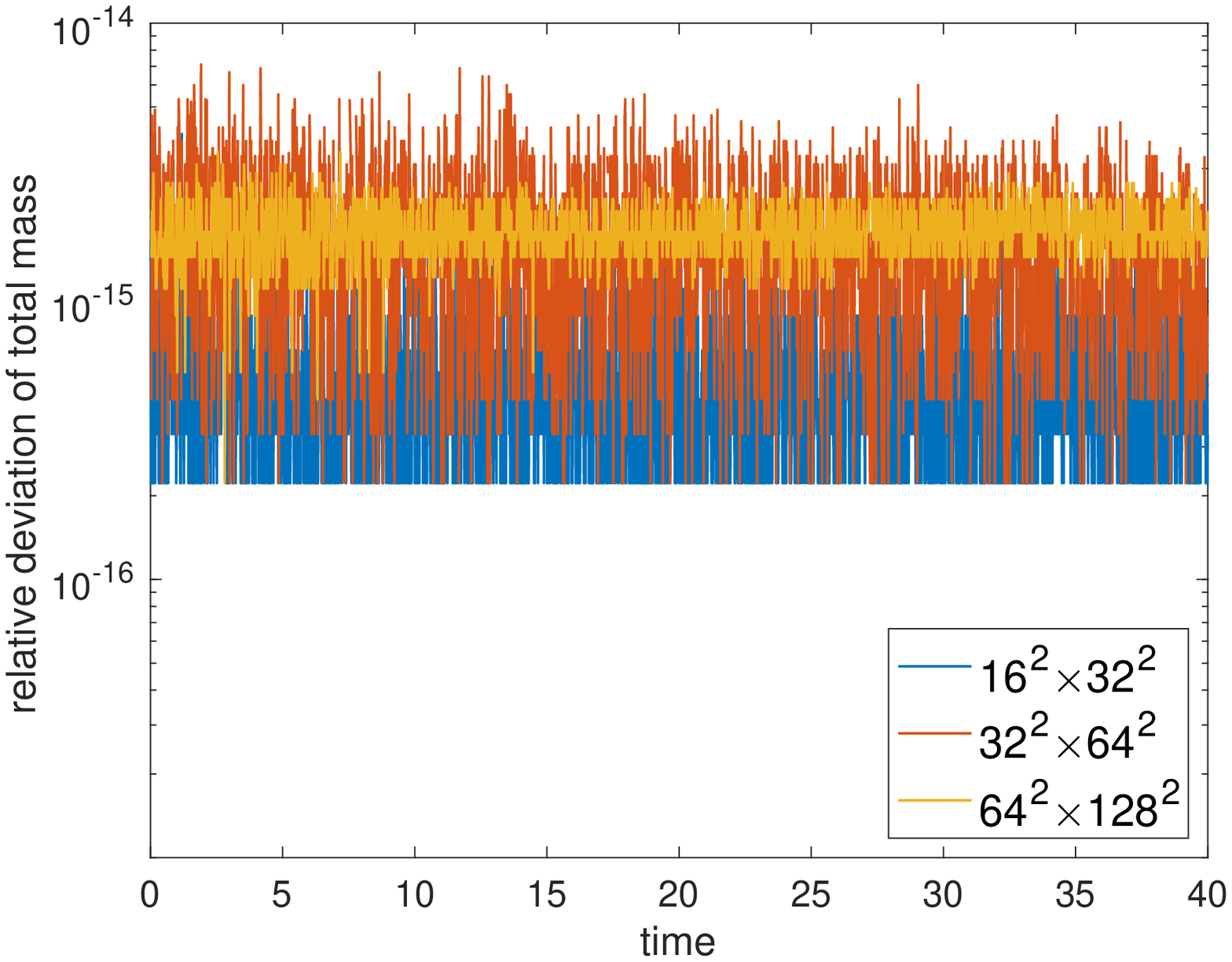}}
		\subfigure[]{\includegraphics[height=40mm]{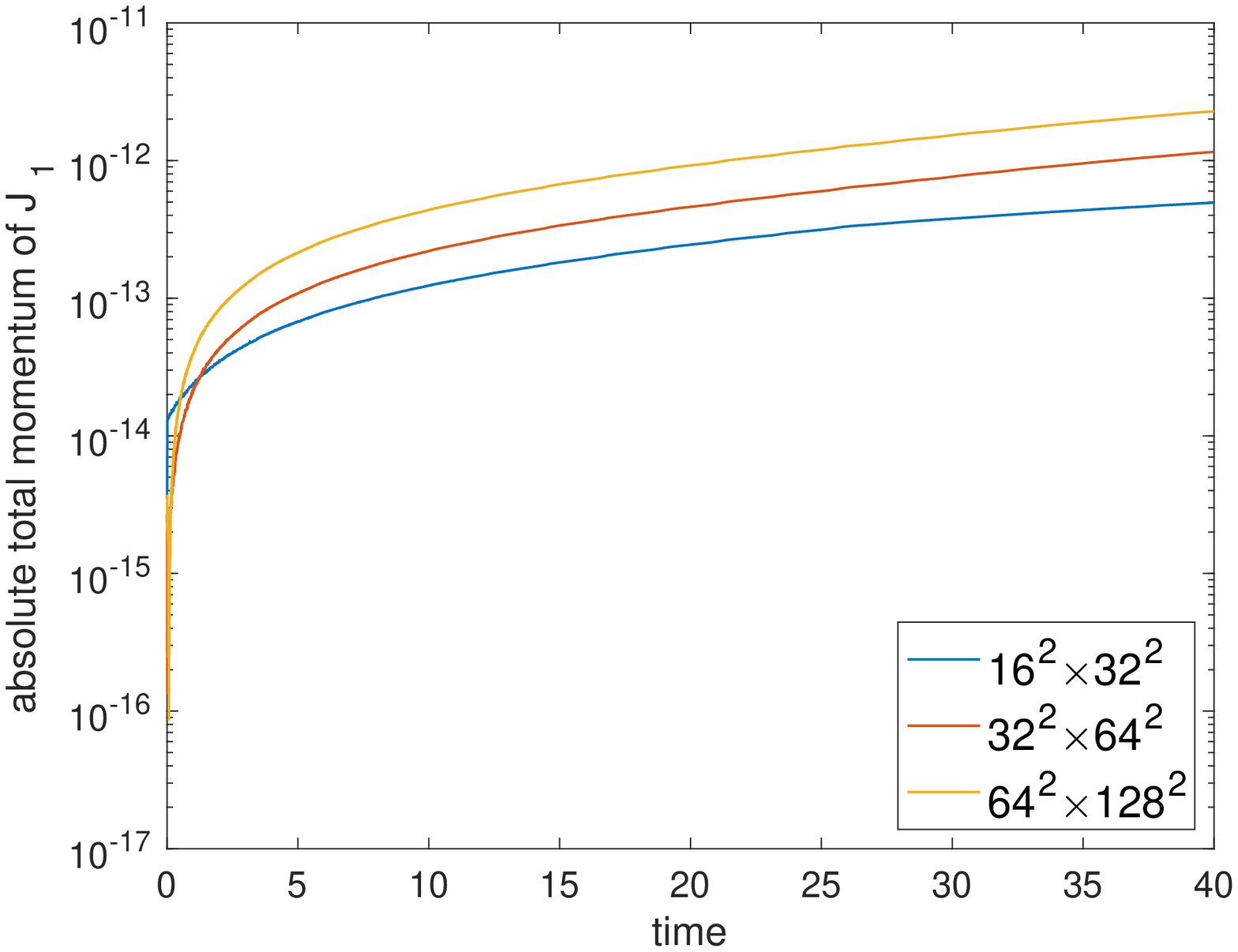}}
			\subfigure[]{\includegraphics[height=40mm]{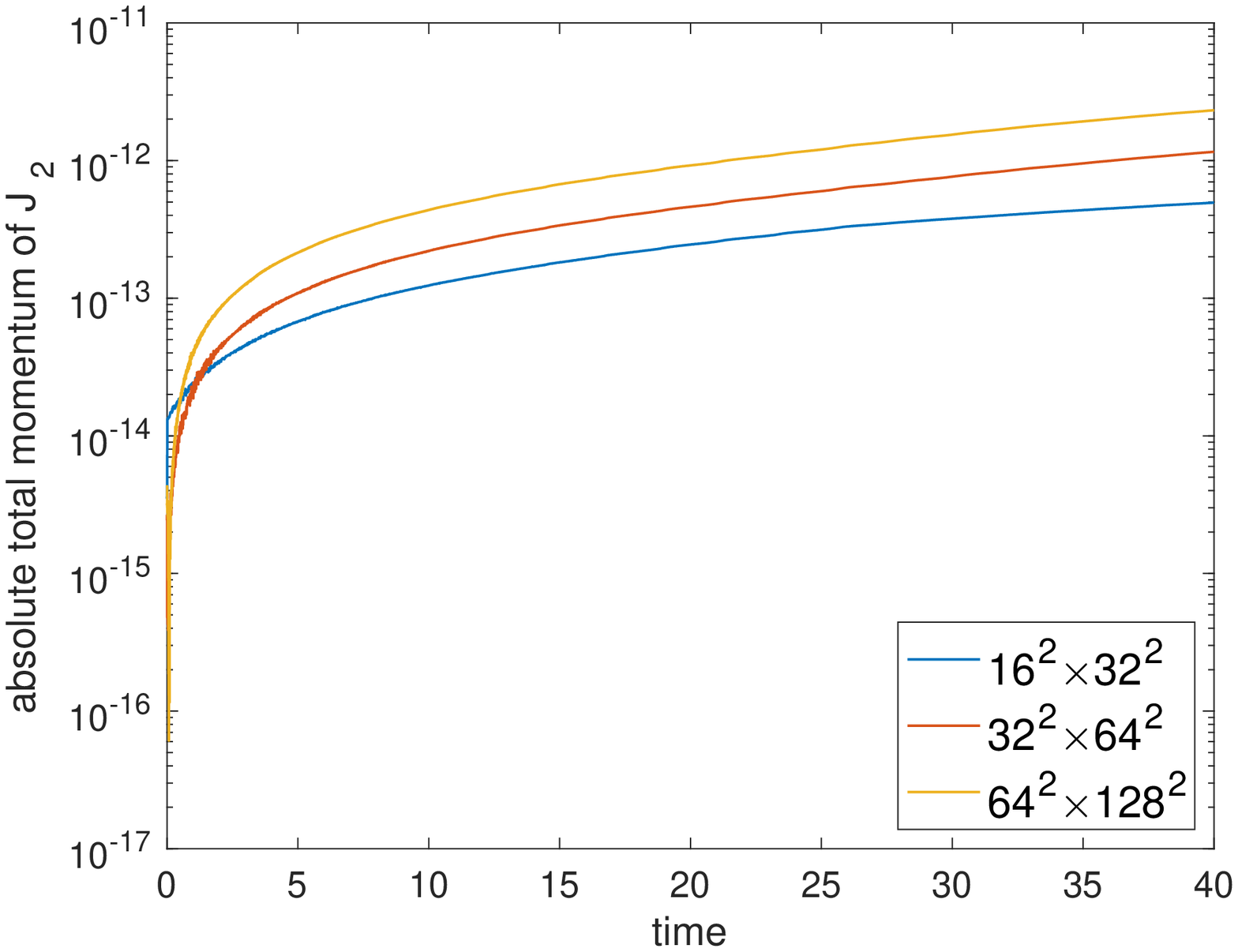}}
		\subfigure[]{\includegraphics[height=40mm]{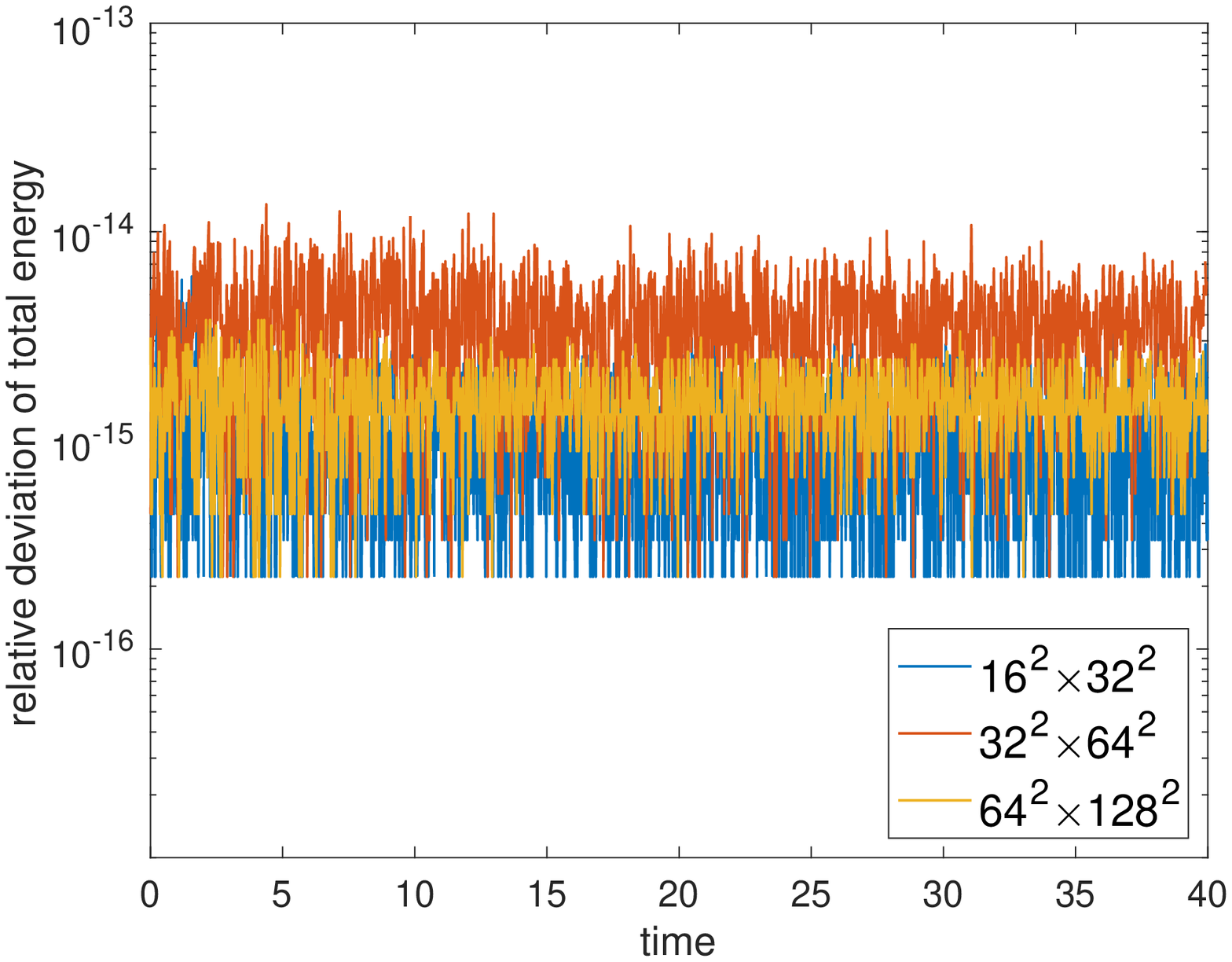}}
	\caption{Example \ref{ex:weak2d}. The time evolution of  electric energy (a), hierarchical ranks of the numerical solution of mesh size $N^2_x\times N^2_v=64^2\times128^2$ (b), relative deviation of total mass (c), absolute total momentum $J_1$ (d), absolute total momentum $J_2$ (e), and relative deviation of total energy (f). $\varepsilon=10^{-4}$. $k=2$. }
	\label{fig:weak2d_elec_con_k2}
\end{figure}

 \end{exa}

 \begin{exa} \label{ex:two2d}(Two-stream instability.) The last example is the 2D2V two-stream instability with initial condition
\begin{equation}
	\label{eq:two2d}
	f(\bx,\bv,t=0) =\frac{1}{2^d(2 \pi)^{d / 2}} \left(1+\alpha \sum_{m=1}^{d} \cos \left(k x_{m}\right)\right)\prod_{m=1}^d\left(\exp\left(-\frac{(v_m-v_0)^2}{2}\right) + \exp\left( -\frac{(v_m+v_0)^2}{2}\right)\right),
\end{equation}
where $d=2$, $\alpha=10^{-3}$, $v_0=2.4$, and $k=0.2$.  The computation domain is set as $[0,L_x]^2\times[-L_v,L_v]^2$, where $L_x=\frac{2\pi}{k}$ and $L_v=8$, and  the truncation threshold is $\varepsilon=10^{-4}$.  In Figure \ref{fig:two2d_elec_con_k1}-\ref{fig:two2d_elec_con_k2},  we report the time evolution of the electric energy, hierarchical ranks of the numerical solution of mesh size  $N_x^2\times N^2_v=128^2\times256^2$,  relative deviation of total mass and energy together with absolute derivation  of total momentum $J_1$ and $J_2$. The results of the electric energy evolution agree with those reported in the literature. In addition, the dynamics is efficiently captured by the low rank DG method, observing that the hierarchical ranks of the solution remains very low until  approximately $t=17$ and then start to increase due to the instability developed. Lastly, the proposed method conserves the total mass, momentum, and energy.

\begin{figure}[h!]
	\centering
	\subfigure[]{\includegraphics[height=40mm]{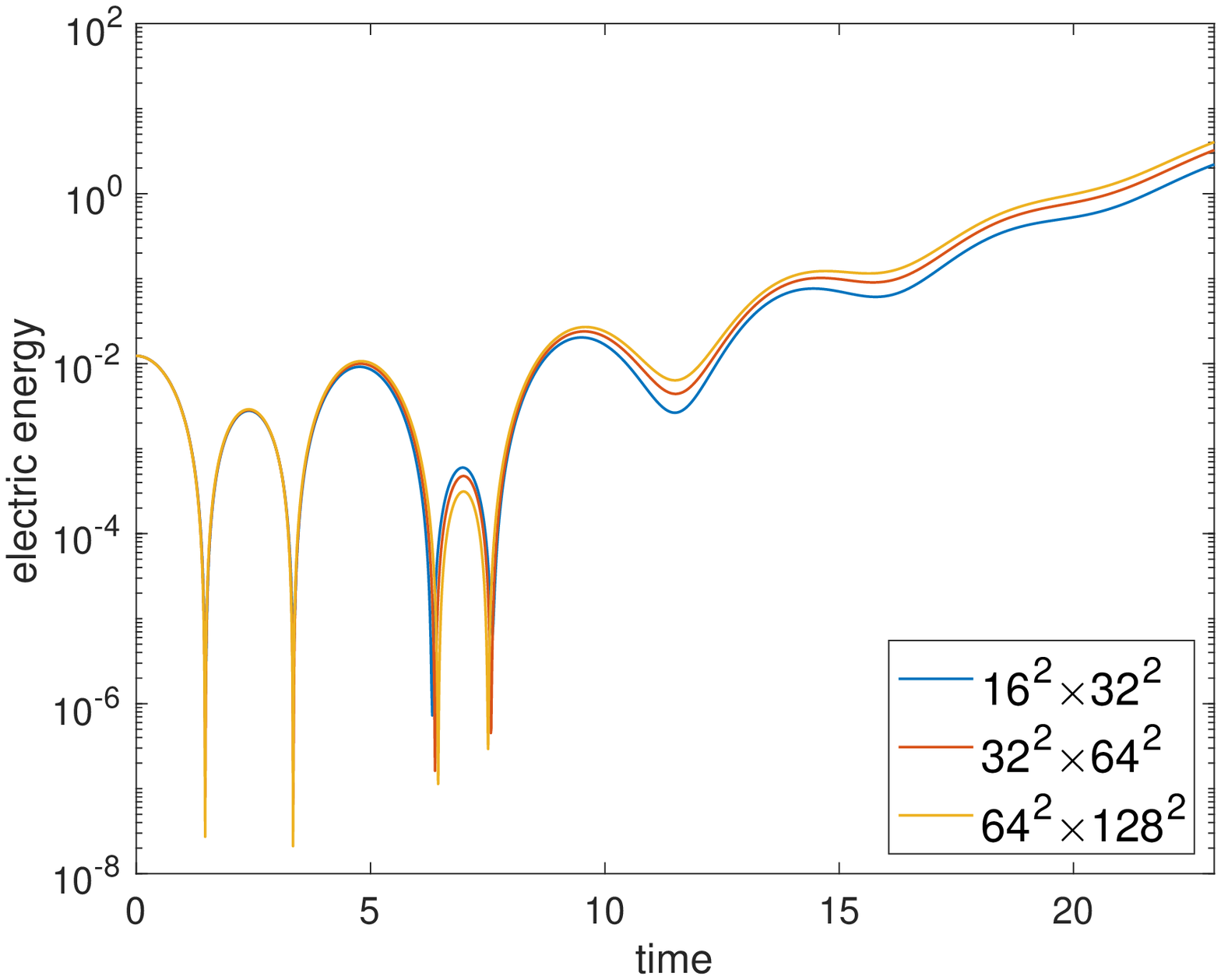}}
		\subfigure[]{\includegraphics[height=40mm]{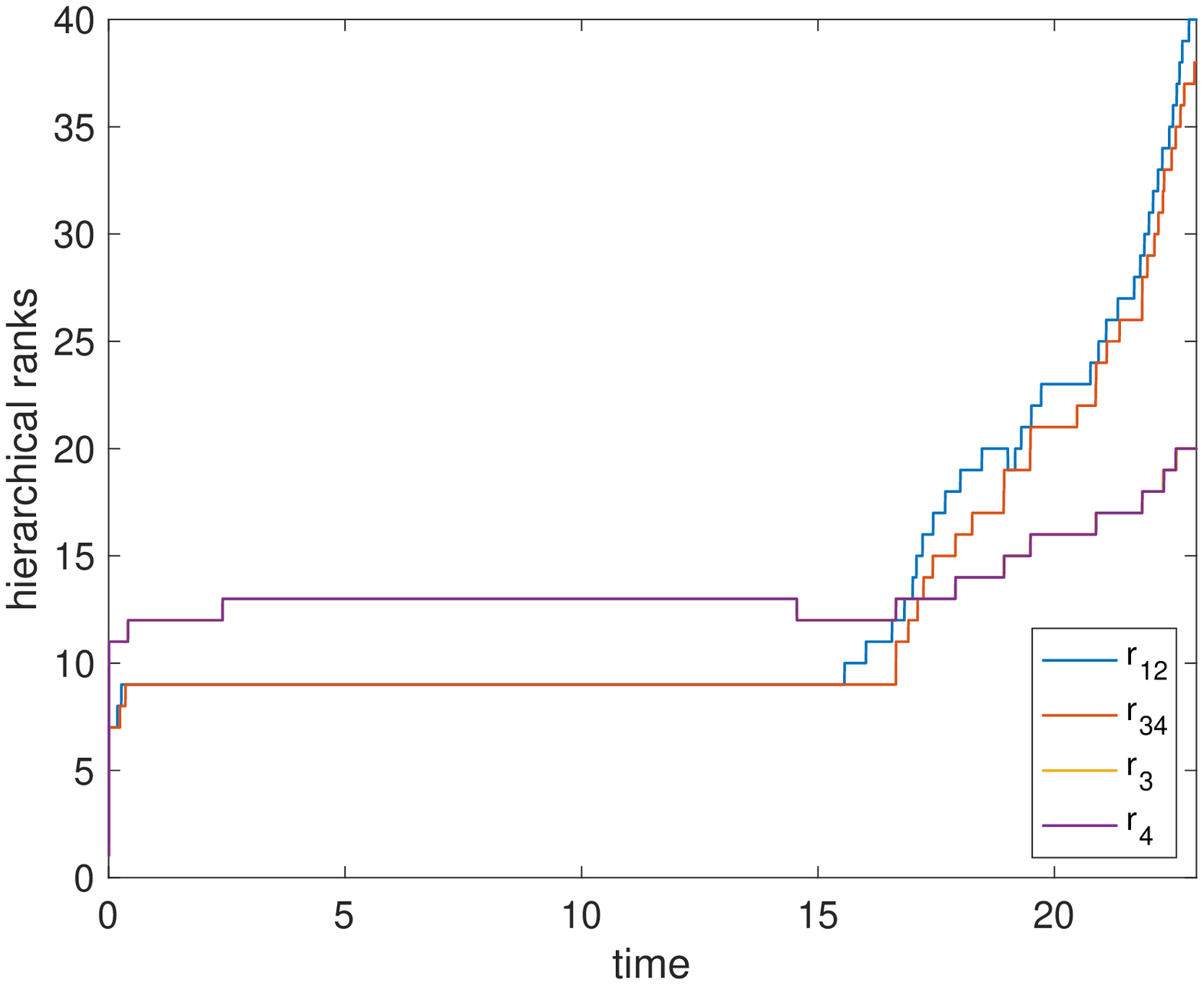}}
		\subfigure[]{\includegraphics[height=40mm]{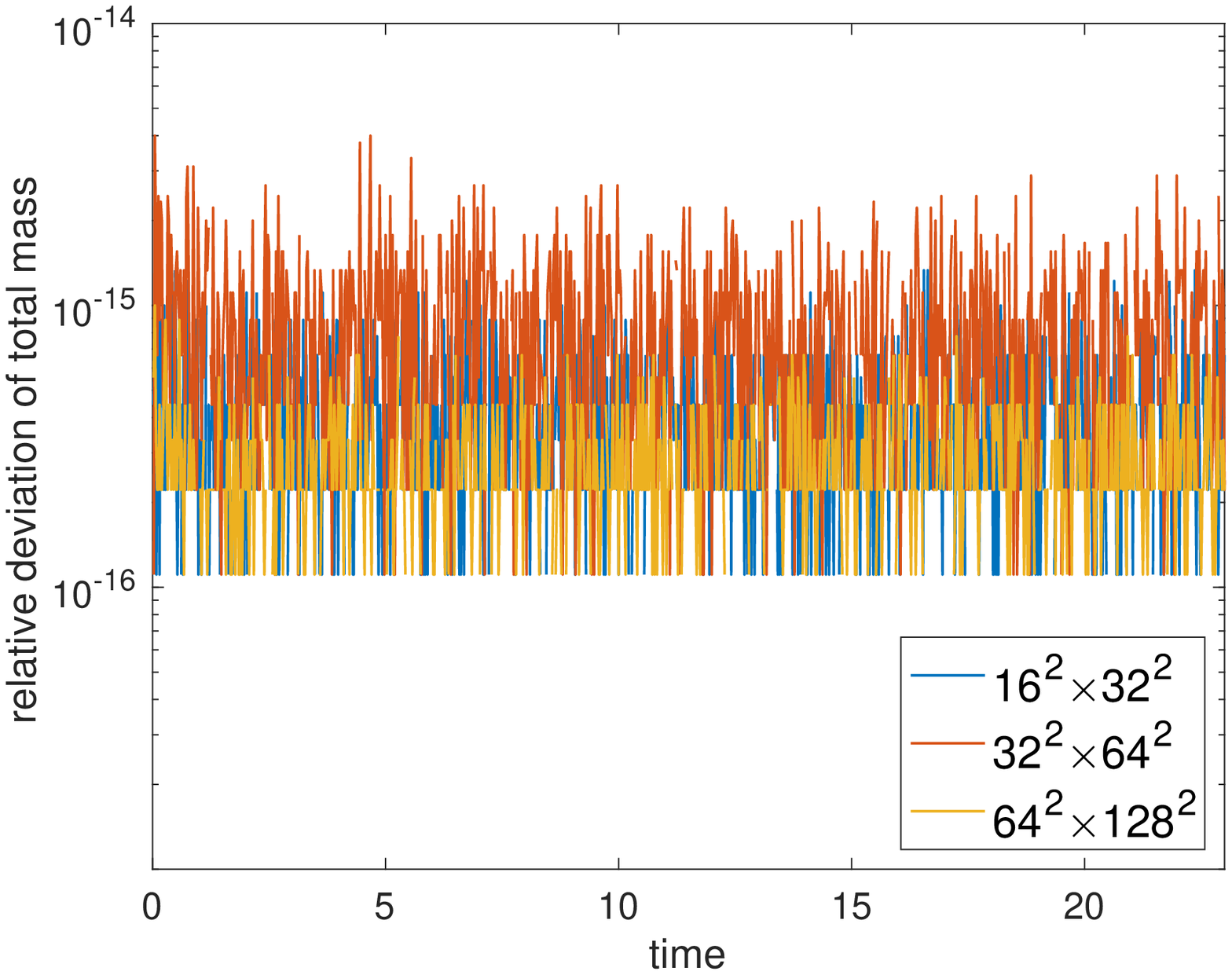}}
		\subfigure[]{\includegraphics[height=40mm]{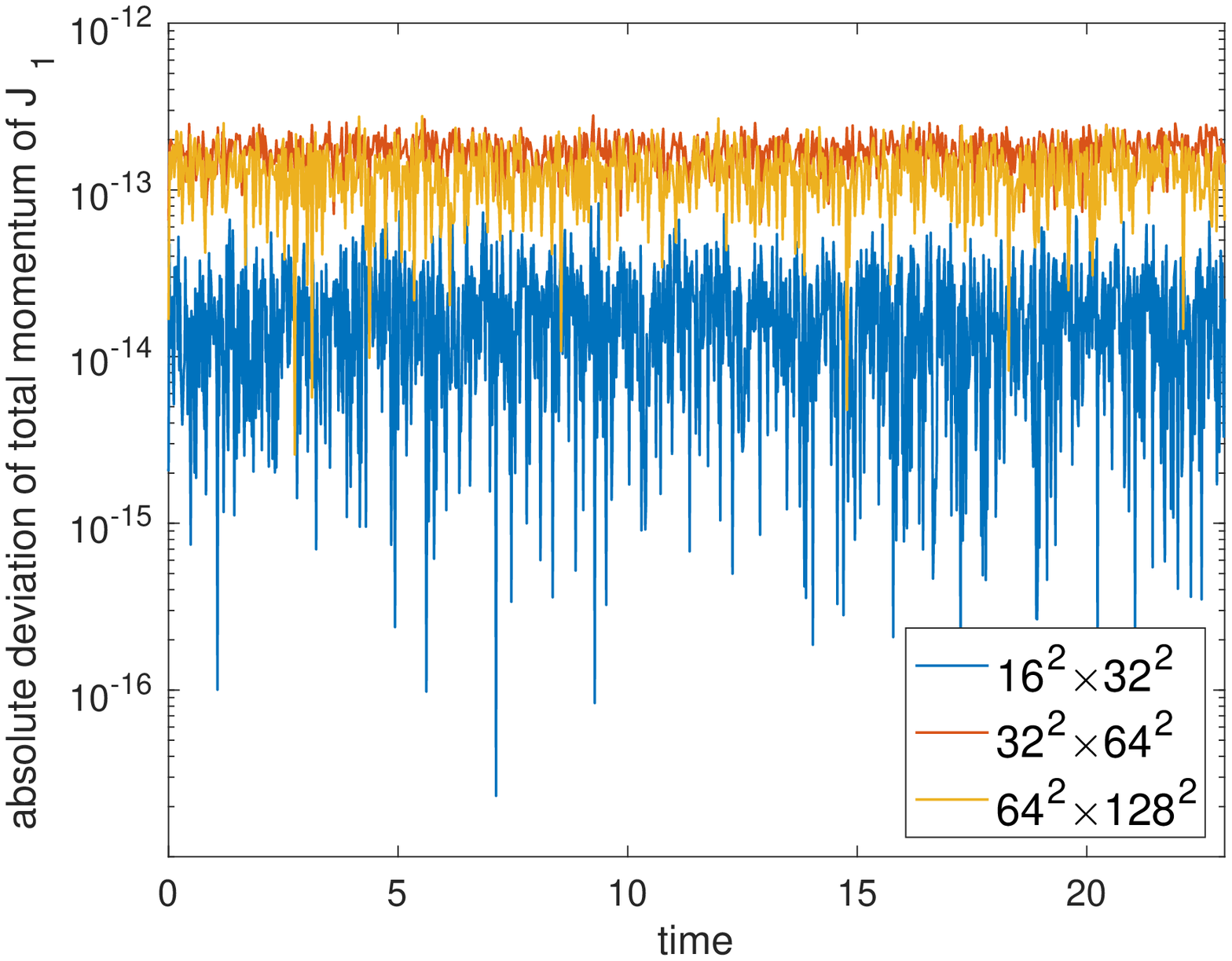}}
			\subfigure[]{\includegraphics[height=40mm]{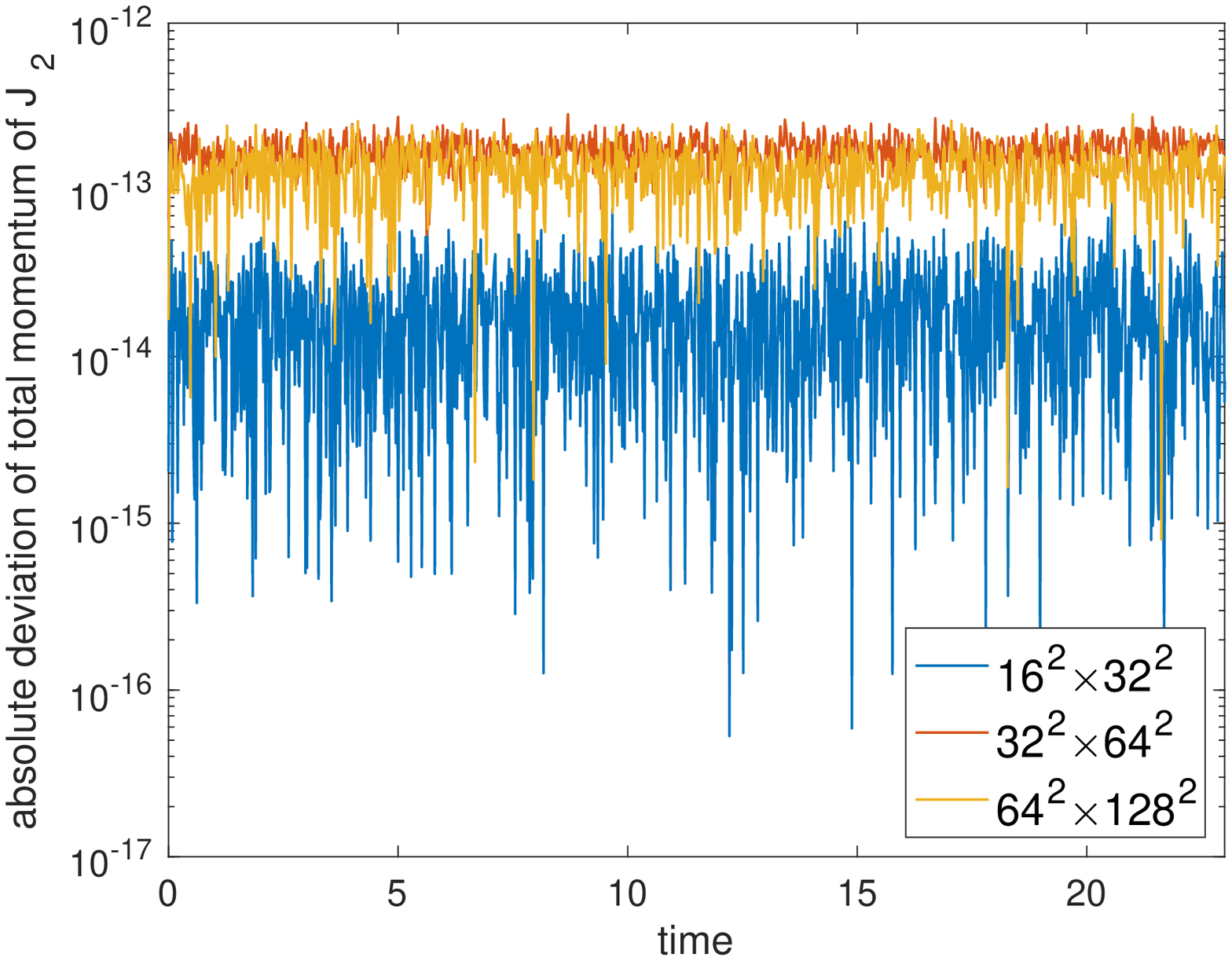}}
		\subfigure[]{\includegraphics[height=40mm]{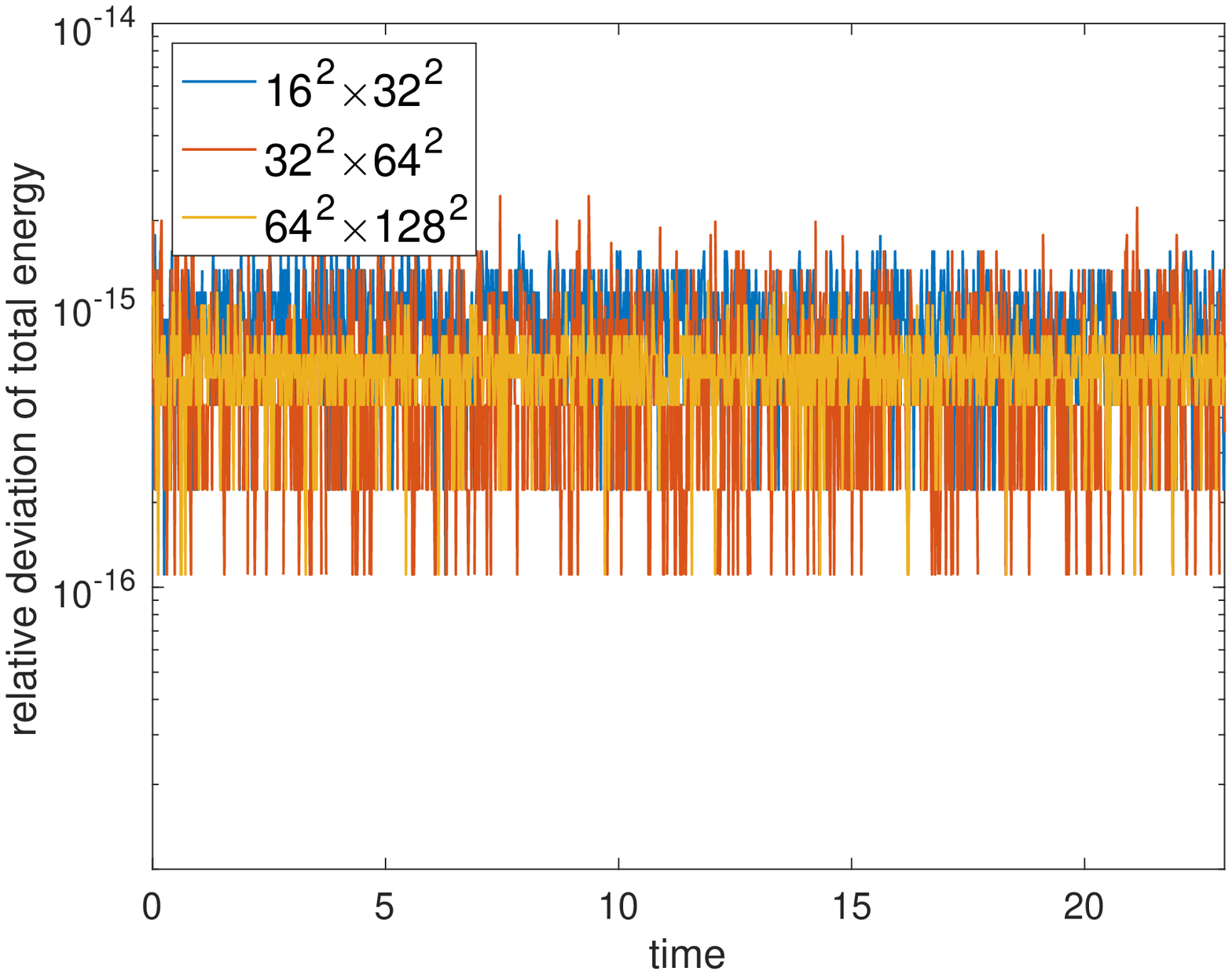}}
	\caption{Example \ref{ex:two2d}. The time evolution of  electric energy (a), hierarchical ranks of the numerical solution of mesh size $N^2_x\times N^2_v=64^2\times128^2$ (b), relative deviation of total mass (c), absolute deviation of total momentum $J_1$ (d) and total momentum $J_2$ (e), and relative deviation of total energy (f). $\varepsilon=10^{-5}$. $k=1$. }
	\label{fig:two2d_elec_con_k1}
\end{figure}

\begin{figure}[h!]
	\centering
	\subfigure[]{\includegraphics[height=40mm]{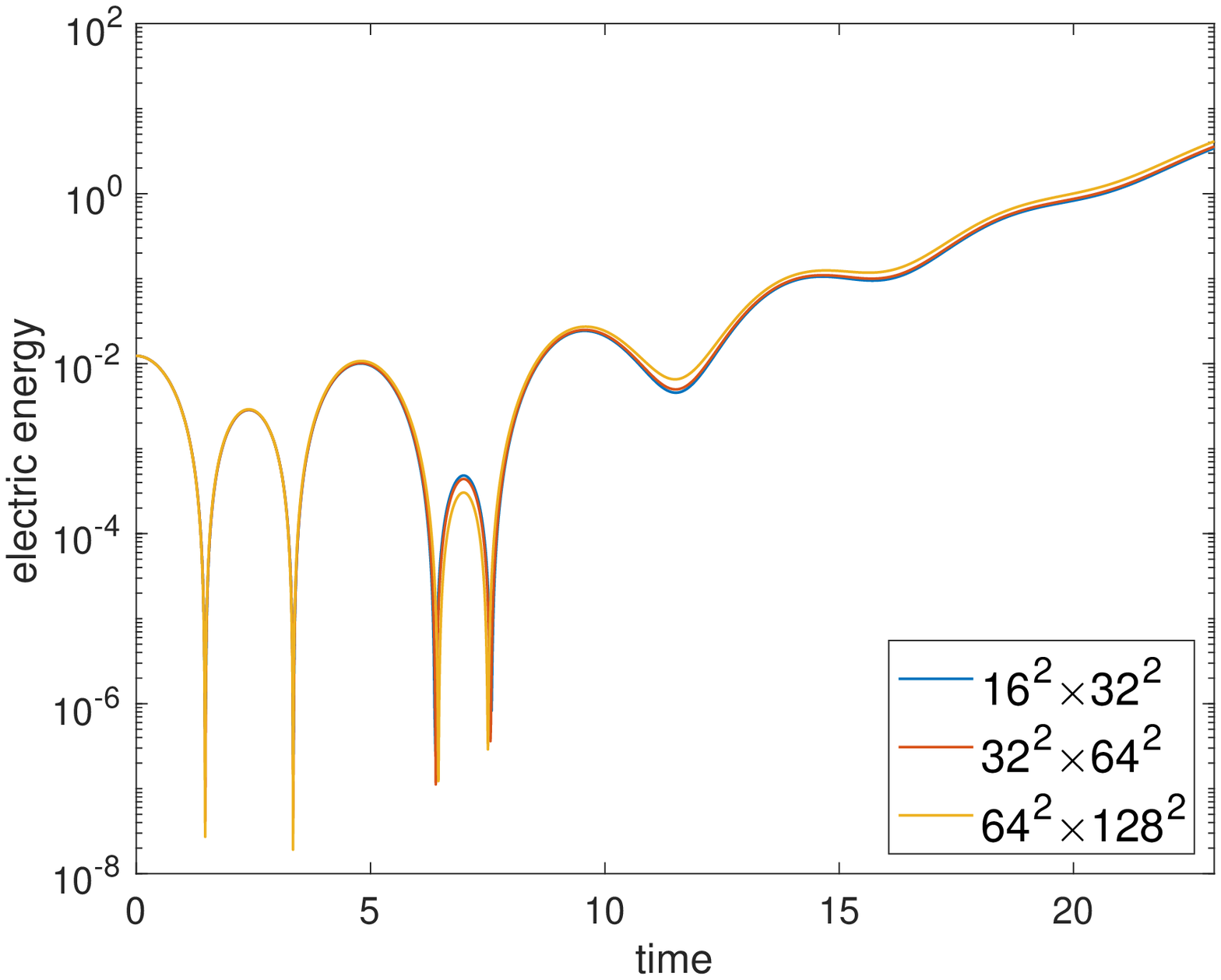}}
		\subfigure[]{\includegraphics[height=40mm]{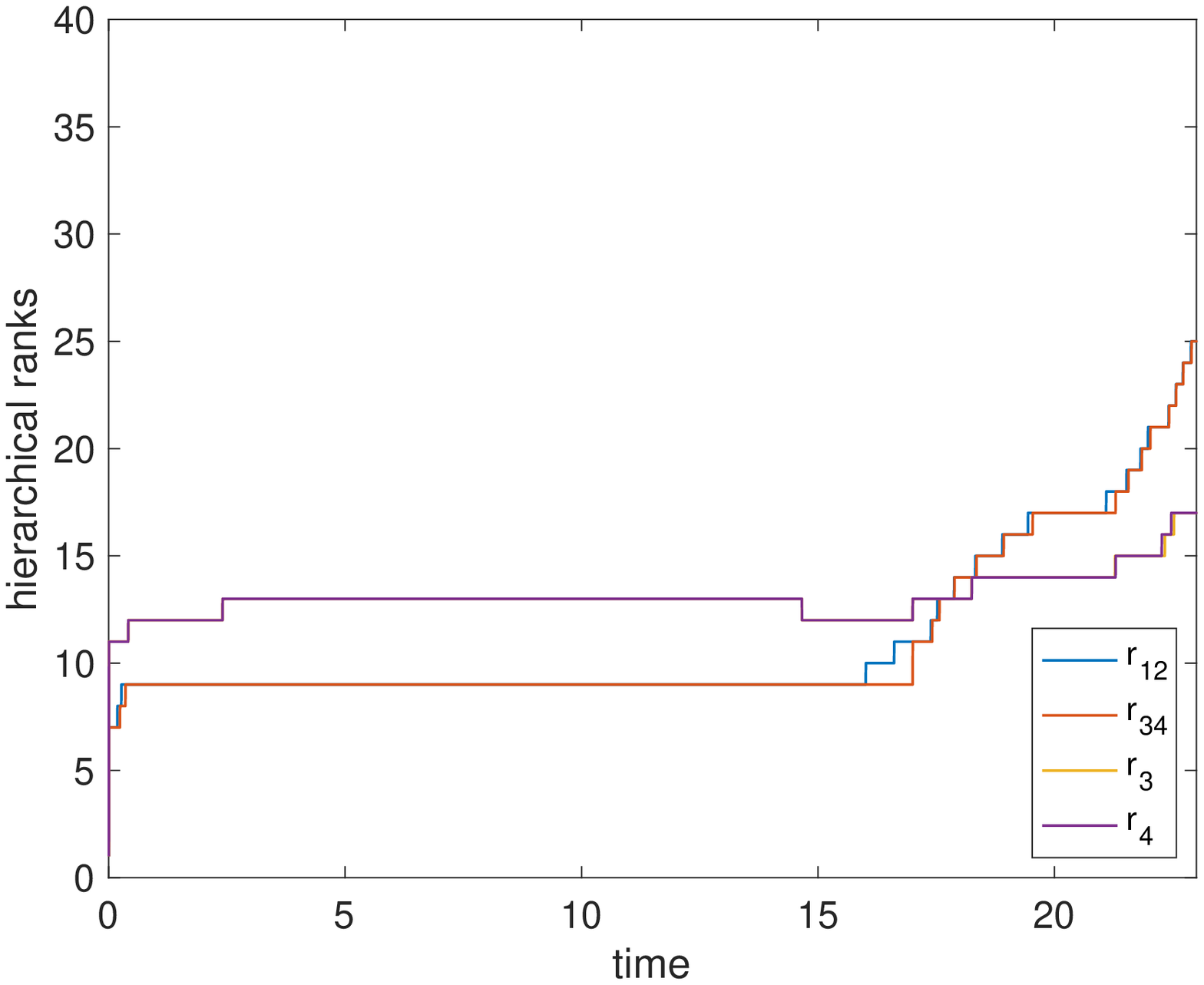}}
		\subfigure[]{\includegraphics[height=40mm]{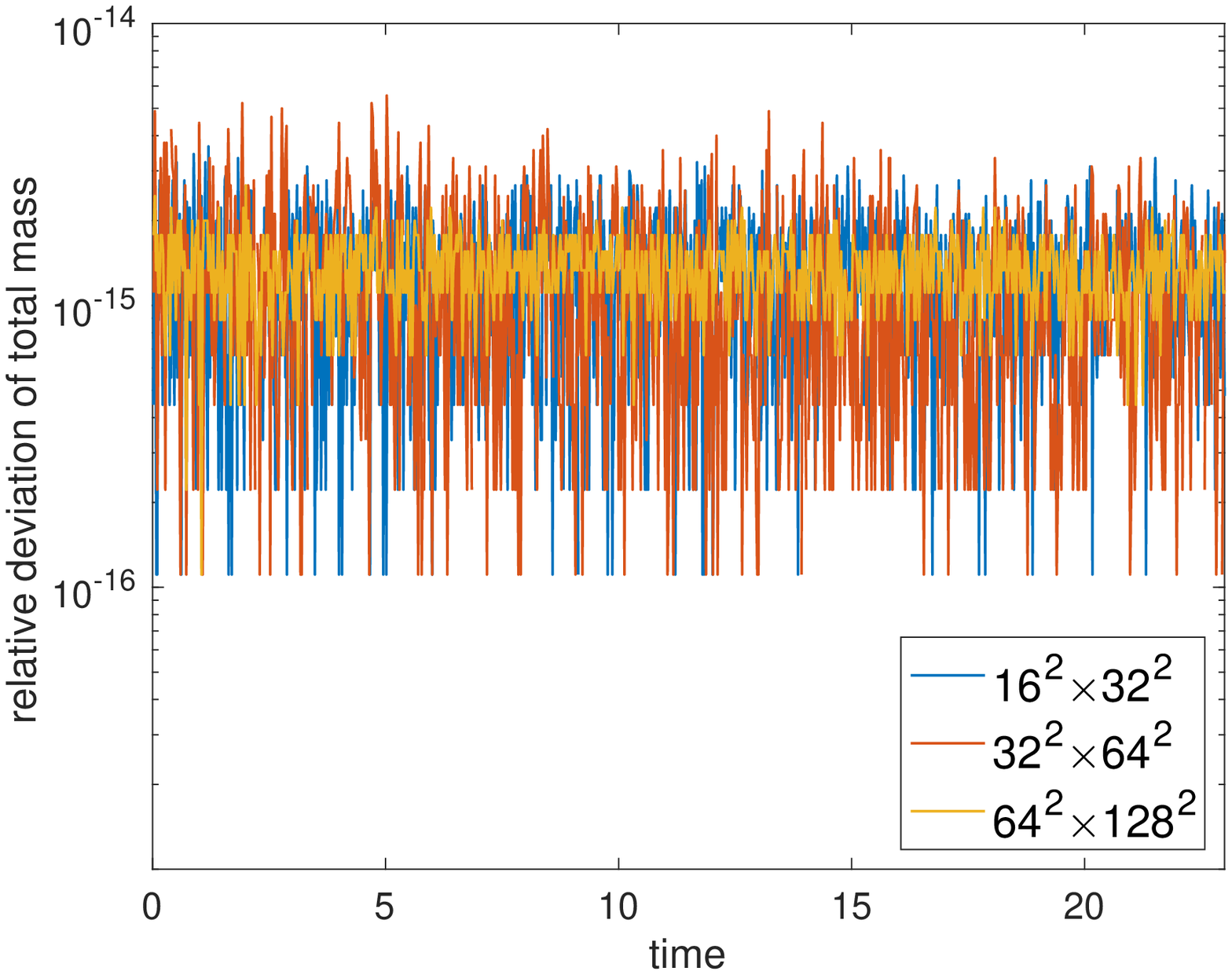}}
		\subfigure[]{\includegraphics[height=40mm]{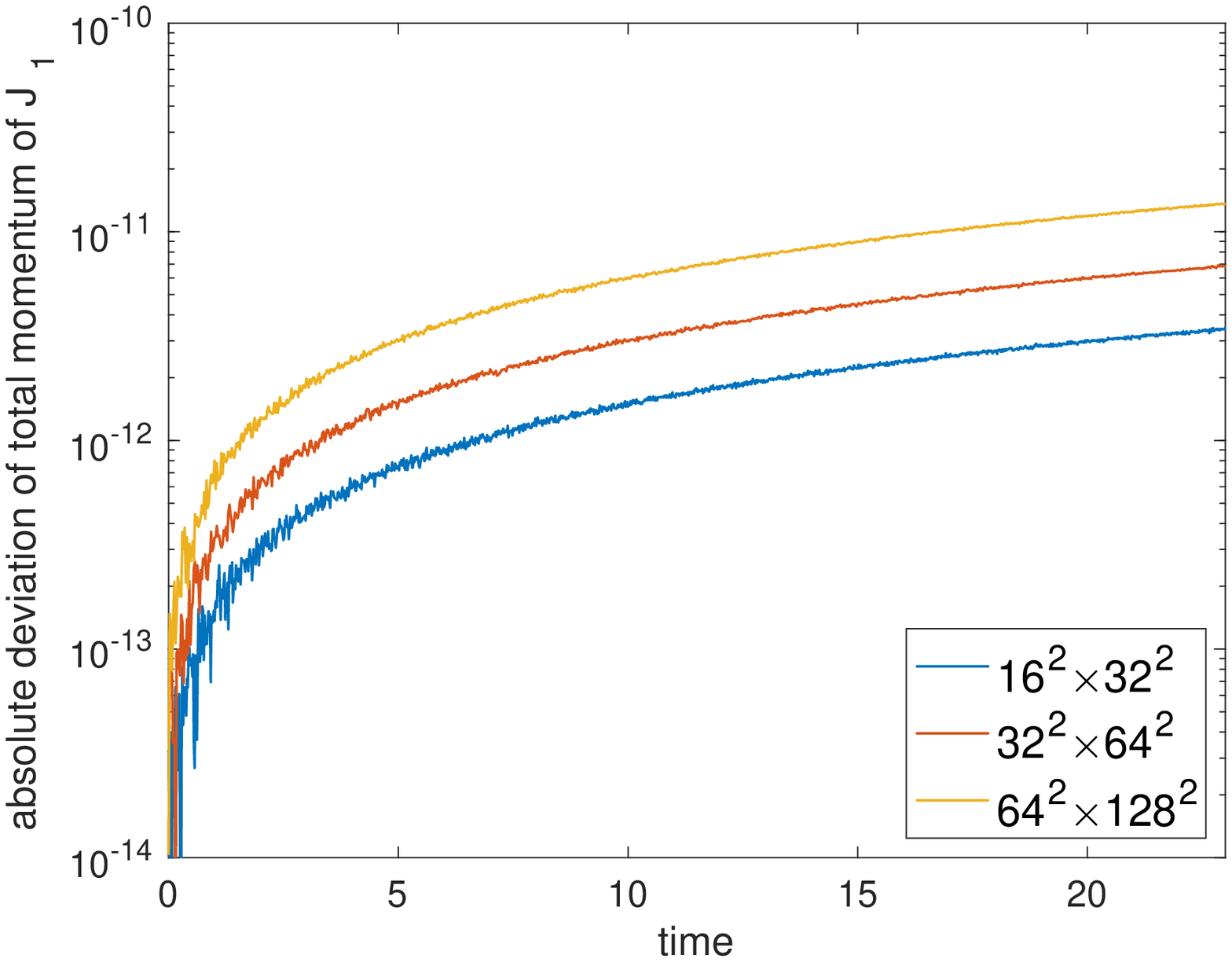}}
			\subfigure[]{\includegraphics[height=40mm]{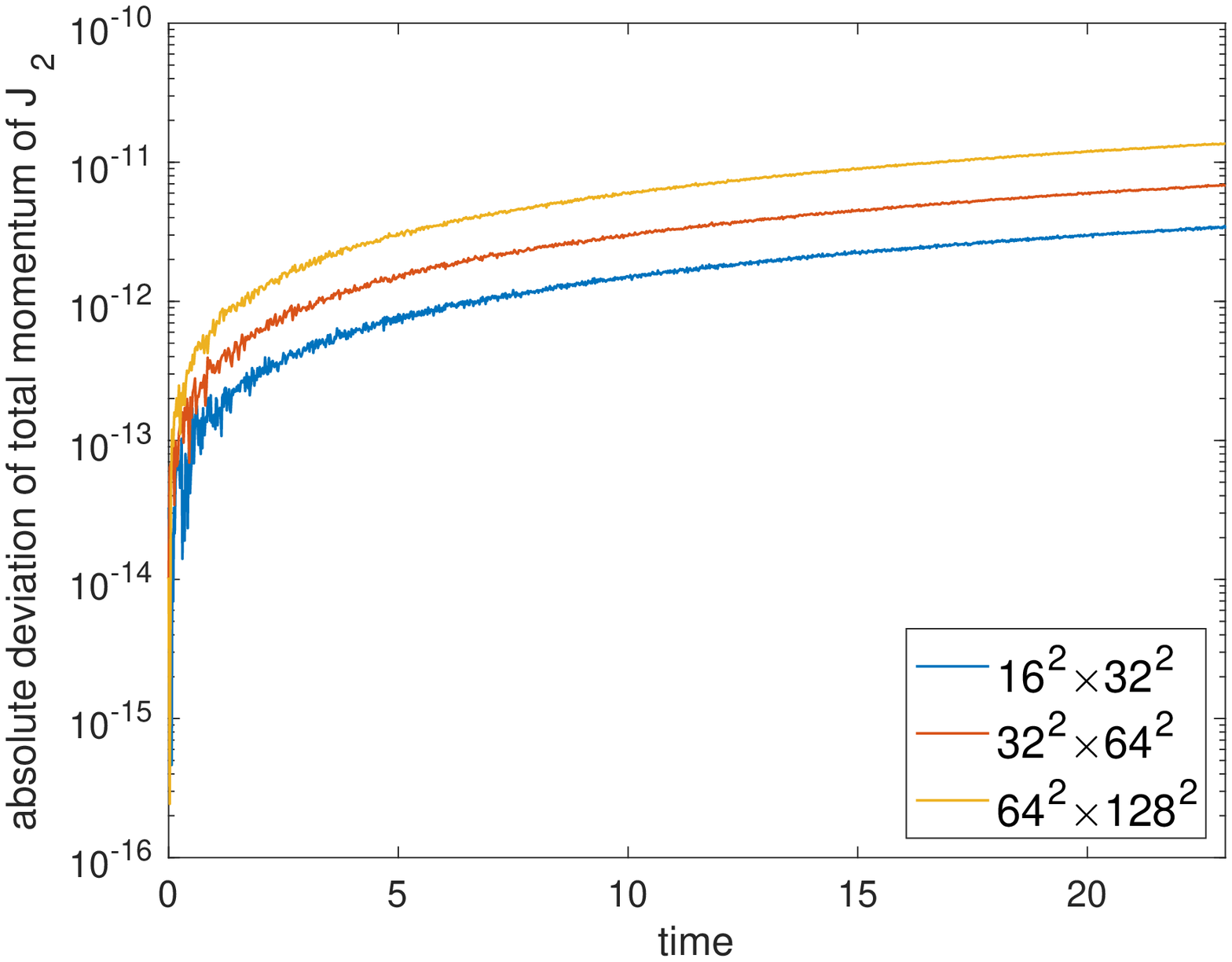}}
		\subfigure[]{\includegraphics[height=40mm]{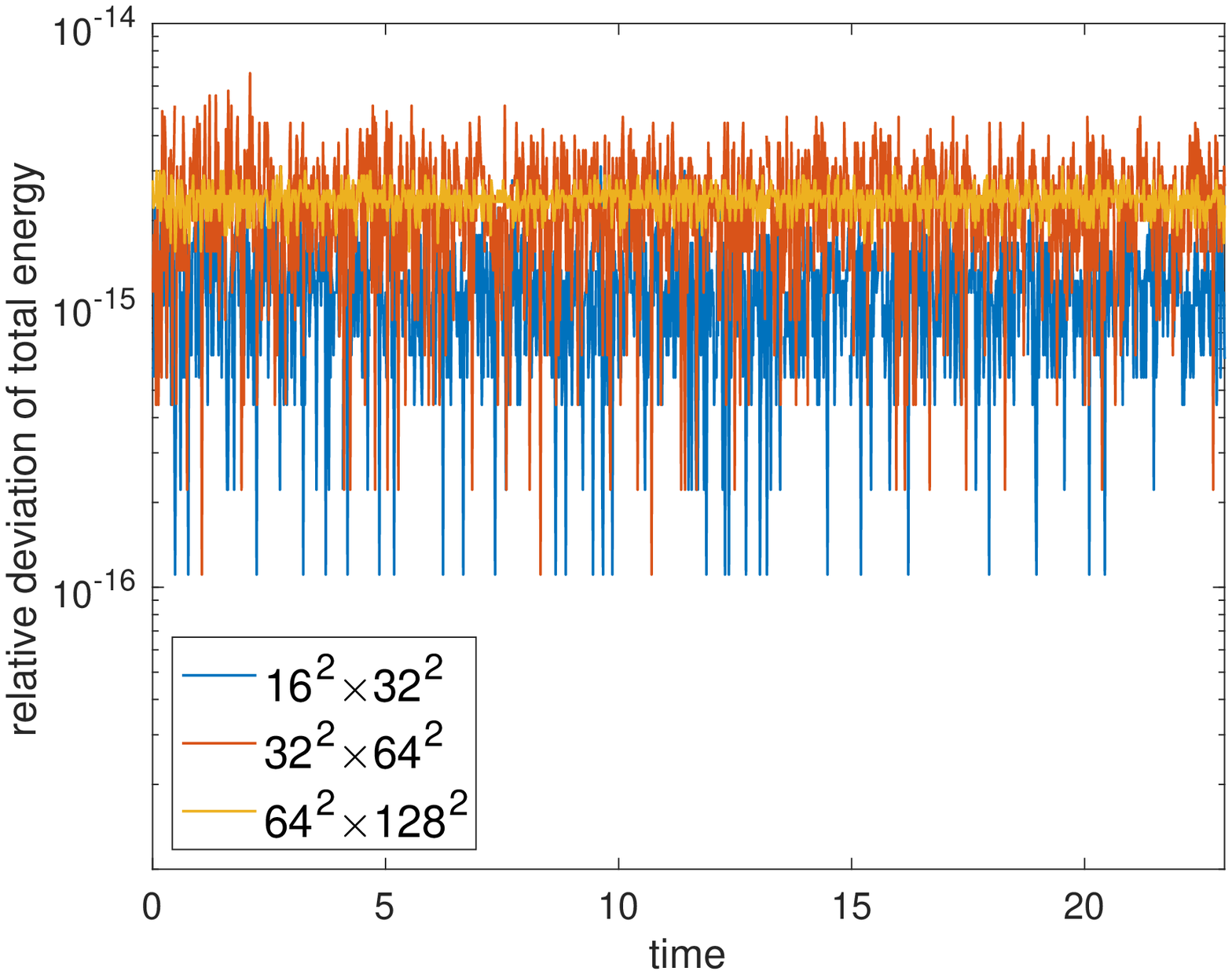}}
	\caption{Example \ref{ex:two2d}. The time evolution of  electric energy (a), hierarchical ranks of the numerical solution of mesh size $N^2_x\times N^2_v=64^2\times128^2$ (b), relative deviation of total mass (c), absolute deviation of total momentum $J_1$ (d) and total momentum $J_2$ (e), and relative deviation of total energy (f). $\varepsilon=10^{-5}$. $k=2$. }
	\label{fig:two2d_elec_con_k2}
\end{figure}

%
%

 \end{exa}

\section{Conclusion}
\setcounter{equation}{0}
\setcounter{figure}{0}
\setcounter{table}{0}

In this paper, we proposed a LoMaC low rank tensor approach with nodal DG discretization for performing high dimensional deterministic Vlasov simulations.
The introduction of DG and nodal DG discretization opens up the potential of low rank tensor algorithm in using general nonsmooth, nonuniform or unstructured meshes and for handling complex boundary conditions. The locally macroscopic conservation property, realized by a macroscopic conservative projection and correction of the kinetic solution, preserves globally total mass, momentum and energy at the discrete level using an explicit scheme. 
The algorithm is extended to the 2D2V VP system by a hierarchical Tucker structure with full rank (no reduction) in the physical space and low rank reduction for the phase space as well as for the linkage between phase and physical spaces. Further work includes the extension to unstructured mesh and in resolving complex boundary conditions arise from applications.

\bigskip
\noindent
{\bf Conflict of interest statement:} 
on behalf of all authors, the corresponding author states that there is no conflict of interest.

\bibliographystyle{abbrv}
\bibliography{refer,ref_guo,ref_cheng,ref_cheng_2}

\end{document}